\documentclass[a4paper,oneside,10pt]{article}%
\usepackage{amsfonts}
\usepackage{amsmath,amssymb}
\usepackage{amsmath}
\usepackage{amssymb}
\usepackage{graphicx}
\usepackage{hyperref}
\usepackage{color}%
\providecommand{\U}[1]{\protect\rule{.1in}{.1in}}
\newtheorem{theorem}{Theorem}[section]

\newtheorem{definition}[theorem]{Definition}

\newtheorem{lemma}[theorem]{Lemma}

\newtheorem{proposition}[theorem]{Proposition}
\newtheorem{remark}[theorem]{Remark}
\newenvironment{proof}[1][Proof]{\noindent \textbf{#1.} }{\  $\Box$}
\numberwithin{equation}{section}
\begin{document}

\title{Maximum principle for optimal control of stochastic evolution equations with
recursive utilities}
\author{Guomin Liu \thanks{School of Mathematical Sciences, Nankai University,
Tianjin, China. gmliu@nankai.edu.cn. Research supported by National Natural Science Foundation of
China (No. 12201315 and No. 12071256), China Postdoctoral
Science Foundation (No. 2020M670960)  and Natural Science Foundation of Shandong Province for
Excellent Youth Scholars (No. ZR2021YQ01).}
\and and \quad\quad Shanjian Tang\thanks{School of Mathematical Sciences, Fudan
University, Shanghai, China. sjtang@fudan.edu.cn. Research supported by
National Key R\&D Program of China (No. 2018YFA0703900) and National Natural
Science Foundation of China (No. 12031009). }}
\date{}
\maketitle

\begin{abstract}
We consider the optimal control problem of stochastic evolution equations in a
Hilbert space under a recursive utility, which is described as the solution of
a backward stochastic differential equation (BSDE). A very general maximum
principle is given for the optimal control, allowing the control domain not to
be convex and the generator of the BSDE to vary with the second unknown
variable $z$. The associated second-order adjoint process is characterized as
a unique solution of a conditionally expected operator-valued backward
stochastic integral equation.

 \smallskip


\end{abstract}

\vspace{-10pt}

\section{Introduction}

In this paper, we consider the optimal control problem of stochastic evolution
equations (SEEs)
\begin{equation}%
\begin{cases}
{d}x(t) & =[A(t)x(t)+a(t,x(t),u(t))]{d}t+[B(t)x(t)+b(t,x(t),u(t))]dw({t}%
),\quad t\in[0,T],\\
x(0) & =x_{0}\in H:\text{\textrm{a Hilbert space}}%
\end{cases}
\label{eq0-1}%
\end{equation}
with a recursive utility which solves the backward stochastic differential
equation (BSDE)
\begin{equation}
y(t)=h(x(T))+\int_{t}^{T}k(s,x(s),y(s),z(s),u(s))ds-\int_{t}^{T}z(s){d}w(s).
\label{eq0-22}%
\end{equation}
Here, $w(\cdot)$ is a Brownian motion, $(A(t),B(t))$ are random
linear unbounded operators for $t\in\lbrack0,T]$, $(a,b,h,k)$ are nonlinear
functions and $u(\cdot)$ is a control process, and taking values in a given
metric space. The objective is to minimize the initial value $y(0)$ as a
functional of the control:
\begin{equation}
J(u(\cdot)):=y(0). \label{Myeq0-1}%
\end{equation}
The notion of a recursive utility in continuous time was introduced by Duffie
and Epstein \cite{DE92} and generalized to the form of (\ref{eq0-22}) in Peng
\cite{Pe93} and El Karoui, Peng and Quenez \cite{KPQ97}. When $k$ is invariant
with $(y,z)$, by taking expectation on both sides of (\ref{eq0-22}), we get
\[
J(u(\cdot))=\mathbb{E}[h(x(T))+\int_{0}^{T}k(t,x(t),u(t))dt],
\]
and the stochastic optimal control problem is reduced to the conventional one,
which has been addressed in \cite{DM13,FHT13,LZ14}.

Pontryagin's maximum principle for optimally controlled ordinary differential
equations is a milestone in the modern optimal control theory. By now, the
maximum principle for optimally controlled finite-dimensional systems is quite
complete. The maximum principle for a general stochastic optimal control
problem was finally given by Peng \cite{Pe90}, by introducing a second-order
adjoint process which solves a matrix-valued BSDE. In the extension to
incorporate the recursive utility, an essential difficulty is how to derive
the second-order variational equation of the recursive BSDE (\ref{eq0-2}). It
was listed as an open problem by Peng \cite{Pe98}. Until recently, Hu
\cite{Hu-17} completely solved this problem by developing a clever Taylor's
expansion, so to reduce the order of the variation of the recursive BSDEs.

To formulate the counterpart of the infinite-dimensional stochastic optimal
control system, a crucial issue is the characterization of the second-order
adjoint process $P$, which takes values in the space $\mathfrak{L}(H)$ of all
bounded linear operators from $H$ to $H$. Since the operator space
$\mathfrak{L}(H)$ is not a (separable) Hilbert space, the dynamics of the
second adjoint process could not be described by a conventional BSDE as in the
finite-dimensional case. In the existing maximum principles for the
conventional stochastic optimal control problem, the second-order process $P$
is given in various ways. L\"{u} and Zhang \cite{LZ14,LZ18} utilize the notion of
transposition solutions in the context of real-valued equations, assuming the coefficients, such as the terminal condition and the generator of the equation, to be strongly measurable (hence separably valued; see  \cite[Theorem 2.1]{KR81}) and the  space $L^{2}(\mathcal{F}_T)$ to be separable. Derived from the limit of the quadratic terms
in the variational calculation of the maximum principle, Du and Meng
\cite{DM13} and Fuhrman, Hu and Tessitore \cite{FHT13} define $P$ through a
stochastic bilinear form. In both approaches, no dynamics of
the second adjoint process are given. On the other hand, Guatteri and Tessitore
\cite{GT05,GT14} characterize $P$ using the mild solution of an operator-valued BSDE. They impose either the Hilbert-Schmidt assumption on the coefficients (which can be relaxed only for a suitable limit of solutions with such data, referred to as a generalized solution) 
or a rather restrictive
regularity condition on the unbounded operators. Similarly, Stannat and Wessels
\cite{SW21} employ a function-valued backward SPDE when the coefficients (of the
system and the cost functional) depend on the state variables in a Nemytskii
manner. However, their diffusion coefficient contains no unbounded operator
(this also happens in \cite{FHT13,GT05,GT14,LZ14,LZ18}) and is further
required to have a very high regularity when the space dimension is greater
than one (see   \cite[Remark 4.3]{SW21}).

The aim of this paper is to study the maximum principle for the optimal
control problem (\ref{Myeq0-1}) of infinite-dimensional stochastic system with
recursive utilities. To characterize the dynamics of the second-order adjoint
process $P,$ we propose a notion of conditionally expected operator-valued
backward stochastic integral equations (BSIEs in short) to serve as the
second-order adjoint equations. The formulation of our BSIEs is very naturally
inspired by the variation of constants method for operator-valued SPDEs (see
Remark \ref{Rm2-2} (i)). Under mild conditions, the existence and uniqueness
of solutions to the operator-valued BSIEs is obtained in virtue of a concept
of aggregated-defined operator-valued conditional expectation and a
contraction mechanism, without imposing additional separability assumption on
the coefficients.

On the other hand, the It\^{o}'s formulas (or the duality formulas) for $\langle
P({t})x({t}),x({t})\rangle$ in the above mentioned works of characterizing $P$
require that the homogeneous terms in both equations of $P$ and $x$ are dual
(in a proper sense) so that they can cancel out in the final duality formula,
which are not satisfied for our recursive utility context. In this paper, to
obtain the maximum condition, we shall derive a more general It\^{o}'s formula
in which some homogeneous terms in both the equations of $P$ and $x$ remain to
appear (see Theorem \ref{Myth3-7} and Remark \ref{Re2-1} (iii)), by using the
explicit formula of linear BSDEs and an approximation argument. Furthermore, unlike the finite-dimensional or non-recursive case, the variational equations of utility BSDE (\ref{eq0-2}) involve additional terms $\langle p(\cdot),{B}(\cdot)x^{1,\rho}(\cdot)\rangle$ and $\langle p(\cdot),{B}(\cdot)x^{2,\rho}(\cdot)\rangle$, which incorporate  the unbounded operator $B$ and thus cannot be handled using the usual estimates for $p$ and $x^{1,\rho},x^{2,\rho}$ in $H$. Here, $p$ is the first-order adjoint process, $x^{1,\rho}$ and $x^{2,\rho}$ are the solutions of the first- and second-order variational equations for the state equation (\ref{eq0-1}), respectively. To overcome this difficulty, we deduce and utilize an $L^\beta$-estimate of $p$ in the space $V$  (see  (\ref{Myeq4-22}), the proof of Proposition \ref{Myth4-9} and Remark \ref{Rm-B}).

The rest of this paper is organized as follows. In Section 2, we introduce a
conditionally expected operator-valued BSIE and further give its It\^{o}'s
formula. We formulate our infinite-dimensional optimal control problem under a
recursive utility and derive the maximum principle in Section 3. The appendix
includes the proofs of some important technical results used in the paper.

\section{Conditionally expected operator-valued BSIEs}

In this section, we give an existence and uniqueness result for a
conditionally expected operator-valued backward stochastic integral equation
(BSIE). It will be used to characterize the dynamics of the second-order
adjoint process in the maximum principle for optimally controlled stochastic
evolution equations (SEEs).

Let $(\Omega,\mathcal{F},\mathbb{P})$ be a probability space. Fix a terminal
time $T>0,$ let $\mathbb{F}:=\{\mathcal{F}_{t}\}_{0\leq t\leq T}$ be a
filtration on $(\Omega,\mathcal{F},\mathbb{P})$ satisfying the usual
conditions. We denote by $\Vert\cdot\Vert_{X}$ the norm on a Banach space $X$.
By $\mathfrak{L}(X;Y)$, we denote the space of all bounded linear operators
from $X$ to another Banach space $Y$, equipped with the operator norm. We
write $\mathfrak{L}(X)$ for $\mathfrak{L}(X;X).$

Let $H$ be a separable Hilbert space with inner product $\langle\cdot
,\cdot\rangle$. We adopt the standard identification viewpoint of
$\mathfrak{L}(H;\mathbb{R})=H.$ By $M^{\ast}$, we denote the adjoint of an
operator $M.$ We denote by $I_{d}$ the identity operator on $H.$

Given a sub-$\sigma$-algebra $\mathcal{G}$ of $\mathcal{F}$. For $\alpha
\geq1,$ we denote by $L^{\alpha}(\mathcal{G},H)$ the space of $H$-valued
$\mathcal{G}$-measurable mapping $y$ with norm $\Vert y\Vert_{L^{\alpha
}(\mathcal{G},H)}=\{\mathbb{\mathbb{E}}[\Vert y\Vert_{H}^{\alpha}]\}^{\frac
{1}{\alpha}}$, and by $L_{\mathbb{F}}^{\alpha}(0,T;H)$ (resp. $L_{\mathbb{F}%
}^{2,\alpha}(0,T;H)$) the space of $H$-valued progressively measurable
processes $y(\cdot)$ with norm $\Vert y\Vert_{L_{\mathbb{F}}^{\alpha}%
(0,T;H)}=\{\mathbb{\mathbb{E}}[\int_{0}^{T}\Vert y(t)\Vert_{H}^{\alpha}%
{d}t]\}^{\frac{1}{\alpha}}$ (resp. $\Vert y\Vert_{L_{\mathbb{F}}^{2,\alpha
}(0,T;H)}=\{\mathbb{\mathbb{E}}[(\int_{0}^{T}\Vert y(t)\Vert_{H}^{2}%
{d}t)^{\frac{\alpha}{2}}]\}^{\frac{1}{\alpha}}$). We write $L^{\alpha
}(\mathcal{G})$, $L_{\mathbb{F}}^{\alpha}(0,T)$ and $L_{\mathbb{F}}^{2,\alpha
}(0,T)$ for $L^{\alpha}(\mathcal{G},\mathbb{R})$, $L_{\mathbb{F}}^{\alpha
}(0,T;\mathbb{R})$ and $L_{\mathbb{F}}^{2,\alpha}(0,T;\mathbb{R})$, respectively.

We say a mapping $Z:\Omega\rightarrow\mathfrak{L}(H)$ is weakly $\mathcal{G}%
$-measurable if for each $(u,v)\in H\times H,$ $\langle Zu,v\rangle
:\Omega\rightarrow\mathbb{R}$ is $\mathcal{G}$-measurable. A process
$Y:\Omega\times\lbrack0,T]\rightarrow\mathfrak{L}(H)$ is said to be weakly
progressively measurable (weakly adapted, resp.) if for each $(u,v)\in H\times
H,$ the process $\langle Yu,v\rangle:\Omega\times\lbrack0,T]\rightarrow$
$\mathbb{R}$ is progressively measurable (adapted, resp.).

By $L_{w}^{\alpha}(\mathcal{G},\mathfrak{L}(H)),$ we denote the space of
$\mathfrak{L}(H)$-valued weakly $\mathcal{G}$-measurable mapping $F$ with norm
$\Vert F\Vert_{L_{w}^{\alpha}(\mathcal{G},\mathfrak{L}(H))}%
=\{\mathbb{\mathbb{E}}[\Vert F\Vert_{\mathfrak{L}(H)}^{\alpha}]\}^{\frac
{1}{\alpha}}$. Since there is a countable dense subset $V$ of $H$ such that%
\[
\Vert F(\omega)\Vert_{\mathfrak{L}(H)}=\sup_{\substack{(u,v)\in V\times
V,\\\Vert u\Vert_{H},\Vert v\Vert_{H}\leq1}}|\langle F(\omega)u,v\rangle|
,\quad\omega\in\Omega,
\]
the real-valued function $\omega\mapsto\Vert F(\omega)\Vert_{\mathfrak{L}(H)}%
$\ is $\mathcal{G}$-measurable and the norm $\Vert F\Vert_{L_{w}^{\alpha
}(\mathcal{G},\mathfrak{L}(H))}$ is well-defined. Similarly, we denote by
$L_{\mathbb{F},w}^{\alpha}(0,T;\mathfrak{L}(H))$ (resp. $L_{\mathbb{F}%
,w}^{2,\alpha}(0,T;\mathfrak{L}(H))$) the space of $\mathfrak{L}(H)$-valued
weakly progressively measurable processes $F(\cdot)$ with norm $\Vert
F\Vert_{L_{\mathbb{F},w}^{\alpha}(0,T;\mathfrak{L}(H))}=\{\mathbb{\mathbb{E}%
}[\int_{0}^{T}\Vert F(t)\Vert_{\mathfrak{L}(H)}^{\alpha}{d}t]\}^{\frac
{1}{\alpha}}$ (resp. $\Vert F\Vert_{L_{\mathbb{F},w}^{2,\alpha}%
(0,T;\mathfrak{L}(H))}=\{\mathbb{\mathbb{E}}[(\int_{0}^{T}\Vert F(t)\Vert
_{\mathfrak{L}(H)}^{2}{d}t)^{\frac{\alpha}{2}}]\}^{\frac{1}{\alpha}}$). From
standard arguments, we can see that $L_{w}^{\alpha}(\mathcal{G},\mathfrak{L}%
(H))$, $L_{\mathbb{F},w}^{\alpha}(0,T;\mathfrak{L}(H))$ and $L_{\mathbb{F}%
,w}^{2,\alpha}(0,T;\mathfrak{L}(H))$ are all Banach spaces. In the following,
we shall not distinguish two random variables if they coincide $P$-a.s$.$ and
two processes if one is a modification of the other, unless other stated.

\begin{remark}
\upshape{In general, there are mainly three kinds of measurability notions for Banach space-valued random
variables: strongly measurable (can be approximated by a sequence of simple measurable
functions), measurable (the preimage of each Borel set is measurable) and
weakly measurable (the composition with any element in the dual space or in a
proper subspace (called a norming
subspace; see \cite[p. 2]{JN07}) of the dual space
is a real-valued measurable function). These three notions are equivalent in a separable Banach space (see  \cite[Theorem 1.5 and Prop.
1.8]{JN07}) and it is not necessary to indicate the notion of measurability in the above for $H$-valued random mappings.  Moreover, the notion of \textquotedblleft measurable"
does not work well in the non-separable case since even the sum of two
measurable functions may not be measurable (see \cite{JN57}). The operator space $\mathfrak{L}(H)$ is not separable in general (even when $H$ is; see
\cite[Solution 99]{Ha82}), so these notions are quite different for it. We adopt the above
weak measurability notion for $\mathfrak{L}(H)$-valued mappings in which the
test functions are from $H\times H.$ Note that $H\times H$ can be regarded as
a  subset of the dual of $\mathfrak{L}(H)$ by taking $f_{u,v}(z)=\langle z(u),v\rangle$, for
$z\in\mathfrak{L}(H)$ and $(u,v)\in H\times H$ and $\text{span}(H\times H)$ is a norming subspace of $\mathfrak{L}(H)$. Thus this weak
measurability notion is still one kind of  standard forms.}
\end{remark}

Denote by $L_{w}$ the weak $\sigma$-algebra generated by all the sets in the
form of
\[
\{z\in\mathfrak{L}(H):\langle zu,v\rangle\in A \},\quad u,v\in H,\ A\in
\mathcal{B}(\mathbb{R}).
\]
Then it is straightforward to verify that $Z:\Omega\rightarrow\mathfrak{L}(H)$
is weakly $\mathcal{G}$-measurable if and only if it is measurable from
$(\Omega,\mathcal{G})$ to $(\mathfrak{L}(H),L_{w})$ (see also \cite{Ed77}).
Similarly, $Y:\Omega\times\lbrack0,T]\rightarrow\mathfrak{L}(H)$ is weakly
progressively measurable if and only if it is measurable from $(\Omega
\times[0,T],\mathcal{P})$ to $(\mathfrak{L}(H),L_{w})$, where $\mathcal{P}$ is
the progressive $\sigma$-algebra on $\Omega\times[0,T].$

\subsection{Conditional expectation for operator-valued random variables}

The operator-valued BSIE is based on a notion of conditional expectations for
random variables taking values in the operator space $\mathfrak{L}(H)$. As is
well known, the classical theory on the conditional expectations for Banach or
Hilbert space-valued random variables requires the separability of the value
spaces (see, e.g., \cite{DZ92,Ro18}). But in general the operator space
$\mathfrak{L}(H)$ is not separable and thus the above-mentioned result does
not apply. In this subsection we shall construct a new kind of conditional
expectations for operator-valued random variables by exploring the
separability of $H$, rather than that of $\mathfrak{L}(H)$ (which is the case
when the classical Banach or Hilbert space-valued conditional expectation
theory applies to this situation).

Recall that for any Banach space $X,$ we have the identity (see \cite{Co12}
for more details)
\[
\mathfrak{L}_{2}(H\times H;X)=\mathfrak{L}(H;\mathfrak{L}(H;X))
\]
by identifying $\tilde{\varphi}\in\mathfrak{L}_{2}(H\times H;X)$ with
$\varphi\in\mathfrak{L}(H;\mathfrak{L}(H;X))$ through
\[
\tilde{\varphi}(u,v):=\varphi(u)v,\quad\forall(u,v)\in H\times H,
\]
where $\mathfrak{L}_{2}(H\times H;X)$ is the space of all bounded bilinear
operators from $H\times H$ to $X$, equipped with the operator norm. Thus,
\[
\mathfrak{L}(H;\mathfrak{L}(H;L^{1}(\mathcal{G})))=\mathfrak{L}_{2}(H\times
H;L^{1}(\mathcal{G})).
\]
From $\mathfrak{L}(H;\mathbb{R})=H,$ we also have the isometry
\[
\mathfrak{L}_{2}(H\times H):=\mathfrak{L}_{2}(H\times H;\mathbb{R}%
)=\mathfrak{L}(H;\mathfrak{L}(H;\mathbb{R}))=\mathfrak{L}(H;H)=\mathfrak{L}%
(H).
\]
This new space $\mathfrak{L}_{2}(H\times H)$ is easier to work with than the
previous $\mathfrak{L}(H)$ and is more essential for us to construct the
conditional expectations. So we shall state the construction (of the
conditional expectations) for $\mathfrak{L}_{2}(H\times H)$-valued random
variables, and the original $\mathfrak{L}(H)$ form can be obtained directly
via the above isometry after this procedure completes.

\subsubsection{Existence of the conditional expectation}

We adopt the same weak measurability meaning for $\mathfrak{L}_{2}(H\times
H)$-valued random variables according to the isometry $\mathfrak{L}%
_{2}(H\times H)=\mathfrak{L}(H).$ That is, a mapping $Z:\Omega\rightarrow
\mathfrak{L}_{2}(H\times H)$ is called $\mathcal{G}$-weakly measurable if for
each $(u,v)\in H\times H,$ $Z(u,v):\Omega\rightarrow\mathbb{R}$ is
$\mathcal{G}$-measurable. The definition of weakly progressive measurability
and weakly adaptedness for $\mathfrak{L}_{2}(H\times H)$-valued processes is
similar. In the same manner, we define $L_{w}^{\alpha}(\mathcal{G}%
,\mathfrak{L}_{2}(H\times H))$ as the space of $\mathfrak{L}_{2}(H\times
H)$-valued weakly $\mathcal{G}$-measurable mapping $F$ with norm $\Vert
F\Vert_{L_{w}^{\alpha}(\mathcal{G},\mathfrak{L}_{2}(H\times H))}%
=\{\mathbb{\mathbb{E}}[\Vert F\Vert_{\mathfrak{L}_{2}(H\times H)}^{\alpha
}]\}^{\frac{1}{\alpha}}$ and have $L_{w}^{\alpha}(\mathcal{G},\mathfrak{L}%
_{2}(H\times H))=L_{w}^{\alpha}(\mathcal{G},\mathfrak{L}(H))$.

It is very natural to define the conditional expectation for $\mathfrak{L}%
_{2}(H\times H)$-valued random variables, i.e., for an $\mathfrak{L}%
_{2}(H\times H)$-valued $Y$, to find an $\mathfrak{L}_{2}(H\times H)$-valued
$\mathbb{E}[Y|\mathcal{G}]$ as its conditional expectation. But in the
infinite-dimensional case, the quantity to be conditionally expected in the
formulation of the BSIEs later lies in a larger space $\mathfrak{L}%
_{2}(H\times H;L^{1}(\mathcal{F}))$ (see (\ref{Myeq2-12})). So in the
following we shall define conditional expectations for this larger class (the
conditional expectation is still $\mathfrak{L}_{2}(H\times H)$-valued), and
the $\mathfrak{L}_{2}(H\times H)$-valued situation can be regarded as a
special case.

We first verifies that
\begin{equation}
\label{Myeq2-7}L_{w}^{1}(\mathcal{G},\mathfrak{L}_{2}(H\times H))\subset
\mathfrak{L}_{2}(H\times H;L^{1}(\mathcal{G})).
\end{equation}
Indeed, for any $Y\in L_{w}^{1}(\mathcal{G},\mathfrak{L}_{2}(H\times H))$ and
$(u,v)\in H\times H,$ from the definition of weak measurability, we see that
$Y(u,v)\in\mathcal{G}$. Moreover,%
\[
\mathbb{\mathbb{E}}[|Y(u,v)|]\leq\mathbb{\mathbb{E}}[\Vert Y\Vert
_{\mathfrak{L}_{2}(H\times H)}\Vert u\Vert_{H}\Vert v\Vert_{H}%
]=\mathbb{\mathbb{E}}[\Vert Y\Vert_{\mathfrak{L}_{2}(H\times H)}]\Vert
u\Vert_{H}\Vert v\Vert_{H}<\infty,
\]
and thus the mapping $(u,v)\longmapsto Y(u,v)$ is bounded bilinear from
$H\times H$ to $L^{1}(\mathcal{G}).$ So $Y\in\mathfrak{L}(H\times
H;L^{1}(\mathcal{G})).$

The main difference between the elements in $L_{w}^{1}(\mathcal{G}%
,\mathfrak{L}_{2}(H\times H))$ and $\mathfrak{L}_{2}(H\times H;L^{1}%
(\mathcal{G}))$ is that the definition and bilinearity for the one in the
first space is pointwise or say, independent the effect arguments $(u,v)\in
H\times H$, but the definition and bilinearity for the one in the second space
is only in a rough way and may depend on its effect arguments $(u,v)\in
H\times H.$ To be more detailed, given any $Y\in L_{w}^{1}(\mathcal{G}%
,\mathfrak{L}_{2}(H\times H)),$ for each (or at least $P$-a.s.) $\omega$, we
have $Y(\omega)\in\mathfrak{L}_{2}(H\times H),$ which is also
\[
Y(\alpha_{1}u_{1}+u_{2},v_{1})(\omega)=\alpha_{1}Y(u_{1},v_{1})(\omega
)+Y(u_{2},v_{1})(\omega),\ Y(u_{1},\alpha_{2}v_{1}+v_{2})(\omega)=\alpha
_{2}Y(u_{1},v_{1})(\omega)+Y(u_{1},v_{2})(\omega),
\]
for every $(u_{1},v_{1}),(u_{2},v_{2})\in H\times H$ and $\alpha_{1}%
,\alpha_{2}\in\mathbb{R}$ (The negligible set is universal for all
$(u_{1},v_{1}),(u_{2},v_{2})\in H\times H$ and $\alpha_{1},\alpha_{2}%
\in\mathbb{R)}$. Whereas for $Y\in\mathfrak{L}_{2}(H\times H;L^{1}%
(\mathcal{G})),$ since we do not distinguish the $P$-a.s. equal elements
in\ $L^{1}(\mathcal{G})$, we can only have that, for any $(u_{1},v_{1}%
),(u_{2},v_{2})\in H\times H$ and $\alpha_{1},\alpha_{2}\in\mathbb{R}$,
\[
Y(\alpha_{1}u_{1}+u_{2},v_{1})=\alpha_{1}Y(u_{1},v_{1})+Y(u_{2},v_{1}%
)\ \ \text{and}\ \ Y(u_{1},\alpha_{2}v_{1}+v_{2})=\alpha_{2}Y(u_{1}%
,v_{1})+Y(u_{1},v_{2}),\quad P\text{-a.s.}%
\]
(The negligible set depends on $(u_{1},v_{1}),(u_{2},v_{2})\in H\times H$ and
$\alpha_{1},\alpha_{2}\in\mathbb{R)}$.

For $Y\in\mathfrak{L}_{2}(H\times H;L^{1}(\mathcal{F})),$ we call an
$\mathfrak{L}_{2}(H\times H)$-valued weakly $\mathcal{G}$-measurable mapping
$Z$ the conditional expectation of $Y$ with respect to $\mathcal{G}$, denoted
by $\mathbb{E}[Y|\mathcal{G}]$, if for each $(u,v)\in H\times H$,
\begin{equation}
\label{Myeq2-15}Z(u,v)=\mathbb{E}[Y(u,v)|\mathcal{G}],\quad P\text{-a.s.}%
\end{equation}
meaning that $Z$ coincides with the classical conditional expectation at all
the test points $(u,v)$.

In general, for $Y\in\mathfrak{L}_{2}(H\times H;L^{1}(\mathcal{F})),$ we
always have that the mapping defined by $H\times H\ni(u,v)\longmapsto
\mathbb{E}[Y(u,v)|\mathcal{G}]$ (we can still denote it $\mathbb{E}%
[Y|\mathcal{G}]$ by a slight abuse of the notations) belongs to $\mathfrak{L}%
_{2}(H\times H;L^{1}(\mathcal{G})).$ Indeed,%
\[
\mathbb{\mathbb{E[}}|\mathbb{E}[Y(u,v)|\mathcal{G}]|]\leq\mathbb{\mathbb{E}%
}[|Y(u,v)|]\leq C\Vert u\Vert_{H}\Vert v\Vert_{H},
\]
where the last inequality is due to $Y\in\mathfrak{L}_{2}(H\times
H;L^{1}(\mathcal{F})).$ But whether some of its versions can be operator
$\mathfrak{L}_{2}(H\times H)$-valued so that it is the conditional expectation
we are searching for, is not known. To find such a version can be regarded as
an aggregation problem of constructing a better version among all the
equivalent admissible rough classes, which will be discussed in the next subsection.

We generally have the following existence and uniqueness theorem on the
conditional expectation of an operator-valued random variable.

\begin{theorem}
\label{Myth2-17} Let $Y\in\mathfrak{L}_{2}(H\times H;L^{1}(\mathcal{F}))$.
Then the conditional expectation $\mathbb{E}[Y|\mathcal{G}]$ exists and is
integrable (i.e., $\mathbb{E}[Y|\mathcal{G}]\in L_{w}^{1}(\mathcal{G}%
,\mathfrak{L}_{2}(H\times H))$) if and only if the mapping $(u,v)\longmapsto
\mathbb{E}[Y(u,v)|\mathcal{G}]\in\mathfrak{L}_{2}(H\times H;L^{1}%
(\mathcal{G}))$ satisfies the domination condition
\begin{equation}
|\mathbb{E}[Y(u,v)|\mathcal{G}]|\leq g\Vert u\Vert_{H}\Vert v\Vert_{H},\quad
P\text{-a.s.},\ \forall(u,v)\in H\times H, \label{Myeq2-29}%
\end{equation}
for some $0\leq g\in L^{1}(\mathcal{G})$. Moreover, such an $\mathbb{E}%
[Y|\mathcal{G}]$ is unique (up to $P$-a.s. equality) and satisfies
\begin{equation}
\Vert\mathbb{E}[Y|\mathcal{G}]\Vert_{\mathfrak{L}_{2}(H\times H)}\leq g,\quad
P\text{-a.s.} \label{Myeq2-11}%
\end{equation}

\end{theorem}

Before going to the proof, we present the following remarks.

\begin{remark}
\upshape{In the above definition of conditional expectations, we make use of a similar idea of test as the one for $H$-valued random variables (see, e.g., \cite[Definition 2.4]{KR81} and  \cite[Definition 2.1]{Ro18}), but  apply it to a more general bilinear
situation. By similar
arguments (see the proofs of Theorems \ref{Myth1-1} and \ref{Myth2-17}), this $\mathfrak{L}_{2}(H\times H)$-valued conditional expectation
holds for the more general $k$-linear operator (i.e., $\mathfrak{L}_{k}(H_{1}\times H_{2}\times\cdots\times H_{k})$-valued) case with different
separable Hilbert spaces $H_{j},j\leq k$, for $k=1,2,3,\cdots,$  and when $k=1$, it constructs the conditional expectation for $H$-valued
random variables in a slightly new way. Indeed, at this case, from
$H=\mathfrak{L}(H;\mathbb{R}),$ the relationship (\ref{Myeq2-7}) becomes
$L^{1}(\mathcal{G},H)=L^{1}(\mathcal{G},\mathfrak{L}(H;\mathbb{R}))\subset\mathfrak{L}(H;L^{1}(\mathcal{G}))$ (we delete the subscript $w$ (for the first and second spaces)
since the measurability and weak measurability are the same now due to the
separability of $H$); the conditional expectation for $Y\in\mathfrak{L}(H;L^{1}(\mathcal{F}))$ is a $H$-valued  $\mathcal{G}$-measurable mapping $Z$
satisfying $\langle Z,u\rangle=Z(u)=\mathbb{E}[Y(u)|\mathcal{G}]$ $P$-a.s.,
for all $u\in H;$ the above theorem reads: for $Y\in\mathfrak{L}(H;L^{1}(\mathcal{F})),$ the conditional expectation $\mathbb{E}[Y|\mathcal{G}]\in L^{1}(\mathcal{G},H)$ exists iff\[
|\mathbb{E}[Y(u)|\mathcal{G}]|\leq g\Vert u\Vert_{H},\quad P\text{-a.s.},\ \forall u\in H,
\]
for some $0\leq g\in L^{1}(\mathcal{G}),$ $\mathbb{E}[Y|\mathcal{G}]$ is
unique and satisfies $\Vert\mathbb{E}[Y|\mathcal{G}]\Vert_{H}\leq g,\
P$-a.s. This generalizes the classical result for the conditional expectation
of $H$-valued random variables since $Y$ does not need to be true
$H$-valued.
}
\end{remark}

\begin{remark}
\upshape{
From the proofs latter, the condition $Y\in\mathfrak{L}_{2}(H\times H;L^{1}(\mathcal{F}))$ in the definition of the conditional expectation and  in Theorem \ref{Myth2-17}
can be weaken to $(u,v)\longmapsto\mathbb{E}[Y(u,v)|\mathcal{G}]\in
\mathfrak{L}_{2}(H\times H;L^{1}(\mathcal{G}))$. Note that, if the conditional expectation $\mathbb{E}[Y|\mathcal{G}]\in L_{w}^{1}(\mathcal{G},\mathfrak{L}_{2}(H\times H))$ exists, this
new condition also holds (see (\ref{Myeq2-7})), so it (plus the domination condition) is the
weakest condition to guarantee the existence of integrable $\mathfrak{L}_{2}(H\times H)$-valued
conditional expectations. This generalization also holds for the
$k$-linear operator case, and in particular, when $k=1,$ it provides a necessary and sufficient characterization for the
existence of integrable $H$-valued conditional expectations.}
\end{remark}

\subsubsection{An aggregation theorem and proof of Theorem \ref{Myth2-17}}

For a mapping $G\in\mathfrak{L}_{2}(H\times H;L^{1}(\mathcal{G})),$ by a
version of $G,$ we mean another $G^{\prime}:H\times H\longmapsto
L^{1}(\mathcal{G})$ satisfying $G(u,v)=G^{\prime}(u,v)$ in $L^{1}%
(\mathcal{G})$ (which is also, $P$-a.s.), for each $(u,v)\in H\times H$. It is
easy to check that $G^{\prime}\in\mathfrak{L}_{2}(H\times H;L^{1}%
(\mathcal{G})).$

The construction of the conditional expectation is based on the following
aggregation theorem for operator-valued random variables in the space of
bilinear mappings.

\begin{theorem}
\label{Myth1-1} The mapping $G\in\mathfrak{L}_{2}(H\times H;L^{1}%
(\mathcal{G}))$ admits a version $\bar{G}\in L_{w}^{1}(\mathcal{G}%
,\mathfrak{L}_{2}(H\times H))$ if and only if the following the domination
condition holds: there exists some $0\leq g\in L^{1}(\mathcal{G})$ such that
\begin{equation}
|G(u,v)|\leq g\Vert u\Vert_{H}\Vert v\Vert_{H},\quad P\text{-a.s.}%
,\ \forall(u,v)\in H\times H. \label{Myeq1-8}%
\end{equation}
Moreover, such an $\mathfrak{L}_{2}(H\times H)$-valued version is unique (up
to $P$-a.s. equality) and satisfies
\begin{equation}
\Vert\bar{G}\Vert_{\mathfrak{L}_{2}(H\times H)}\leq g,\quad P\text{-a.s.}
\label{Myeq2-10}%
\end{equation}

\end{theorem}

\begin{remark}
\upshape{The proof is based on an  idea of extension from a countable
dense subset of indexes,
which is motivated from \cite{El79}, see also \cite{DM13,FHT13,Tang15}.}
\end{remark}

\begin{proof}
Let $\{e_{i}\}_{i=1}^{\infty}$ be a countable basis of $H$.

\textit{Step 1: an auxiliary deterministic result.} For any given real values
$\{a_{ij}\}_{i,j=1}^{\infty},$ define
\[
F(e_{i},e_{j}):=a_{ij}, \quad\text{for}\ i,j\geq1.
\]
Then $F$ can be extended uniquely to be an element in $\mathfrak{L}%
_{2}(H\times H),$ which we still denote by $F$, if and only if there exists
some constant $C>0$ such that
\[
|\sum_{i=1}^{n}\sum_{j=1}^{m}\alpha_{i}\beta_{j}a_{ij}|\leq C\Vert\sum
_{i=1}^{n}\alpha_{i}e_{i}\Vert_{H}\Vert\sum_{j=1}^{m}\beta_{j}e_{j}\Vert
_{H},\quad\text{for all} \ \alpha_{i},\beta_{j}\in\mathbb{Q}\text{, and
integers }n,m\geq1.
\]
Moreover, this extension satisfies $\Vert F\Vert_{\mathfrak{L}_{2}(H\times
H)}\leq C. $

Indeed, we take a dense linear subspace with field $\mathbb{Q}$ of $H$
\[
V:=\{\sum_{i=1}^{n}\alpha_{i}e_{i}:\alpha_{i}\in\mathbb{Q},\ n\geq1\}.
\]
We define on $V\times V$
\[
F(\sum_{i=1}^{n}\alpha_{i}e_{i},\sum_{j=1}^{m}\beta_{j}e_{j}):=\sum_{i=1}%
^{n}\sum_{j=1}^{m}\alpha_{i}\beta_{j}F(e_{i},e_{j}).
\]
It is easy to check that $F$ is a well-defined bilinear mapping with field
$\mathbb{Q}$ on $V\times V$ and $|F(u,v)|\leq C\Vert u\Vert_{H}\Vert
v\Vert_{H}$, for all $(u,v)\in V\times V.$ Then by the continuous extension
theorem (see, e.g., \cite[Lemma 2.4]{Ku06}), $F$ can be extended to be an
element in $\mathfrak{L}_{2}(H\times H)$ satisfying $|F(u,v)|\leq C\Vert
u\Vert_{H}\Vert v\Vert_{H}$, for all $(u,v)\in H\times H,$ which is also
$\Vert F\Vert_{\mathfrak{L}_{2}(H\times H)}\leq C.$

Now we show that such an extension from basis $\{e_{i}\}_{i=1}^{\infty}$ is
unique. Let $F^{1}$, $F^{2}$ be two such extensions. Then $F^{1}(e_{i}%
,e_{j})=F^{2}(e_{i},e_{j})$ for each $i,j,$ which implies $F^{1}%
(u,v)=F^{2}(u,v)$ for all $(u,v)\in V\times V$ by the bilinearity. Thus from
the continuity of the extension, we have $F^{1}(u,v)=F^{2}(u,v)$ for all
$(u,v)\in H\times H.$ That is, $F^{1}=F^{2}$.

The converse of the assertion is trivial.

\textit{Step 2: proof of the theorem.} We fix any versions of $G(e_{i},e_{j})$
for $i,j\geq1.$ For each given $\omega$, we define the effect of $\bar
{G}(\cdot,\cdot)(\omega)$ on the basis:
\begin{equation}
\bar{G}(e_{i},e_{j})(\omega):=a_{ij}^{\omega}:=G(e_{i},e_{j})(\omega),\quad
i,j\geq1. \label{Myeq1-1}%
\end{equation}
Since the elements in $V\times V$ is countable, we have from (\ref{Myeq1-8})
that, for $P$-a.s. $\omega,$
\begin{equation}%
\begin{split}
|\sum_{i=1}^{n}\sum_{j=1}^{m}\alpha_{i}\beta_{j}a_{ij}^{\omega}| =  &  |
G(\sum_{i=1}^{n}\alpha_{i}e_{i},\sum_{j=1}^{m}\beta_{j}e_{j})(\omega)|\leq
g(\omega)\Vert\sum_{i=1}^{n}\alpha_{i}e_{i}\Vert_{H}\Vert\sum_{j=1}^{m}%
\beta_{j}e_{j}\Vert_{H},\\
&  \ \ \ \ \ \ \ \ \ \ \ \ \ \ \ \ \ \ \ \ \ \ \ \ \ \text{for all }(u,v)
=(\sum_{i=1}^{n}\alpha_{i}e_{i},\sum_{j=1}^{m}\beta_{j}e_{j})\in V\times V.
\end{split}
\label{Myeq1-2}%
\end{equation}
We denote by $\Omega_{0}$ the $\mathcal{G}$-measurable set of full measure in
which the inequality (\ref{Myeq1-2}) holds. For each fixed $\omega\in
\Omega_{0},$ we can apply Step 1 to extend $\bar{G}$ to be an element in
$\mathfrak{L}_{2}(H\times H),$ which satisfies $\Vert\bar{G}(\omega
)\Vert_{\mathfrak{L}_{2}(H\times H)}\leq g(\omega).$ On the exception set
$\Omega\setminus\Omega^{0}$, let $\bar{G}$ take the zero element in
$\mathfrak{L}_{2}(H\times H)$. Thus we obtain a $\bar{G}\in L_{w}%
^{1}(\mathcal{G},\mathfrak{L}_{2}(H\times H))$ such that (\ref{Myeq2-10}) holds.

Now we prove $\bar{G}$ is a version of $G$. From the construction of $\bar
{G},$ we have $\bar{G}(e_{i},e_{j})=G(e_{i},e_{j})$ $P\text{-a.s.}$, for each
$i,j$. Assume $(u,v)=(\sum_{i=1}^{\infty}\alpha_{i}e_{i},\sum_{j=1}^{\infty
}\beta_{j}e_{j}),$ for $\alpha_{i},\beta_{j}\in\mathbb{R}$, $i,j\geq1.$ Then%
\[
\bar{G}(\sum_{i=1}^{n}\alpha_{i}e_{i},\sum_{j=1}^{m}\beta_{j}e_{j})=\sum
_{i=1}^{n}\sum_{j=1}^{m}\alpha_{i}\beta_{j}\bar{G}(e_{i},e_{j})=\sum_{i=1}%
^{n}\sum_{j=1}^{m}\alpha_{i}\beta_{j}G(e_{i},e_{j})=G(\sum_{i=1}^{n}\alpha
_{i}e_{i},\sum_{j=1}^{m}\beta_{j}e_{j}),\quad P\text{-a.s.}%
\]
Letting $n,m\rightarrow\infty$ (on a subsequence if necessary), from the
continuity of $\bar{G}$ and $G$ ($\bar{G}$ is continuous from $H\times H$ to
$\mathbb{R}$ pointwise, and $G$ is continuous from $H\times H$ to
$L^{1}(\mathcal{G})$ by (\ref{Myeq1-8}))$,$ we obtain
\[
\bar{G}(u,v)=G(u,v),\quad P\text{-a.s.}%
\]

To see the uniqueness, consider two $\mathfrak{L}_{2}(H\times H)$-valued
versions $\bar{G}^{1}$ and $\bar{G}^{2}$ of $G.$ For each $(u,v),$ we have
$\bar{G}^{1}(u,v)=G(u,v)=\bar{G}^{2}(u,v)$, $P\text{-a.s.}$ Thus, $\bar{G}%
^{1}(e_{i},e_{j})=\bar{G}^{2}(e_{i},e_{j})$ for all $i,j$, $P\text{-a.s}$.
From the uniqueness result in Step 1, we obtain that $\bar{G}^{1}=\bar{G}^{2}%
$, $P\text{-a.s.}$

Taking $g(\omega)=\Vert\bar{G}(\omega)\Vert_{\mathfrak{L}_{2}(H\times H)}$ for
$\omega\in\Omega$, we have the converse of the theorem.
\end{proof}

\begin{proof}
[Proof of Theorem \ref{Myth2-17}]We define $G(u,v):=\mathbb{E}%
[Y(u,v)|\mathcal{G}],$ for $(u,v)\in H\times H$. In view of Theorem
\ref{Myth1-1}, there is a $Z\in L_{w}^{1}(\mathcal{G},\mathfrak{L}_{2}(H\times
H))$ such that for each $(u,v)\in H\times H,$
\begin{equation}
Z(u,v)=G(u,v)=\mathbb{E}[Y(u,v)|\mathcal{G}],\quad P\text{-a.s.,}
\label{Myeq2-18}%
\end{equation}
and satisfies (\ref{Myeq2-11}) according to (\ref{Myeq2-10}). It is the
expectation of $Y$ conditioned on $\mathcal{G}$ and is unique by the
uniqueness result of $\mathfrak{L}_{2}(H\times H)$-valued versions in Theorem
\ref{Myth1-1}. On the contrary, assume there exists such a conditional
expectation $\mathbb{E}[Y|\mathcal{G}]\in L_{w}^{1}(\mathcal{G},\mathfrak{L}%
_{2}(H\times H)).$ Then for any $(u,v)\in H\times H$, from the definition of
the conditional expectation that
\[
\mathbb{E}[Y|\mathcal{G}](u,v)=\mathbb{E}[Y(u,v)|\mathcal{G}],\quad
P\text{-a.s.},
\]
we have%
\[
|\mathbb{E}[Y(u,v)|\mathcal{G}]|=|\mathbb{E}[Y|\mathcal{G}](u,v)|\leq
\Vert\mathbb{E}[Y|\mathcal{G}]\Vert_{\mathfrak{L}_{2}(H\times H)}\Vert
u\Vert_{H}\Vert v\Vert_{H},\quad P\text{-a.s.}
\]
By taking $g=\Vert\mathbb{E}[Y|\mathcal{G}]\Vert_{\mathfrak{L}_{2}(H\times
H)},$ we obtain the domination condition.
\end{proof}

\begin{remark}
\label{rm2-2} \upshape{In  Theorem \ref{Myth2-17}, it is not necessary that $Y$ itself can be aggregated, for $\mathbb{E}[Y|\mathcal{G}]$ (referred to the mapping defined by $(u,v)\mapsto\mathbb{E}[Y(u,v)|\mathcal{G}]\in\mathfrak{L}_{2}(H\times H;L^{1}	(\mathcal{G}))$) to have an aggregated version. This may not be true in the subsequent applications; see Remark \ref{Myrm2-2}. Thus, we take $Y$ to be in the larger space $\mathfrak{L}_{2}(H\times H;L^{1}(\mathcal{F}))$ than $L_{w}^{1}(\mathcal{F},\mathfrak{L}_{2}(H\times H))$.}
\end{remark}

From $\mathfrak{L}_{2}(H\times H)=\mathfrak{L}(H)$, we can also write
(\ref{Myeq2-15}) as, for weakly $\mathcal{G}$-measurable $Z$ taking values in
$\mathfrak{L}(H)=\mathfrak{L}_{2}(H\times H)$ and $Y\in\mathfrak{L}%
_{2}(H\times H;L^{1}(\mathcal{F}))$,
\begin{equation}
\langle Zu,v\rangle=Z(u,v)=\mathbb{E}[Y(u,v)|\mathcal{G}],\quad P\text{-a.s.,}%
\ \forall(u,v)\in H\times H. \label{Myeq2-8}%
\end{equation}

\subsection{Formulation of the BSIE}

By a stochastic evolution operator on $H$, we mean a family of mappings
\[
\{L(t,s)\in\mathfrak{L}(L^{2}(\mathcal{F}_{t},H);L^{2}(\mathcal{F}%
_{s},H)):(t,s)\in\Delta\}
\]
with $\Delta=\{(t,s):0\leq t\leq s\leq T\}$. We adopt a definition of the
following formal adjoint $L^{\ast}$ for $L$: For any fixed $(t,s)\in\Delta$
and $u\in L^{1}(\mathcal{F}_{s},H),$ define $L^{\ast}(t,s)u$ by%
\[
(L^{\ast}(t,s)u)(v):=\langle u,L(t,s)v\rangle\ \ P\text{-a.s.},\quad\text{for
each}\ v\in L^{2}(\mathcal{F}_{t},H).
\]

Motivated by the constants of variation method for operator-valued SPDEs (see
(i) of Remark \ref{Rm2-2}), we shall consider a conditionally expected
$\mathfrak{L}(H)$-valued BSIE (i.e., $\mathfrak{L}(H)$-valued BSIE in the
conditional expectation form):
\begin{equation}
P(t)=\mathbb{E}[L^{\ast}(t,T)\xi L(t,T)+\int_{t}^{T}L^{\ast}%
(t,s)f(s,P(s))L(t,s)ds|\mathcal{F}_{t}],\quad t\in\lbrack0,T], \label{Myeq2-1}%
\end{equation}
where the coefficients $\xi$, $f$ and $L$ are given and subject to the
following assumptions:

\begin{description}
\item[$(H1)$] There exists some constant $\Lambda\geq0$ such that for each
$(t,s)\in\Delta$ and $u\in L^{4}(\mathcal{F}_{t},H),$ it holds that
$L(t,s)u\in L^{4}(\mathcal{F}_{s},H)$,
\[
\mathbb{E}[\left\Vert L(t,s)u\right\Vert _{H}^{4}|\mathcal{F}_{t}]\leq
\Lambda\Vert u\Vert_{H}^{4},\quad P\text{-a.s}.,
\]
and $(\omega,t,s)\mapsto(L(t,s)u)(\omega)$ admits a jointly measurable version.

\item[$(H2)$] $\xi\in L_{w}^{2}(\mathcal{F}_{T},\mathfrak{L}(H))$; the
function $f(w,t,p):\Omega\times\lbrack0,T]\times\mathfrak{L}(H)\rightarrow
\mathfrak{L}(H)$ is $\mathcal{P}\otimes L_{w}/L_{w}$-measurable and satisfies
the Lipschitz condition in $p$ with constant $\lambda\geq0$; $f(\cdot
,\cdot,0)\in L_{\mathbb{F},w}^{2}(0,T;\mathfrak{L}(H)).$
\end{description}

\begin{remark}\upshape{
 Fix
	any $u\in L^{1}(\mathcal{F}_{s},H).$ For each $v\in L^{2}(\mathcal{F}%
	_{t},H),$ $L^{\ast}(t,s)u$ maps $v$ to a real-valued $\mathcal{F}_{s}$-measurable  random variable $\langle u,L(t,s)v\rangle$. But the quantity  $\langle u,L(t,s)v\rangle$ is not necessarily integrable. It is integrable if, according to the H\"{o}lder inequality, one of the following is imposed: (i) $u\in L^{2}(\mathcal{F}_{s},H)$; (ii) $u\in L^{\frac43}(\mathcal{F}_{s},H)$, $v\in L^{4}(\mathcal{F}_{t},H)$ and $(H1)$ holds.}
\end{remark}

We first show that the operator-valued conditional expectation on the right
hand side of the equation is meaningful. To apply the result in Theorem
\ref{Myth2-17}, we begin with assigning a rigorous meaning to the term
$L^{\ast}(t,T)\xi L(t,T)+\int_{t}^{T}L^{\ast}(t,s)f(s,P(s))L(t,s)ds$ inside
the conditional expectation and demonstrate that it belongs to $\mathfrak{L}%
_{2}(H\times H;L^{1}(\mathcal{F}_{T}))$.

\begin{remark}
\label{Myrm2-2} \upshape{From the settings for $L$, we know that
$L(t,s)$ is not $\mathfrak{L}(H)$-valued for pointwise $\omega$ (see also subsection 2.4 for the explanations on this setting), and so
$L^{\ast}(t,T)\xi L(t,T)+\int_{t}^{T}L^{\ast}(t,s)f(s,P(s))L(t,s)ds$ is not. That is, we cannot expect that this term belongs to $L_{w}^{1}(\mathcal{F}_T,\mathfrak{L}_{2}(H\times H))$, but rather, as we shall see later,
is an element in
$\mathfrak{L}_{2}(H\times H;L^{1}(\mathcal{F}_{T}))$.}
\end{remark}

\begin{remark}
\upshape{
	\textrm{{ (i)}} For any sub-$\sigma$-algebra
$\mathcal{G}$ of $\mathcal{F}$ and a mapping $\eta:$%
 $\Omega\rightarrow\mathfrak{L}(H)$, the following four statements are equivalent:

(a) $\eta$ is weakly $\mathcal{G}$-measurable;

(b) For any $u\in H,$ $\eta u:\Omega\rightarrow H$ is (strongly)
$\mathcal{G}$-measurable  (note that since $H$ is separable, the notions of
\textit{measurable}, \textit{weakly measurable} and \textit{strongly
	measurable} are the same);

(c) For any (strongly) $\mathcal{G}$-measurable $u,v:\Omega\rightarrow H,$
the real-valued function $\langle\eta u,v\rangle$ is $\mathcal{G}$-measurable;

(d) For any (strongly) $\mathcal{G}$-measurable $u:\Omega\rightarrow H,$ the function
$\eta u:\Omega\rightarrow H$ is (strongly) $\mathcal{G}$-measurable.

Indeed, it can be proved as follows:

\noindent(a)$\Longrightarrow$(b): The real-valued function $\langle\eta
u,v\rangle$ is $\mathcal{F}_{s}$-measurable for each $v\in H$. This means that
$\eta u:\Omega\rightarrow H$ is weakly $\mathcal{F}_{s}$-measurable. Noting
that $H$ is separable, this is equivalent to stating that $\eta u:\Omega
\rightarrow H$ is (strongly) $\mathcal{F}_{s}$-measurable.

\noindent(b)$\Longrightarrow$(a): Since $\eta u:\Omega\rightarrow H$ is (strongly)
$\mathcal{G}$-measurable, then it is weakly measurable, i.e., for any $v\in
H,$ the real-valued function $\langle\eta u,v\rangle$ is $\mathcal{G}$-measurable.

\noindent Surely, (c)$\Longrightarrow$(a) and (d)$\Longrightarrow$(bi).

\noindent Now we only prove (b)$\Longrightarrow$(d), and the proof of
(a)$\Longrightarrow$(c) is similar. 
First for any simple%
\[
u=\sum_{i=1}^{N}u_{i}I_{A_{i}},\quad\text{with}\ u_{i}\in H,\ A_{i}%
\in\mathcal{F}_{s},
\]
we have that%
\[
\eta u=\sum_{i=1}^{N}(\eta u_{i})I_{A_{i}}%
\]
is (strongly) $\mathcal{G}$-measurable. Finally, for any $H$-valued (strongly)
$\mathcal{G}$-measurable $u$, we can take a simple sequence%
\[
u_{k}\rightarrow u\quad\text{pointwise},\ \text{as}\ k\rightarrow\infty.
\]
Then%
\[
\eta u=\eta(\lim_{k\rightarrow\infty}u_{k})=\lim_{k\rightarrow\infty}\eta
u_{k}%
\]
is (strongly) $\mathcal{G}$-measurable. The proof is complete.

\textrm{{ (ii)}} From (i), we know that, the \textit{weakly measurability} notion used in this paper is coincide with the notion of \textit{strongly measurability} used in \cite{DZ92}. But we prefer to call it weak measurability since it is weak than the usual (norm-) measurability.
 According to (i), we know that for any $\eta\in L_{w}^{2}(\mathcal{F}_{s},\mathfrak{L}(H))$ and $u\in H$, the
random mapping $\eta L(t,s)u:\Omega\rightarrow H$ is (strongly) $\mathcal{F}%
_{s}$-measurable. 

It is easy to see that similar results hold for weakly adapted and progressively measurable processes. 
Moreover, by a similar proof, the above equivalence relationship also holds for different separable Hilbert spaces $H_1,H_2$ and mappings taking values in $\mathfrak{L}(H_1,H_2)$.
}
\end{remark}

Under the assumption $(H1)$, given any $\eta\in L_{w}^{2}(\mathcal{F}%
_{s},\mathfrak{L}(H))$ and $(u,v)\in H\times H,$  from the H\"{o}lder inequality and the condition $(H1)$, it is
straightforward to check that%
\[
\mathbb{\mathbb{E}}[\Vert\eta L(t,s)u\Vert_{H}^{\frac{4}{3}}]\leq
(\mathbb{\mathbb{E}}[\Vert\eta\Vert_{H}^{2}])^{\frac{2}{3}}(\mathbb{\mathbb{E}%
}[\Vert L(t,s)u\Vert_{H}^{4}])^{\frac{1}{3}}\leq\Lambda^{\frac{1}{3}%
}(\mathbb{\mathbb{E}}[\Vert\eta\Vert_{H}^{2}])^{\frac{2}{3}}\Vert u\Vert
_{H}^{\frac{4}{3}}<\infty.
\]
Thus the random function $\eta L(t,s)u\in L^{\frac{4}{3}%
}(\mathcal{F}_{s},H)$.

Moreover,
\begin{align*}
\mathbb{E}[|(L^{\ast}(t,s)\eta L(t,s)u)(v)|]  &  =\mathbb{E}[|\langle\eta
L(t,s)u,L(t,s)v\rangle|]\\
&  \leq(\mathbb{E}[\Vert L(t,s)u\Vert_{H}^{4}])^{\frac{1}{4}}(\mathbb{E}%
[\Vert\eta\Vert_{\mathfrak{L}(H)}^{2}])^{\frac{1}{2}}(\mathbb{E}[\Vert
L(t,s)v\Vert_{H}^{4}])^{\frac{1}{4}}\\
&  \leq\Lambda^{\frac{1}{2}}(\mathbb{E}[\Vert\eta\Vert_{\mathfrak{L}(H)}%
^{2}])^{\frac{1}{2}}\Vert u\Vert_{H}\Vert v\Vert_{H}.
\end{align*}
Thus, we have $L^{\ast}(t,s)\eta L(t,s)\in\mathfrak{L}(H;\mathfrak{L}(H;L^{1}
(\mathcal{F}_{s})))=\mathfrak{L}_{2}(H\times H;L^{1}(\mathcal{F}_{s}))$ and we
can also write that $(L^{\ast}(t,s)\eta L(t,s)u)(v)=L^{\ast}(t,s)\eta L(t,s)(u,v)$.
In particular, $L^{\ast}(t,T)\xi L(t,T)\in\mathfrak{L}_{2}(H\times
H;L^{1}(\mathcal{F}_{T})).$

Now we consider the integral term. In general, for a $g\in\mathfrak{L}%
_{2}(H\times H;L_{\mathbb{F}}^{1}(t,T)),$ following the idea of Pettis
integration (see, e.g., \cite{Pe38}), we define its integral with respect to
time $\int_{t}^{T}g(s)ds$ in a weak sense by%
\[
(\int_{t}^{T}g(s)ds)(u,v):=\int_{t}^{T}g(s)(u,v)ds \ \ P\text{-a.s}%
,\quad\forall(u,v)\in H\times H.
\]
Then $\int_{t}^{T}g(s)ds\in\mathfrak{L}_{2}(H\times H;L^{1}(\mathcal{F}_{T}))$
by the observation that
\[
\mathbb{E}[|(\int_{t}^{T}g(s)ds)(u,v)|] \leq\mathbb{E}[\int_{t}^{T}%
|g(s)(u,v)|ds] \leq C\Vert u\Vert_{H}\Vert v\Vert_{H}.
\]
Note that for any $h\in L_{\mathbb{F},w}^{2}(t,T;\mathfrak{L}(H))$ and
$(u,v)\in H\times H,$
\begin{align*}
& \mathbb{E}[\int_{t}^{T}|L^{\ast}(t,s)h(s)L(t,s)(u,v)|ds]   =\mathbb{E}%
[\int_{t}^{T}|\langle h(s)L(t,s)u,L(t,s)v\rangle|ds]\\
& \ \ \ \ \ \ \ \ \ \ \ \ \leq(\int_{t}^{T}\mathbb{E}[\Vert L(t,s)u\Vert_{H}^{4}]ds)^{\frac{1}{4}%
}(\mathbb{E}[\int_{t}^{T}\Vert h(s)\Vert_{\mathfrak{L}(H)}^{2}ds])^{\frac
{1}{2}}(\int_{t}^{T}\mathbb{E}[\Vert L(t,s)v\Vert_{H}^{4}]ds)^{\frac{1}{4}}\\
& \ \ \ \ \ \ \ \ \ \ \ \  \leq\Lambda^{\frac{1}{2}}T^{\frac{1}{2}}(\mathbb{E}[\int_{t}^{T}\Vert
h(s)\Vert_{\mathfrak{L}(H)}^{2}ds])^{\frac{1}{2}}\Vert u\Vert_{H}\Vert
v\Vert_{H}.
\end{align*}
Thus $[t,T]\ni s\mapsto L^{\ast}(t,s)h(s)L(t,s)\in\mathfrak{L}_{2}(H\times
H;L_{\mathbb{F}}^{1}(t,T))$ and the integral $\int_{t}^{T}L^{\ast
}(t,s)h(s)L(t,s)ds\in\mathfrak{L}_{2}(H\times H;L^{1}(\mathcal{F}_{T}))$ is defined.

The BSIE is considered as an equation in the space $\mathfrak{L}(H)$ as follows.

\begin{definition}
\label{Def1-1} A process $P\in L_{\mathbb{F},w}^{2}(0,T;\mathfrak{L}(H))$ is
called a solution of (\ref{Myeq2-1}) if for each $0\leq t\leq T,$
\begin{equation}
P(t)=\mathbb{E}[L^{\ast}(t,T)\xi L(t,T)+\int_{t}^{T}L^{\ast}
(t,s)f(s,P(s))L(t,s)ds|\mathcal{F}_{t}],\quad P\text{-a.s}. \label{Myeq2-6}%
\end{equation}

\end{definition}

Given any $P\in L_{\mathbb{F},w}^{2}(0,T;\mathfrak{L}(H))$. Since it is
$\mathcal{P}/L_{w}$-measurable, we deduce by $(H2)$ and the measurability of
composition that $f(\cdot,P(\cdot))$ is $\mathcal{P}/L_{w}$-measurable, i.e.,
weakly progressively measurable. From this and the Lipschitz continuity of
$f$, we obtain that $f(\cdot,P(\cdot))\in L_{\mathbb{F},w}^{2}%
(0,T;\mathfrak{L}(H))$. Thus,
\begin{equation}
\label{Myeq2-12}L^{\ast}(t,T)\xi L(t,T)+\int_{t}^{T}L^{\ast}%
(t,s)f(s,P(s))L(t,s)ds\in\mathfrak{L}_{2}(H\times H;L^{1}(\mathcal{F}_{T})).
\end{equation}
Then the conditional expectation on right hand side of BSIE (\ref{Myeq2-6}) is
a well-defined operator-valued random variable as long as we check the
domination condition (\ref{Myeq2-11}), which shall be done in the next
subsection. In what follows, $C>0$ will denote a constant which may vary from
line to line.

\subsection{Existence and uniqueness of solutions}

We have the following well-posedness result on BSIEs.

\begin{theorem}
\label{Mainthm-1} Let Assumptions $(H1)$ and $(H2)$ be satisfied. Then there
exists a unique (up to modification) solution $P$ to BSIE (\ref{Myeq2-1}).
Moreover, for each $t\in\lbrack0,T],$
\begin{equation}
\Vert P(t)\Vert_{\mathfrak{L}(H)}^{2}\leq C\mathbb{E}[\Vert\xi\Vert
_{\mathfrak{L}(H)}^{2}+\int_{t}^{T}\Vert f(s,0)\Vert_{\mathfrak{L}(H)}%
^{2}ds|\mathcal{F}_{t}],\quad P\text{-a.s}., \label{Myeq3-23}%
\end{equation}
for some constant $C$ depending on $\Lambda$ and $\lambda$.
\end{theorem}

To prove this theorem, we need the following lemmas. First we see that the
conditional expectation on the right-hand side of (\ref{Myeq2-6}) is well-defined.

\begin{lemma}
\label{Myth2-5} Suppose $(H1)$ and $(H2)$ hold. For any $p\in L_{\mathbb{F}%
,w}^{2}(0,T;\mathfrak{L}(H))$ and $0\leq t\leq T$, we define
\[
Y_{t,T}^{p}:=L^{\ast}(t,T)\xi L(t,T)+\int_{t}^{T}L^{\ast}%
(t,s)f(s,p(s))L(t,s)ds.
\]
Then $\mathbb{E}[Y_{t,T}^{p}|\mathcal{F}_{t}]\in L_{w}^{2}(\mathcal{F}%
_{t},\mathfrak{L}(H))$, and there exists some constant $C>0$ depending on
$\Lambda$ and $\lambda$ such that
\begin{equation}
\Vert\mathbb{E}[Y_{t,T}^{p}|\mathcal{F}_{t}]\Vert_{\mathfrak{L}(H)}\leq
C(\mathbb{E}[\Vert\xi\Vert_{\mathfrak{L}(H)}^{2}+\int_{t}^{T}\Vert
p(s)\Vert_{\mathfrak{L}(H)}^{2}ds+\int_{t}^{T}\Vert f(s,0)\Vert_{\mathfrak{L}%
(H)}^{2}ds|\mathcal{F}_{t}])^{\frac{1}{2}},\quad P\text{-a.s.}
\label{Myeq3-17}%
\end{equation}
Moreover, $\{\mathbb{E}[Y_{t,T}^{p}|\mathcal{F}_{t}]\}_{t\in\lbrack0,T]}\in
L_{\mathbb{F},w}^{2}(0,T;\mathfrak{L}(H))$.
\end{lemma}

\begin{proof}
First we have $Y_{t,T}^{p}\in\mathfrak{L}_{2}(H\times H;L^{1}(\mathcal{F}%
_{T}))$ from the discussions in the last subsection. For any $(u,v)\in H\times
H,$ we directly calculate
\begin{align*}
|\mathbb{E}[Y_{t,T}^{p}(u,v)|\mathcal{F}_{t}]|  &  =|\mathbb{E}[\langle\xi
L(t,T)u,L(t,T)v\rangle+\int_{t}^{T}\langle f(s,p(s))L(t,s)u,L(t,s)v\rangle
ds|\mathcal{F}_{t}]|\\
&  \leq(\mathbb{E}[\Vert L(t,T)u\Vert_{H}^{4}|\mathcal{F}_{t}])^{\frac{1}{4}%
}(\mathbb{E}[\Vert\xi\Vert_{\mathfrak{L}(H)}^{2}|\mathcal{F}_{t}])^{\frac
{1}{2}}(\mathbb{E}[\Vert L(t,T)v\Vert_{H}^{4}|\mathcal{F}_{t}])^{\frac{1}{4}%
}\\
&  \quad+(\int_{t}^{T}\mathbb{E}[\Vert L(t,s)u\Vert_{H}^{4}|\mathcal{F}%
_{t}]ds)^{\frac{1}{4}}(\mathbb{E}[\int_{t}^{T}\Vert f(s,p(s))\Vert
_{\mathfrak{L}(H)}^{2}ds|\mathcal{F}_{t}])^{\frac{1}{2}}(\int_{t}%
^{T}\mathbb{E}[\Vert L(t,s)v\Vert_{H}^{4}|\mathcal{F}_{t}]ds)^{\frac{1}{4}}\\
&  \leq C\Vert u\Vert_{H}\Vert v\Vert_{H}\{(\mathbb{E}[\Vert\xi\Vert
_{\mathfrak{L}(H)}^{2}|\mathcal{F}_{t}])^{\frac{1}{2}}+(\mathbb{E}[\int%
_{t}^{T}\Vert f(s,p(s))\Vert_{\mathfrak{L}(H)}^{2}ds|\mathcal{F}_{t}%
])^{\frac{1}{2}}\}\\
&  \leq C\Vert u\Vert_{H}\Vert v\Vert_{H}(\mathbb{E}[\Vert\xi\Vert
_{\mathfrak{L}(H)}^{2}+\int_{t}^{T}\Vert p(s)\Vert_{\mathfrak{L}(H)}%
^{2}ds+\int_{t}^{T}\Vert f(s,0)\Vert_{\mathfrak{L}(H)}^{2}ds|\mathcal{F}%
_{t}])^{\frac{1}{2}},\quad P\text{-a.s.}%
\end{align*}
Then by Theorem \ref{Myth2-17}, $\mathbb{E}[Y_{t,T}^{p}|\mathcal{F}_{t}],$ the
expectation of $Y_{t,T}^{p}$ conditioned on $\mathcal{F}_{t},$ is a
well-defined $\mathfrak{L}(H)$-valued random variables (see (\ref{Myeq2-8})),
and (\ref{Myeq3-17}) follows from (\ref{Myeq2-11}). Thus, $\mathbb{E}%
[Y_{t,T}^{p}|\mathcal{F}_{t}]\in L_{w}^{2}(\mathcal{F}_{t},\mathfrak{L}(H))$

It remains to show that $\{\mathbb{E}[Y_{t,T}^{p}|\mathcal{F}_{t}%
]\}_{t\in\lbrack0,T]}$ has a weakly progressively measurable version. This is
obtained from the following Lemma \ref{Myle3-10} and the fact that, for each
$(u,v)\in H\times H,$ $\{\mathbb{E}[Y_{t,T}^{p}(u,v)|\mathcal{F}_{t}%
]\}_{t\in\lbrack0,T]}$ has a progressively measurable version by considering
its optional projection (see \cite[Corollary 7.6.8]{CE15}).
\end{proof}

\begin{lemma}
\label{Myle3-10} Let $Y$ be an $\mathfrak{L}(H)$-valued weakly adapted process
satisfying $Y_{t}\in L_{w}^{1}(\mathcal{F}_{t},\mathfrak{L}(H))$ for $0\leq
t\leq T$. Then $Y$ has an $\mathfrak{L}(H)$-valued weakly progressively
measurable modification $\bar{Y}$ if and only if for each $(u,v)\in H\times
H,$ $\{\langle Y_{t}u,v\rangle\}_{0\leq t\leq T}$ has a progressively
measurable modification.
\end{lemma}

\begin{proof}
We look for the desired process by a variant of Step 2 in the proof of Theorem
\ref{Myth1-1} in the space $\mathfrak{L}_{2}(H\times H)$ of bilinear mapping,
and the result in the original form can be obtained via the isometry
$\mathfrak{L}_{2}(H\times H)=\mathfrak{L}(H)$. For any $(u,v)\in H\times H,$
we denote by $\{y_{t}(u,v)\}_{0\leq t\leq T}$ and $\{h_{t}\}_{0\leq t\leq T}$
the progressively measurable modifications of $\{\langle Y_{t}u,v\rangle
\}_{0\leq t\leq T}$ and $\{\Vert Y_{t}\Vert_{\mathfrak{L}(H)}\}_{0\leq t\leq
T},$ respectively. Then for any $t\in\lbrack0,T],$
\begin{equation}
|y_{t}(u,v)|=\langle Y_{t}u,v\rangle\leq\Vert Y_{t}\Vert_{\mathfrak{L}%
(H)}\Vert u\Vert_{H}\Vert v\Vert_{H}=h_{t}\Vert u\Vert_{H}\Vert v\Vert
_{H},\quad P\text{-a.s.} \label{Myeq2-27}%
\end{equation}
Thus, $y_{t}\in\mathfrak{L}_{2}(H\times H;L^{1}(\mathcal{F}_{t})).$ Adopt the
notions in the proof of Theorem \ref{Myth1-1} and fix any versions of process
$y(e_{i},e_{j})$ for $i,j\geq1$. For every $t$, we define
\[
\bar{Y}_{t}(e_{i},e_{j})(\omega):=a_{ij}^{t,\omega}:=y_{t}(e_{i},e_{j}%
)(\omega),\quad i,j\geq1,\text{ for each }\omega.
\]
For any fixed $t$, according to (\ref{Myeq2-27}), we have $P$-a.s. that
\begin{equation}%
\begin{split}
|\sum_{i=1}^{n}\sum_{j=1}^{m}\alpha_{i}\beta_{j}a_{ij}^{t,\omega}|  &
=|y_{t}(\sum_{i=1}^{n}\alpha_{i}e_{i},\sum_{j=1}^{m}\beta_{j}e_{j}%
)(\omega)|\leq h_{t}(\omega)\Vert\sum_{i=1}^{n}\alpha_{i}e_{i}\Vert_{H}%
\Vert\sum_{j=1}^{m}\beta_{j}e_{j}\Vert_{H},\\
&  \ \ \ \ \ \ \ \ \ \ \ \ \ \ \ \ \ \ \ \ \ \ \ \ \ \ \ \ \ \ \text{for all
}(u,v) =(\sum_{i=1}^{n}\alpha_{i}e_{i},\sum_{j=1}^{m}\beta_{j}e_{j})\in
V\times V,
\end{split}
\label{Myeq3-3}%
\end{equation}
and we denote the set (on $\Omega)$ in which the above relationship holds by
$\Omega_{t}.$ Similar to Step 2 in the proof of Theorem \ref{Myth1-1},
$\bar{Y}_{t}(\omega)$ has an extension in $\mathfrak{L}_{2}(H\times H)$ on
$\Omega_{t}$ and we set $\bar{Y}_{t}=0$ in $\Omega_{t}^{c}.$ Then $\bar{Y}%
_{t}\leq h_{t}$ $P$-a.s. Denote by $A$ the progressively measurable set of all
points $(t,\omega)$ in $\Omega\times\lbrack0,T]$ such that (\ref{Myeq3-3})
holds. Note that $\Omega_{t}$ is the section of $A$ for each $t$. Then the
$\mathfrak{L}_{2}(H\times H)$-valued process $\bar{Y}$ is automatically weakly
progressively measurable and $\bar{Y}_{t}(u,v)=y_{t}(u,v)$ $P$-a.s., for any
$(u,v)\in H\times H$ and $t\in\lbrack0,T],$ by a similar analysis as in the
proof of Theorem \ref{Myth1-1}. Since for each $t$, ${Y}_{t}$ and $\bar{Y}%
_{t}$ are both aggregated versions of $y_{t}$ in the sense of Theorem
\ref{Myth1-1}, we deduce from the uniqueness result in that theorem that
${Y}_{t}=\bar{Y}_{t}\ P\text{-a.s.}$ That is, $\bar{Y}$ is a modification of
$Y$.

The inversed assertion is trivial, by noting that for each $(u,v)\in H\times
H,$ $\{\langle\bar{Y}_{t}u,v\rangle\}_{0\leq t\leq T}$ is a progressively
measurable modification of $\{\langle Y_{t}u,v\rangle\}_{0\leq t\leq T}.$
\end{proof}

The following is the a priori estimate for the difference between two solutions.

\begin{theorem}
\label{Myth2-10} Let $L$ satisfy $(H1)$ and $(\xi, f)$ and $(\tilde{\xi},
\tilde{f})$ satisfy $(H2)$. Assume that $P,\tilde{P}\in L_{\mathbb{F},w}%
^{2}(0,T;\mathfrak{L}(H))$ are solutions to BSIEs
\[
P(t)=\mathbb{E}[L^{\ast}(t,T)\xi L(t,T)+\int_{t}^{T}L^{\ast}%
(t,s)f(s,P(s))L(t,s)ds|\mathcal{F}_{t}],\quad t\in\lbrack0,T]
\]
and%
\[
\tilde{P}(t)=\mathbb{E}[L^{\ast}(t,T)\tilde{\xi}L(t,T)+\int_{t}^{T}L^{\ast
}(t,s)\tilde{f}(s,\tilde{P}(s))L(t,s)ds|\mathcal{F}_{t}],\quad t\in
\lbrack0,T].
\]
Then there exists a constant $C>0$ which depends on $\Lambda$ and $\lambda$
such that, for each $t\in\lbrack0,T],$
\begin{equation}
\Vert P(t)-\tilde{P}(t)\Vert_{\mathfrak{L}(H)}^{2}\leq C\mathbb{E}[\Vert
\xi-\tilde{\xi}\Vert_{\mathfrak{L}(H)}^{2}+\int_{t}^{T}\Vert f(s,\tilde
{P}(s))-\tilde{f}(s,\tilde{P}(s))\Vert_{\mathfrak{L}(H)}^{2}ds|\mathcal{F}%
_{t}],\quad P\text{-a.s.} \label{Myeq2-19}%
\end{equation}

\end{theorem}

\begin{proof}
For any $t\in\lbrack0,T],$ we have $P\text{-a.s.}$ that
\[
P(t)-\tilde{P}(t)=\mathbb{E}[L^{\ast}(t,T)(\xi-\tilde{\xi})L(t,T)+\int_{t}%
^{T}L^{\ast}(t,s)(f(s,P(s)-\tilde{P}(s)+\tilde{P}(s))-\tilde{f}(s,\tilde
{P}(s)))L(t,s)ds|\mathcal{F}_{t}].
\]
Applying Lemma \ref{Myth2-5}, we obtain%
\begin{align*}
\Vert P(t)-\tilde{P}(t)\Vert_{\mathfrak{L}(H)}^{2}  &  \leq C\{\mathbb{E}%
[\Vert\xi-\tilde{\xi}\Vert_{\mathfrak{L}(H)}^{2}+\int_{t}^{T}\Vert
f(s,\tilde{P}(s))-\tilde{f}(s,\tilde{P}(s))\Vert_{\mathfrak{L}(H)}%
^{2}ds|\mathcal{F}_{t}]\\
&  \ \ \ \, +\mathbb{E}[\int_{t}^{T}\Vert P(s)-\tilde{P}(s)\Vert
_{\mathfrak{L}(H)}^{2}ds|\mathcal{F}_{t}]\},\quad P\text{-a.s.}%
\end{align*}
Fix any $r\leq T$ and any $A\in\mathcal{F}_{r}$. For $t\in\lbrack r,T],$
multiplying by $I_{A}$ and taking expectation on both sides, we obtain that
\begin{align*}
\mathbb{E}[\Vert P(t)-\tilde{P}(t)\Vert_{\mathfrak{L}(H)}^{2}I_{A}]  &  \leq
C\{\mathbb{E}[(\Vert\xi-\tilde{\xi}\Vert_{\mathfrak{L}(H)}^{2}+\int_{r}%
^{T}\Vert f(s,\tilde{P}(s))-\tilde{f}(s,\tilde{P}(s))\Vert_{\mathfrak{L}%
(H)}^{2}ds)I_{A}]\\
&  \ \ \ \, +\int_{t}^{T}\mathbb{E}[\Vert P(s)-\tilde{P}(s)\Vert
_{\mathfrak{L}(H)}^{2}I_{A}]ds\}.
\end{align*}
Then an application of Gronwall's inequality yields
\[
\mathbb{E}[\Vert P(t)-\tilde{P}(t)\Vert_{\mathfrak{L}(H)}^{2}I_{A}]\leq
C\mathbb{E}[(\Vert\xi-\tilde{\xi}\Vert_{\mathfrak{L}(H)}^{2}+\int_{r}^{T}\Vert
f(s,\tilde{P}(s))-\tilde{f}(s,\tilde{P}(s))\Vert_{\mathfrak{L}(H)}^{2}%
ds)I_{A}],\quad t\in\lbrack r,T].
\]
From the arbitrariness of $A,$ this implies for $t\in\lbrack r,T]$ that
\[
\mathbb{E}[\Vert P(t)-\tilde{P}(t)\Vert_{\mathfrak{L}(H)}^{2}|\mathcal{F}%
_{r}]\leq C\mathbb{E}[\Vert\xi-\tilde{\xi}\Vert_{\mathfrak{L}(H)}%
^{2}|\mathcal{F}_{r}]+\mathbb{E}[\int_{r}^{T}\Vert f(s,\tilde{P}(s))-\tilde
{f}(s,\tilde{P}(s))\Vert_{\mathfrak{L}(H)}^{2}ds|\mathcal{F}_{r}],\quad
P\text{-a.s.}%
\]
Letting $t= r$, we obtain (\ref{Myeq2-19}).
\end{proof}

Now we prove Theorem \ref{Mainthm-1}.

\begin{proof}
We define the solution mapping $I: L_{\mathbb{F},w}^{2}(0,T;\mathfrak{L}%
(H))\to\ L_{\mathbb{F},w} ^{2}(0,T;\mathfrak{L}(H))$ by $I(p):=P$ for $p\in
L_{\mathbb{F},w}^{2}(0,T;\mathfrak{L}(H))$ with
\[
P(t):=\mathbb{E}[L^{\ast}(t,T)\xi L(t,T)+\int_{t}^{T}L^{\ast}%
(t,s)f(s,p(s))L(t,s)ds|\mathcal{F}_{t}],\quad t\in\lbrack0,T].
\]
In view of Lemma \ref{Myth2-5}, we have $I(p)\in L_{\mathbb{F},w}%
^{2}(0,T;\mathfrak{L}(H)).$ Thus, the mapping $I$ is well-defined.

Now we show that the mapping $I$ is a contraction on the interval
$[T-\delta,T]$ when $\delta>0$ is sufficiently small. Set $P:=I(p)$ and
$\tilde{P}:=I(\tilde{p})$ for $p,\tilde{p}\in L_{\mathbb{F},w}^{2}%
(0,T;\mathfrak{L}(H))$. From Theorem \ref{Myth2-10}, we have
\begin{align*}
\Vert P(t)-\tilde{P}(t)\Vert_{\mathfrak{L}(H)}^{2}  &  \leq C\mathbb{E}%
[\int_{t}^{T}\Vert f(s,p(s))-f(s,\tilde{p}(s))\Vert_{\mathfrak{L}(H)}%
^{2}ds|\mathcal{F}_{t}]\\
&  \leq C\mathbb{E}[\int_{t}^{T}\Vert p(s)-\tilde{p}(s)\Vert_{\mathfrak{L}%
(H)}^{2}ds|\mathcal{F}_{t}],\quad t\in\lbrack0,T].
\end{align*}
For any $0<\delta<T$, taking expectation on both sides and integrating over
time on $[T-\delta,T]$, we get
\[
\mathbb{E}[\int_{T-\delta}^{T}\Vert P(t)-\tilde{P}(t)\Vert_{\mathfrak{L}%
(H)}^{2}]\leq C\delta\mathbb{E}[\int_{T-\delta}^{T}\Vert p(s)-\tilde
{p}(s)\Vert_{\mathfrak{L}(H)}^{2}ds].
\]
So for sufficiently small $\delta>0$, we obtain a unique $P\in L_{\mathbb{F}%
,w}^{2}(T-\delta,T;\mathfrak{L}(H))$ such that $P=I(P)$ in $L_{\mathbb{F}%
,w}^{2}(T-\delta,T;\mathfrak{L}(H))$, which is also
\begin{equation}
P(t)=\mathbb{E}[L^{\ast}(t,T)\xi L(t,T)+\int_{t}^{T}L^{\ast}
(t,s)f(s,P(s))L(t,s)ds|\mathcal{F}_{t}],\quad P\text{-a.s., a.e. on
$[T-\delta,T].$} \label{Myeq3-24}%
\end{equation}

We take
\[
\tilde{P}(t):=\mathbb{E}[L^{\ast}(t,T)\xi L(t,T)+\int_{t}^{T}L^{\ast
}(t,s)f(s,P(s))L(t,s)ds|\mathcal{F}_{t}],\quad t\in\lbrack T-\delta,T].
\]
Then $\tilde{P}$ satisfies%
\[
\tilde{P}(t)=\mathbb{E}[L^{\ast}(t,T)\xi L(t,T)+\int_{t}^{T}L^{\ast
}(t,s)f(s,\tilde{P}(s))L(t,s)ds|\mathcal{F}_{t}],\quad P\text{-a.s}.,\ \forall
t\in\lbrack T-\delta,T],
\]
and thus is a solution of BSIE (\ref{Myeq2-1}) in the meaning of Definition
\ref{Def1-1} on $[T-\delta,T]$. The uniqueness of $\tilde{P}$ in this sense
follows from that in the meaning of (\ref{Myeq3-24}). Indeed, on $[T-\delta,
T]$, if $\tilde{P}^{\prime}$ is another solution of BSIE (\ref{Myeq2-1}) in
the sense of Definition \ref{Def1-1}, then they are both solutions of
(\ref{Myeq2-1}) in the meaning of (\ref{Myeq3-24}). Thus $\tilde{P}=\tilde
{P}^{\prime}$ $P\text{-a.s., a.e.}$ on $[T-\delta, T]$. From the identity
(\ref{Myeq2-6}) on $[T-\delta, T]$, we then obtain that $\tilde{P}%
(t)=\tilde{P}^{\prime}(t)$ $P$-a.s$.,$ for all $t\in\lbrack T-\delta,T].$

Denoting $\tilde{P}$ by $P$, we can apply a backward iteration procedure to
obtain a $P\in L_{\mathbb{F},w}^{2}(0,T;\mathfrak{L}(H))$ such that
\begin{equation}
P(t)=\mathbb{E}[L^{\ast}(t,T)\xi L(t,T)+\int_{t}^{T}L^{\ast}%
(t,s)f(s,P(s))L(t,s)ds|\mathcal{F}_{t}], \quad P\text{-a.s}.,\ \forall
t\in\lbrack0,T], \label{Myeq3-25}%
\end{equation}
since the constant $\delta$ can be chosen to be independent of the terminal
time in each step. The uniqueness of $P$ follows from the one on each interval.
\end{proof}

We end this section with the following continuity of $P$, and the proof is
given in the appendix. We first note that, if $P$ is a solution of
(\ref{Myeq2-1}), then for each $(u,v)\in H\times H$,
\begin{equation}%
\begin{split}
\langle P(t)u,v\rangle &  =\langle\mathbb{E}[L^{\ast}(t,T)\xi L(t,T)+\int%
_{t}^{T}L^{\ast}(t,s)f(s,P(s))L(t,s)ds|\mathcal{F}_{t}]u,v\rangle\\
&  =\mathbb{E}[\langle\xi L(t,T)u,L(t,T)v\rangle+\int_{t}^{T}\langle
f(s,P(s))L(t,s)u,L(t,s)v\rangle ds|\mathcal{F}_{t}],\quad P\text{-a.s.}%
\end{split}
\label{Myeq2-25}%
\end{equation}
From this and an approximation of simple random variables, we can obtain that
(\ref{Myeq2-25}) holds for $(u,v)\in L^{4}(\mathcal{F}_{t},H)\times
L^{4}(\mathcal{F}_{t},H)$.

\begin{proposition}
\label{Myth2-7} For some $\alpha\geq1,$ suppose $(H1)$, $(H2)$ and

\begin{description}
\item[$(H3)$] $(\xi,f(\cdot,\cdot,0))\in L_{w}^{2\alpha}(\mathcal{F}
_{T},\mathfrak{L}(H))\times L_{\mathbb{F},w}^{2,2\alpha}(0,T;\mathfrak{L}(H))$
and there exists some constant $\Lambda_{\alpha}\geq0$ such that for each
$0\leq t\leq r\leq s\leq T$ and $u\in L^{4\alpha}(\mathcal{F}_{t},H),$ it
holds that $L(t,s)=L(t,r)L(r,s)$,
\[
\mathbb{E}[\left\Vert L(t,s)u\right\Vert _{H}^{4\alpha}|\mathcal{F}_{t}%
]\leq\Lambda_{\alpha}\Vert u\Vert_{H}^{4\alpha}\ P\text{-a.s.} \ \text{ and
}[t,T]\ni s\mapsto L(t,s)u\text{ is strongly continuous in }L^{4\alpha
}(\mathcal{F}_{T},H)\text{.}%
\]

\end{description}

Let $P$ be the solution of (\ref{Myeq2-1}). Then, for each $t\in\lbrack0,T)$
and $u,v\in L^{4\alpha}(\mathcal{F}_{t},H),$ we have
\[
\lim_{\delta\downarrow0} \mathbb{E}[|\langle P(t+\delta)u,v\rangle-\langle
P(t)u,v\rangle|^{\alpha}]=0.
\]

\end{proposition}

\begin{remark}
\label{rem1-10}%
\upshape{ According to similar proofs, the discussions and results in this section hold for
		a more general setting that the bilinear framework is replaced by the
		$k$-linear framework, for $k=1,2,3\cdots,$ for possibly different separable
		Hilbert spaces (and even more generally, Banach spaces with Schauder basis)
		$H_{j}$ and stochastic evolution operators $L_{j}(t,s)$ on $H_{j}$, for $1\leq
		j\leq k$. We only give a short description of BSIEs for convenience as follows.
		
		We make use of the similar weakly measurability meaning for $\mathfrak{L}_{k}		(H_{1}\times H_{2}\times\cdots\times H_{k})$-valued random variables and
		stochastic processes as in the bililnear case (with a direct modificastion
		from the case of $k=2$ to the general $k)$.\ Given terminal $\xi\in L_{w}		^{2}(\mathcal{F}_{T},\mathfrak{L}_{k}(H_{1}\times H_{2}\times\cdots\times
		H_{k}))$, generator $f(w,t,p):\Omega\times\lbrack0,T]\times\mathfrak{L}		_{k}(H_{1}\times H_{2}\times\cdots\times H_{k})\rightarrow\mathfrak{L}		_{k}(H_{1}\times H_{2}\times\cdots\times H_{k})\mathcal{\ }$satisfying the
		Lipschitz condition and $f(\cdot,\cdot,0)\in L_{\mathbb{F},w}^{2}		(0,T;\mathfrak{L}_{k}(H_{1}\times H_{2}\times\cdots\times H_{k})),$ and
		stochastic evolution operator $L_{j}(t,s)$\ satisfying $\mathbb{E}[\left\Vert
		L_{j}(t,s)u\right\Vert _{H}^{2k}|\mathcal{F}_{t}]\leq\Lambda\Vert u\Vert
		_{H}^{2k}$\ for some $\Lambda\geq0,$ for $j=1,2,\cdots,k;$ other measurability
		assumptions are imposed similarly as in $(H1)$ and $(H2)$ (with some possible
		direct modifications).
		
		For an $\eta\in L_{w}^{2}(\mathcal{F}_{s},\mathfrak{L}_{k}(H_{1}\times
		H_{2}\times\cdots\times H_{k})),$ we define the mapping
		\[
		\eta(L_{1}(t,s)\cdot,L_{2}(t,s)\cdot,\cdots,L_{k}(t,s)\cdot)\in\mathfrak{L}		_{k}(H_{1}\times H_{2}\times\cdots\times H_{k};L^{1}(\mathcal{F}_{s}))
		\]
		by
		\[
		(u_{1},u_{2},\cdots,u_{k})\mapsto\eta(L_{1}(t,s)u_{1},L_{2}(t,s)u_{2}		,\cdots,L_{k}(t,s)u_{k}),\quad\forall(u_{1},u_{2},\cdots,u_{k})\in
		H_{1}\times H_{2}\times\cdots\times H_{k}.
		\]
		For a mapping $g\in\mathfrak{L}_{k}(H_{1}\times H_{2}\times\cdots\times
		H_{k};L_{\mathbb{F}}^{1}(t,T)),$ we define $\int_{t}^{T}g(s)ds\in
		\mathfrak{L}_{k}(H_{1}\times H_{2}\times\cdots\times H_{k};L^{1}		(\mathcal{F}_{T}))$ by
		\[
		(\int_{t}^{T}g(s)ds)(u_{1},u_{2},\cdots,u_{k}):=\int_{t}^{T}g(s)(u_{1}		,u_{2},\cdots,u_{k})ds\ \ P\text{-a.s},\quad\forall(u_{1},u_{2},\cdots
		,u_{k})\in H_{1}\times H_{2}\times\cdots\times H_{k}.
		\]
		Then for any $h\in L_{\mathbb{F},w}^{2}(t,T;\mathfrak{L}_{k}(H_{1}\times
		H_{2}\times\cdots\times H_{k})),$
		\[
		h(s)(L_{1}(t,s)\cdot,L_{2}(t,s)\cdot,\cdots,L_{k}(t,s)\cdot)\in\mathfrak{L}		_{k}(H_{1}\times H_{2}\times\cdots\times H_{k};L_{\mathbb{F}}^{1}(t,T)).
		\]
		Therefore,
		\begin{align*}
			&  \xi(L_{1}(t,T)\cdot,L_{2}(t,T)\cdot,\cdots,L_{k}(t,T)\cdot)+\int_{t}			^{T}f(s,P(s))(L_{1}(t,T)\cdot,L_{2}(t,T)\cdot,\cdots,L_{k}(t,T)\cdot)ds\\
			&  \ \ \ \ \in\mathfrak{L}_{k}(H_{1}\times H_{2}\times\cdots\times
			H_{k};L^{1}(\mathcal{F}_{T})).
		\end{align*}
		We consider the $\mathfrak{L}_{k}(H_{1}\times H_{2}\times\cdots\times H_{k}		)$-valued conditionally expected BSIE
		\begin{align*}
			P(t) &  =\mathbb{E}[\xi(L_{1}(t,T)\cdot,L_{2}(t,T)\cdot,\cdots,L_{k}			(t,T)\cdot)+\int_{t}^{T}f(s,P(s))(L_{1}(t,T)\cdot,L_{2}(t,T)\cdot,\cdots
			,L_{k}(t,T)\cdot)ds|\mathcal{F}_{t}],\\
			&  \text{ }t\in\lbrack0,T].
		\end{align*}
		By a solution of it, we mean a process $P\in L_{\mathbb{F},w}^{2}		(0,T;\mathfrak{L}_{k}(H_{1}\times H_{2}\times\cdots\times H_{k}))$ satisfying:
		for $t\in\lbrack0,T],$ it hold $P$-a.s. that in $\mathfrak{L}_{k}(H_{1}\times
		H_{2}\times\cdots\times H_{k})$
		\[
		P(t)=\mathbb{E}[\xi(L_{1}(t,T)\cdot,L_{2}(t,T)\cdot,\cdots,L_{k}		(t,T)\cdot)+\int_{t}^{T}f(s,P(s))(L_{1}(t,T)\cdot,L_{2}(t,T)\cdot,\cdots
		,L_{k}(t,T)\cdot)ds|\mathcal{F}_{t}].
		\]
		This equation can be solved by firstly defining and constructing the
		$k$-linear operator $\mathfrak{L}_{k}(H_{1}\times H_{2}\times\cdots\times
		H_{k})$-valued conditional expectations, and then making use a contraction
		argument, similarly as in the bilinear situation (Whereas in the multilinear
		situation, it seems awkward to introduce the formal adjoint operators for $L,$
		which will make the notations complicated$)$.
		
		This is a multilinear operator-valued backward stochastic evolution equations.
}
\end{remark}
\begin{remark}	\upshape{If we strength the growth assumption for $L$ in (H1) to: for each $t$, $L(t,s)u$ is continuous in $s$ and
		\[
		\mathbb{E}[\sup_{t\leq s\leq T}\left\Vert L(t,s)u\right\Vert _{H}^{4}|\mathcal{F}_{t}]\leq
		\Lambda\Vert u\Vert_{H}^{4},\quad P\text{-a.s}.;
		\]
		or more generally,  for each $t$,
		\[
		\mathbb{E}[\underset{0\leq t\leq T}{ess\sup
		}
		\left\Vert L(t,s)u\right\Vert _{H}^{4}|\mathcal{F}_{t}]\leq
		\Lambda\Vert u\Vert_{H}^{4},\quad P\text{-a.s}.;
		\]
		By a standard modifications of the proofs, we can weaken the assumption $f(\cdot
		,\cdot,0)\in L_{\mathbb{F},w}^{2}(0,T;\mathfrak{L}(H))$ of $f$ to   $f(\cdot
		,\cdot,0)\in L_{\mathbb{F},w}^{1,2}(0,T;\mathfrak{L}(H))$ in (H2) for the well-posedness result and estimates of the solutions, as well as other results (For Proposition \ref{Myth2-7}, we also need a similar modifiction for the condition $(H3)$) obtained for the BSIEs in this paper. Here, $L_{\mathbb{F},w}^{1,2}(0,T;\mathfrak{L}(H))$ is the space of $\mathfrak{L}(H)$-valued
		weakly progressively measurable processes $F(\cdot)$ with norm  $\Vert F\Vert_{L_{\mathbb{F},w}^{1,2}%
			(0,T;\mathfrak{L}(H))}=\{\mathbb{\mathbb{E}}[(\int_{0}^{T}\Vert F(t)\Vert
		_{\mathfrak{L}(H)}{d}t)^{{2}}]\}^{\frac{1}{2}}$.
		
		We illustrate this change for the first condition in the proof of Lemma \ref{Myth2-5}:
		\begin{align*}
			&  |\mathbb{E}[\int_{t}^{T}\langle f(s,p(s))L(t,s)u,L(t,s)v\rangle
			ds|\mathcal{F}_{t}]|\\
			&  \leq|\mathbb{E}[\sup_{t\leq s\leq T}\Vert L(t,s)u\Vert_{H}\sup_{t\leq s\leq
				T}\Vert L(t,s)v\Vert_{H}\int_{t}^{T}\Vert f(s,p(s))\Vert_{\mathfrak{L}%
				(H)}ds|\mathcal{F}_{t}]|\\
			&  \leq(\mathbb{E}[\sup_{t\leq s\leq T}\Vert L(t,s)u\Vert_{H}^{4}%
			|\mathcal{F}_{t}])^{\frac{1}{4}}(\mathbb{E}[(\int_{t}^{T}\Vert f(s,p(s))\Vert
			_{\mathfrak{L}(H)}ds)^{2}|\mathcal{F}_{t}])^{\frac{1}{2}}(\mathbb{E}%
			[\sup_{t\leq s\leq T}\Vert L(t,s)v\Vert_{H}^{4}|\mathcal{F}_{t}])^{\frac{1}%
				{4}}\\
			&  \leq C\Vert u\Vert_{H}\Vert v\Vert_{H}\{(\mathbb{E}[(\int_{t}^{T}\Vert
			f(s,0)\Vert_{\mathfrak{L}(H)}ds)^{2}+\int_{t}^{T}\Vert p(s)\Vert
			_{\mathfrak{L}(H)}^{2}ds|\mathcal{F}_{t}])^{\frac{1}{2}}\},\quad P\text{-a.s.}%
		\end{align*}
		For reader's convenience, we  present the improved result for the well-posedness of the BSIEs: Under one of the above new conditions, there
		exists a unique solution $P$ to BSIE (\ref{Myeq2-1}).
		Moreover, for each $t\in\lbrack0,T],$
		\begin{equation}\label{Myeq5-1}
			\Vert P(t)\Vert_{\mathfrak{L}(H)}^{2}\leq C\mathbb{E}[\Vert\xi\Vert
			_{\mathfrak{L}(H)}^{2}+(\int_{t}^{T}\Vert f(s,0)\Vert_{\mathfrak{L}(H)}%
			ds)^2|\mathcal{F}_{t}],\quad P\text{-a.s}., 
		\end{equation}
		for some constant $C$ depending on $\Lambda$ and $\lambda$.
		
		Obviously, this change also holds for Remark \ref{rem1-10}.
		
	}
\end{remark}
\subsection{It\^{o}'s formula}

A typical example and main prototype of the stochastic evolution operator $L$
is the formal solution of forward operator-valued SDEs, which can be
rigorously defined as the solution map of forward vector-valued SEEs. In this
section, we shall derive an It\^{o}'s formula for the product of BSIEs with
two forward SEEs when $L$ takes this concrete form. It is needed in the
derivation of the maximum principle.

\subsubsection{Evolution operators associated to forward SEEs}

Let $V$ be a separable Hilbert space densely embedded in $H$. Denote $V^{\ast
}:=\mathfrak{L}(V;\mathbb{R}),$ then $V\subset H\subset V^{\ast}$ form a
Gelfand triple. We denote by $\langle\cdot,\cdot\rangle_{\ast}$ the duality
between $V^{\ast}$ and $V$.

Let $w:=\{w(t)\}_{t\geq0}$ be a one-dimensional standard Brownian motion with
respect to $\mathbb{F}$. Consider the following linear homogeneous SEE on $[t,T]$:%
\begin{equation}%
\begin{cases}
{d}u^{t,u_{0}}(s) & =A(s)u^{t,u_{0}}(s){d}s+B(s)u^{t,u_{0}}(s){d}w(s),\quad
s\in\lbrack t,T],\\
u^{t,u_{0}}(t) & =u_{0},
\end{cases}
\label{Myeq2-17}%
\end{equation}
where $u_{0}\in L^{2}(\mathcal{F}_{t},H)$ and $(A,B):[0,T]\times
\Omega\rightarrow\mathfrak{L}(V;V^{\ast}\times H).$

\begin{remark}
	\upshape{We only write the one-dimensional Brownian motion case for simplicity
	of presentation. With direct modifications, the results throughout this paper still hold for the more general case that $w$ is a Hilbert space $K$-valued cylindrical $Q$-Brownian motion (including multi-dimensional Brownian motion, finite-trace $Q$-Brownian motion, cylinderical Brownian motion as special cases) and  the integrands $f$ takes
	valued in the Hilbert-Schmidt space $\mathcal{L}_{2}(Q^{\frac{1}{2}
	}(K),H)$; see \cite{lsw23} and \cite{LZ21} for more discussions on this direction.}
\end{remark}

We make the following assumption.
\begin{description}
\item[$(H4)$] For each $u\in V,$ $A(t,\omega)u$ and $B(t,\omega)u$ are
progressively measurable and satisfying: There exist some constants $\delta>0$
and $K\geq0$ such that the following two assertions hold: for each $t,\omega$
and $u\in V$,

\begin{description}
\item[\rm{{{(i)}}}] coercivity condition:
\[
2\langle A(t,\omega)u,u\rangle_{\ast}+\Vert B(t,\omega)u\Vert_{H}^{2}%
\leq-\delta\Vert u\Vert_{V}^{2}+K\Vert u\Vert_{H}^{2}\quad\text{and}\quad\Vert
A(t,\omega)u\Vert_{V^{\ast}}\leq K\Vert u\Vert_{V};
\]

\item[{\rm{{(ii)}}}] quasi-skew-symmetry condition:
\[
|\langle B(t,\omega)u,u\rangle|\leq K\Vert u\Vert_{H}^{2}.
\]

\end{description}
\end{description}

From \cite{KR81}, Equation (\ref{Myeq2-17}) has a unique solution $u^{t,u_{0}%
}(\cdot)\in L_{\mathbb{F}}^{2}(t,T;V)\cap S_{\mathbb{F}}^{2}(t,T;H),$ where
$S_{\mathbb{F}}^{2}(t,T;H)$ is the space of adapted $H$-valued processes $y$
with continuous paths such that $\mathbb{\mathbb{E}}[\sup_{t\leq s\leq T}\Vert
y(s)\Vert_{H}^{2}]<\infty.$ Through this solution, we define a stochastic
evolution operator $L_{A,B}$ as follows:
\begin{equation}
L_{A,B}(t,s)(u_{0}):=u^{t,u_{0}}(s)\in L^{2}(\mathcal{F}_{s},H),\quad
\text{for}\ t\leq s\leq T\ \text{and}\ u_{0}\in L^{2}(\mathcal{F}%
_{t},H)\text{.} \label{Myeq1-5}%
\end{equation}
From the basic estimates for SEEs, it satisfies the
assumptions $(H1)$ and $(H3)$. In fact, in general, if $y$ is the solution to the SEE
\[%
\begin{cases}
	{d}y(s) & =[A(s)y(s)+a(s)]{d}s+[B(s)y(s)+b(s)]{d}w(s),\quad s\in\lbrack
	t,T],\\
	y(t) & =y_{0},
\end{cases}
\]
for $a,b\in L_{\mathbb{F}}^{1,2\alpha}(t,T;H)\times
L_{\mathbb{F}}^{2,2\alpha}(t,T;H)$ and $y_{0}\in L^{2\alpha}(\mathcal{F}%
_{t},H)$, with $\alpha\geq1$ and $L_{\mathbb{F}}^{1,2\alpha}(0,T;H)$ being the
space of $H$-valued progressively measurable processes $y(\cdot)$ with  norm
$\Vert y\Vert_{L_{\mathbb{F}}^{1,2\alpha}(0,T;H)}=\{\mathbb{\mathbb{E}}%
[(\int_{0}^{T}\Vert y(t)\Vert_{H}{d}t)^{2\alpha}]\}^{\frac{1}{2\alpha}},$ then
there exists a constant $C>0$ depending on $\delta$, $K$ and $\alpha$ (see
\cite[Lemma 3.1]{DM13}) such that
\begin{equation}
	\mathbb{E}[\sup_{s\in\lbrack t,T]}\left\Vert y(s)\right\Vert _{H}^{2\alpha
	}]\leq C\mathbb{E}[\Vert y_{0}\Vert_{H}^{2\alpha}+(\int_{t}^{T}\Vert
	a(s)\Vert_{H}ds)^{2\alpha}+(\int_{t}^{T}\Vert b(s)\Vert_{H}^{2}ds)^{\alpha
	}].\label{SEE-estimate}%
\end{equation}
This implies
\[
\mathbb{E}[\sup_{s\in\lbrack t,T]}\left\Vert y(s)\right\Vert _{H}^{2\alpha
}|\mathcal{F}_{t}]\leq C\{\Vert y_{0}\Vert_{H}^{2\alpha}+\mathbb{E}[(\int%
_{t}^{T}\Vert a(s)\Vert_{H}ds)^{2\alpha}+(\int_{t}^{T}\Vert b(s)\Vert_{H}%
^{2}ds)^{\alpha}|\mathcal{F}_{t}]\},
\] 
by noting that, with  $y$ denoted by $y^{t,y_{0};a,b},$  for any $D\in\mathcal{F}_{t},$
\begin{align*}
	\mathbb{E}[I_{D}\cdot\sup_{s\in\lbrack t,T]}\left\Vert y^{t,y_{0};a,b}%
	(s)\right\Vert _{H}^{2\alpha}] &  =\mathbb{E}[\sup_{s\in\lbrack t,T]}%
	\left\Vert y^{t,I_{D}\cdot y_{0};I_{D}\cdot a,I_{D}\cdot b}(s)\right\Vert _{H}^{2\alpha}]\\
	&  \leq C\mathbb{E}[\Vert I_{D}\cdot y_{0}\Vert_{H}^{2\alpha}+(\int_{t}^{T}\Vert
	I_{D}\cdot a(s)\Vert_{H}ds)^{2\alpha}+(\int_{t}^{T}\Vert I_{D}\cdot b(s)\Vert_{H}%
	^{2}ds)^{\alpha}]\\
	&  =C\mathbb{E}[I_{D}\cdot(\Vert y_{0}\Vert_{H}^{2\alpha}+\mathbb{E}[(\int_{t}%
	^{T}\Vert a(s)\Vert_{H}ds)^{2\alpha}+(\int_{t}^{T}\Vert b(s)\Vert_{H}%
	^{2}ds)^{\alpha}|\mathcal{F}_{t}])].
\end{align*}
Furthermore, the continuity  in $(H3)$ for  $L_{A,B}$ follows from the continuity property of solutions for SEEs.

\begin{remark}
\upshape{
The operator $L_{A,B}$ can be regarded as the formal solution of the following
$\mathfrak{L}(H)$-valued SDEs
\begin{equation}\begin{cases}
{d}L_{A,B}(t,s) & =A(s)L_{A,B}(t,s){d}t+B(s)L_{A,B}(t,s){d}w(s),\quad
s\in [t,T],\\
L_{A,B}(t,t) & =I_{d}.
\end{cases}
\label{Myeq3-7}\end{equation}
When $H$ is finite dimensional (i.e., $H=\mathbb{R}^{n}$ for some integer $n\geq 1$, then $\mathfrak{L}(H)=\mathfrak{L}(\mathbb{R}^{n})=\mathbb{R}^{n\times n}$), it is indeed the
classical (matrix-valued) solution of  (\ref{Myeq3-7}). In the
infinite-dimensional situation, such an equation is far from being well understood (it is not known that it admits an $\mathfrak{L}(H)$-valued solution).}
\end{remark}

Now, in virtue of Theorem \ref{Mainthm-1}, the $\mathfrak{L}(H)$-valued BSIE
\begin{equation}
P(t)=\mathbb{E}[L_{A,B}^{\ast}(t,T)\xi L_{A,B}(t,T)+\int_{t}^{T}L_{A,B}^{\ast
}(t,s)f(s,P(s))L_{A,B}(t,s)ds|\mathcal{F}_{t}],\quad t\in\lbrack0,T],
\label{Myeq3-2}%
\end{equation}
has a unique solution $P\in L_{\mathbb{F},w}^{2}(0,T;\mathfrak{L}(H))$.

In the following, we shall always assume that the filtration $\mathbb{F}%
=\{\mathcal{F}_{t}\}_{0\leq t\leq T}$ is the augmented natural filtration of
Brownian motion $\{w(t)\}_{t\geq0}$.

\begin{remark}
\label{Rm2-2}{\upshape Let $H$ be finite dimensional. }

{\upshape
\textrm{{ (i)}} BSIE (\ref{Myeq3-2}) is equivalent to the following
matrix-valued BSDE
\begin{equation}%
\begin{split}
P(t)  &  =\xi+\int_{t}^{T}[A^{\ast}(s)P(s)+P(s)A(s)+B^{\ast}%
(s)P(s)B(s)+B^{\ast}(s)Q(s)+Q(s)B(s)\\
&  \ \ \ +f(s,P(s))]ds-\int_{t}^{T}Q(s)dw(s).
\end{split}
\label{Myeq1-10}%
\end{equation}
In fact, recall that in the matrix case,  $L_{A,B}(t,s)$ is the solution of matrix-valued SDE
(\ref{Myeq3-7}) and $L^*_{A,B}(t,s)$ is its transpose which satisfies
\[\begin{cases}
	{d}L^*_{A,B}(t,s) & =L^*_{A,B}(t,s)A^*(s){d}s+L^*_{A,B}(t,s)B^*(s){d}w(s),\quad
	s\in [t,T],\\
	L^*_{A,B}(t,t) & =I_{d}.
\end{cases}
\] Then using It\^{o}'s formula to $L_{A,B}^{\ast}(t,s)P(s)L_{A,B}(t,s)$
 on $[t,T]$, we get
\begin{align*}
P(t)  &  =L_{A,B}^{\ast}(t,T)\xi L_{A,B}(t,T)+\int_{t}^{T}L_{A,B}^{\ast
}(t,s)f(s,P(s))L_{A,B}(t,s)ds\\
&  \ \ \ -\int_{t}^{T}L_{A,B}^{\ast}(t,s)(P(s)B(s)+Q(s)+B^{\ast}%
(s)P(s))L_{A,B}(t,s)dw(s).
\end{align*}
Taking conditional expectation on both sides, we obtain
\begin{equation}
P(t)=\mathbb{E}[L_{A,B}^{\ast}(t,T)\xi L_{A,B}(t,T)+\int_{t}^{T}L_{A,B}^{\ast
}(t,s)f(s,P(s))L_{A,B}(t,s)ds|\mathcal{F}_{t}]. \label{Myeq3-09}%
\end{equation}
}

{\upshape Naturally, BSDE is preferred in the characterization of the adjoint
process. Unfortunately, in an infinite-dimensional space without separability,
the stochastic integral and unbounded operators in BSDE (\ref{Myeq1-10}) find
difficult to be well defined. This is why we appeal to a conditionally
expected BSIE to characterize the adjoint process.}

\upshape{
\textrm{{ (ii)}} We can also give the integral equation of the following
matrix-valued BSDEs in a more general form, which will be used in the
recursive optimal control problem latter. Consider
\begin{align*}
P(t)  &  =\xi+\int_{t}^{T}[A^{\ast}(s)P(s)+P(s)A(s)+B^{\ast}(s)P(s)B(s)+B^{\ast}(s)Q(s)+Q(s)B(s)+\beta(s)Q(s)\\
&  \ \ \ +f(s,P(s))]ds-\int_{t}^{T}Q(s)dw({s}),
\end{align*}
where $\beta\in L_{\mathbb{F}}^{\infty}(0,T)$. We can write it into
\begin{align*}
P(t)  &  =\xi+\int_{t}^{T}[(A(s)-\frac{\beta(s)}{2}B(s)-\frac{\beta^{2}(s)}{8}I_{d})^{\ast}P(s)+P(s)(A(s)-\frac{\beta(s)}{2}B(s)-\frac{\beta^{2}(s)}{8}I_{d})\\
&  \ \ \ +(B(s)+\frac{\beta(s)}{2}I_{d})^{\ast}P(s)(B(s)+\frac{\beta(s)}{2}I_{d})+(B(s)+\frac{\beta(s)}{2}I_{d})^{\ast}Q(s)+Q(s)(B(s)\\
&  \ \ \ +\frac{\beta(s)}{2}I_{d})+f(s,P(s))]ds-\int_{t}^{T}Q(s)dw({s}).
\end{align*}
Then from  (i), we have
\[
P(t)=\mathbb{E}[\tilde{L}^{\ast}(t,T)\xi\tilde{L}(t,T)+\int_{t}^{T}\tilde
{L}^{\ast}(t,s)f(s,P(s))\tilde{L}^{\ast}(t,s)ds|\mathcal{F}_{t}]
\]
with\[
\tilde{L}(t,s):=L_{\tilde{A},\tilde{B}}(t,s),\quad\text{for}\ \ \tilde
{A}(s):=A(s)-\frac{\beta(s)}{2}B(s)-\frac{\beta^{2}(s)}{8}I_{d}\ \ \text{and}\ \ \tilde{B}(s):=B(s)+\frac{\beta(s)}{2}I_{d}.
\]
}
\end{remark}

\begin{remark}
\upshape{
	The above (\ref{Myeq2-17}) is the solution of SEEs under the variational solution framework. Also as  examples, in the same way, the solution of vector-valued SEEs under other framework (or conditions, settings) may also generates  such kind of stochastic evolution operator $L$  that satisfies $(H1)$, and then the corresponding conditionally expected BSIE is well-posed. 
	
	We give a detailed mathematical description on mild solution (semigroup solution) case.
	Consider the SEEs
\begin{equation*}%
	\begin{cases}
		{d}u^{t,u_{0}}(s) & =Au^{t,u_{0}}(s){d}s+\bar{A}(s)u^{t,u_{0}}(s){d}s+\bar{B}(s)u^{t,u_{0}}(s){d}w(s),\quad
		s\in\lbrack t,T],\\
		u^{t,u_{0}}(t) & =u_{0},
	\end{cases}
\end{equation*}
where $u_{0}\in L^{2}(\mathcal{F}_{t},H)$, the operator $A:D(A)\subset H\rightarrow H$ is the infinitesimal
generator of a $C_{0}$-semigroup $\{\mathrm{e}^{tA}\in \mathfrak{L}(H);t\geq0\}$, and $\bar{A},\bar{B}:[0,T]\times
\Omega\rightarrow\mathfrak{L}(H)$ are bounded and satisfying: for each $u\in H$, $\bar{A}u,\bar{B}u$ are progressively measurable.
This SEE have a unique solution $u^{t,u_{0}}(\cdot)\in S_{\mathbb{F}}^{2}(t,T;H)$  (see  \cite{DZ92}).  Using the same approach as in (\ref{Myeq1-5}), it also defines
a stochastic evolution operator $L$ satisfying the assumption $(H1)$. Then the corresponding conditionally expected BSIE is also well-posed.
In this kind of concrete mild solution situation, in \cite{GT05} the authors also describe a variation of constant formula characterization for their \textit{generalized solutions}. Compared with that, our BSIE is an operator-valued equation (i.e., the equation itself is operator-valued) and has a fully nonlinear generator for $P$.
	}
\end{remark}

\subsubsection{It\^{o}'s formula in a weak formulation}

Now we derive an It\^{o}'s formula by an approximation argument for the
product of the operator-valued BSIE
\begin{equation}
P(t)=\mathbb{E}[\tilde{L}^{\ast}(t,T)\xi\tilde{L}(t,T)+\int_{t}^{T}\tilde
{L}^{\ast}(t,s)f(s,P(s))\tilde{L}(t,s)ds|\mathcal{F}_{t}],\quad t\in
\lbrack0,T], \label{Myeq4-11}%
\end{equation}
and two forward SEEs in the form of
\begin{equation}%
\begin{cases}
{d}x(t) & =A(t)x(t){d}t+[B(t)x(t)+\zeta(t)I_{E_{\rho}}(t)]{d}w(t),\quad
t\in[0,T],\\
x(0) & =0,
\end{cases}
\label{Myeq3-9}%
\end{equation}
where, for some $\beta\in L_{\mathbb{F} }^{\infty}(0,T)$,
\[
\tilde{L}(t,s):=L_{\tilde{A},\tilde{B}}(t,s)\quad\text{with}\ \ \tilde
{A}=A(s)+\frac{\beta(s)}{2}B(s)-\frac{\beta^{2}(s)}{8}I_{d}\ \ \text{and}%
\ \ \tilde{B}=B(s)+\frac{\beta(s)}{2}I_{d},
\]
and $\zeta$ is an $H$-valued process, $E_{\rho}=[t_{0},t_{0}+\rho)$ for some
$t_{0}\in\lbrack0,T)$ and $\rho\in\lbrack0,T-t_{0}].$

Then we have the following It\^{o}'s formula. The proof is lengthy and
technical, and is thus put in the appendix.

\begin{theorem}
\label{Myth3-7} Let Assumptions $(H2)$ and $(H4)$ be satisfied and for some
$\alpha>1$,
\begin{equation}
\label{Myeq3-13}(\xi,f(\cdot,\cdot,0),\zeta)\in L_{w}^{2\alpha}(\mathcal{F}%
_{T},\mathfrak{L}(H))\times L_{\mathbb{F},w}^{2,2\alpha}(0,T;\mathfrak{L}%
(H))\times L_{\mathbb{F}}^{4\alpha}(0,T;H).
\end{equation}
Then
\begin{equation}%
\begin{split}
\langle P(t)x(t),x(t)\rangle+\sigma(t) =  &  \langle\xi x(T),x(T)\rangle
+\int_{t}^{T}[\langle f(s,P(s))x(s),x(s)\rangle+ \beta(s)\mathcal{Z}(s)\\
&  -\langle P(s)\zeta(s),\zeta(s)\rangle I_{E_{\rho}}(s)]ds-\int_{t}%
^{T}\mathcal{Z}(s)dw({s}),\quad t\in\lbrack0,T],
\end{split}
\label{Myeq3-5}%
\end{equation}
for a unique couple of processes $(\sigma,\mathcal{Z})\in L_{\mathbb{F}%
}^{\alpha}(0,T)\times L_{\mathbb{F}}^{2,\alpha}(0,T)$ satisfying%
\begin{align}
&  \sup_{t\in\lbrack0,T]}\mathbb{E}[|\sigma(t)|^{\alpha}]=o(\rho^{\alpha
}),\label{Myeq2-16}\\
&  \mathbb{E}[(\int_{0}^{T}|\mathcal{Z}(t)|^{2}dt)^{\frac{\alpha}{2}}%
]=O(\rho^{\alpha}). \label{Myeq2-20}%
\end{align}

\end{theorem}

\begin{remark}
\label{Re2-1}
\upshape{ When solving the stochastic optimal control problem for SEEs in
	the conventional case (see Remark \ref{Myrm4-1}), only the form of Theorem
	\ref{Myth3-7} when $\beta\equiv0$ and $f$ is independent of $p$ is needed, and
	it corresponds to \cite[Equality (5.11)]{DM13}, \cite[Equality (5.17)]{FHT13},
	and \cite[Equality (9.61) (plus estimates (9.62), (9.63) and (9.82))]{LZ14}.
	
}
\end{remark}

\begin{remark}
\upshape{ 	To understand the above It\^{o}'s formula, let us look at how this is derived
	in the finite dimensional case. The differential form (taking $A_{1}=A+\beta
	B,$ $B_{1}=B$ in Remark \ref{Rm2-2} (ii)) of BSIE (\ref{Myeq4-11}) is	\begin{align*}
		P(t)  &  =\xi+\int_{t}^{T}[A^{\ast}(s)P(s)+P(s)A(s)+\beta(s)(B^{\ast
		}(s)P(s)+P(s)B(s)))+f(s,P(s))\\
		&  +B^{\ast}(s)P(s)B(s)+B^{\ast}(s)Q(s)+Q(s)B(s)+\beta(s)Q(s)]ds-\int_{t}		^{T}Q(s)dw({s}).
	\end{align*}
	We apply It\^{o}'s formula to $\langle P(t)x(t),x(t)\rangle$ and obtain	\begin{align*}
		\langle{P}(t)x(t),x(t)\rangle &  =\langle\xi x(T),x(T)\rangle+\int_{t}		^{T}\{\beta(s)\langle(B(s)^{\ast}P(s)+P(s)B(s)+Q(s))x(s),x(s)\rangle\\
		&  +\langle f(s,P(s))x(s),x(s)\rangle-\langle{P}(s)\zeta(s),\zeta(s)\rangle
		I_{E_{\rho}}(s)\\
		&  -[\langle Q(s)x(s),\zeta(s)\rangle+\langle Q(s)\zeta(s),x(s)\rangle+\langle
		P(s)B(s)x(s),\zeta(s)\rangle\\
		&  +\langle B^{\ast}(s)P(s)\zeta(s),x(s)\rangle]I_{E_{\rho}}(s)\}ds-\int		_{t}^{T}[\langle(B^{\ast}(s)P(s)+P(s)B(s)+Q(s))x(s),x(s)\rangle\\
		&  +\langle P(s)\zeta(s),x(s)\rangle I_{E_{\rho}}(s)+\langle P(s)x(s),\zeta
		(s)\rangle I_{E_{\rho}}(s)]dw({s}).
	\end{align*}
	Since the depiction of $Q$ is unavailable$,$ we try to merge the martingale
	terms and the small terms together and determine them via the solution of
	BSDEs, as follows. We take
	\[
	\mathcal{Z}_{1}(s)=\langle(B^{\ast}(s)P(s)+P(s)B(s)+Q(s))x(s),x(s)\rangle
	+[\langle P(s)\zeta(s),x(s)\rangle+\langle P(s)x(s),\zeta(s)\rangle
	]I_{E_{\rho}}(s)
	\]
	and
	\begin{align*}
		k(s)  &  =[\langle Q(s)x(s),\zeta(s)\rangle+\langle Q(s)x(s),\zeta
		(s)\rangle+\langle P(s)B(s)x(s),\zeta(s)\rangle+\langle B^{\ast}		(s)P(s)\zeta(s),x(s)\rangle\\
		&  +\langle P(s)\zeta(s),x(s)\rangle+\langle P(s)x(s),\zeta(s)\rangle
		]I_{E_{\rho}}(s).
	\end{align*}
	Then	\begin{equation}		\begin{split}
			\langle P(t)x(t),x(t)\rangle &  =\langle\xi x(T),x(T)\rangle+\int_{t}			^{T}[\langle f(s,P(s))x(s),x(s)\rangle+\beta(s)\mathcal{Z}_{1}(s)\\
			&  -\langle P(s)\zeta(s),\zeta(s)\rangle I_{E_{\rho}}(s)-k(s)]ds-\int_{t}			^{T}\mathcal{Z}_{1}(s)dw({s}).
		\end{split}
		\label{Myeq3-6}	\end{equation}
	Let $(-\sigma,b)$ be the solution of BSDE
	\[
	-\sigma(t)=\int_{t}^{T}[\beta b(s)-k(s)]ds-\int_{t}^{T}b(s)dw({s}),\quad
	t\in\lbrack0,T].
	\]
	and set
	\begin{equation}
		\mathcal{Z}(t):=\mathcal{Z}_{1}(t)-b(t). \label{Myeq3-3-1}	\end{equation}
	Subtracting (\ref{Myeq3-3-1}) from (\ref{Myeq3-6}), we have
	\begin{align*}
		\langle{P}(t)x(t),x(t)\rangle+\sigma(t)  &  =\langle\xi x(T),x(T)\rangle
		+\int_{t}^{T}[\langle f(s,P(s))x(s),x(s)\rangle+\beta(s)\mathcal{Z}(s)\\
		&  -\langle{P}(s)\zeta(s),\zeta(s)\rangle I_{E_{\rho}}(s)]ds-\int_{t}		^{T}\mathcal{Z}(s)dw({s}),
	\end{align*}
	and the corresponding estimates can be obtained from the standard BSDE
	theory$.$
	
	So Theorem \ref{Myth3-7}\ can be regarded as a weak formulation of the
	classical It\^{o}'s formula in the infinite dimensional framework. It is also
	worth noting that the above analysis does not apply to our infinite
	dimensional situation, since we do not have a differential form for
	operator-valued BSDE now.
}
\end{remark}

\section{Stochastic maximum principle for optimally controlled SEEs}

\subsection{Formulation of the problem}

Consider the following controlled SEE:
\begin{equation}%
\begin{cases}
{d}x(t) & =[A(t)x(t)+a(t,x(t),u(t))]{d}t+[B(t)x(t)+b(t,x(t),u(t))]{d}w(t),\\
x(0) & =x_{0},
\end{cases}
\label{SEE1-1}%
\end{equation}
where $x_{0}\in H$,
\[
(A,B):[0,T]\times\Omega\rightarrow\mathfrak{L}(V;V^{\ast}\times H)
\]
are linear unbounded operators satisfying the coercivity and
quasi-skew-symmetry condition $(H4)$ and
\[
(a,b):[0,T]\times\Omega\times H\times U\rightarrow H\times H
\]
are nonlinear functions. Define the cost functional $J(\cdot)$ as
\[
J(u(\cdot)):=y(0),
\]
where $y$ is the recursive utility subject to a BSDE:%
\begin{equation}
y(t)=h(x(T))+\int_{t}^{T}k(s,x(s),y(s),z(s),u(s))ds-\int_{t}^{T}z(s)dw({s}).
\label{BSDE1-2}%
\end{equation}
Here,
\[
k:[0,T]\times\Omega\times H\times\mathbb{R}\times\mathbb{R}\times
U\rightarrow\mathbb{R}\quad\text{and}\quad h:H\times\Omega\rightarrow
\mathbb{R}\text{.}%
\]
The control domain $U$ is a separable metric space with distance ${d}%
(\cdot,\cdot).$ By fixing an element $0$ in $U$, we define the length
$|u|_{U}:={d}(u,0)$. We define the admissible control set
\[
\mathcal{U}[0,T]:=\{u:[0,T]\times\Omega\rightarrow U \ \text{is progressively
measurable and}\ \mathbb{\mathbb{E}}[\int_{0}^{T}|u(t)|_{U}^{\alpha}{d}
t]<\infty,\text{ for each }\alpha\geq1\}.
\]

Our optimal control problem is to find an admissible control $\bar u(\cdot)$
such that the cost functional $J(u(\cdot))$ is minimized at $\bar u(\cdot)$
over the control set $\mathcal{U}[0,T]:$
\[
J(\bar u(\cdot))=\inf_{u(\cdot)\in\mathcal{U}[0,T]}J(u(\cdot)).
\]

We make the following assumption for $a$, $b$, $h$ and $k$.

\begin{description}
\item[$(H5)$] For each $(x,y,z,u),$ $a(\cdot,x,u)$, $b(\cdot,x,u)$,
	$k(\cdot,x,y,z,u)$ are progressively measurable and $h(\cdot,x)$ is
	$\mathcal{F}_{T}$-measurable. For each $(t,\omega,u),$ $a$, $b$, $h$, $k$ are twice continuously
	differentiable with respect to $(x,y,z)$; for each $(t,\omega),$ $a$, $b$, $k$, $a_{x}$, $b_{x}$,
	$Dk$, $a_{xx}$, $b_{xx}$, $D^{2}k$ are continuous in $(x,y,z,u)$, where $Dk$
	and $D^{2}k$ are the gradient and Hessian matrix of $k$ with respect to
	$(x,y,z)$, respectively; $a_{x}$, $b_{x}$, $Dk$, $a_{xx}$, $b_{xx}$, $D^{2}k$,
	$h_{xx}$ are bounded; $a$, $b$ are bounded by $C(1+\Vert x\Vert_{H}+|u|_{U})$
	and $k$ is bounded by $C(1+\Vert x\Vert_{H}+|y|+|z|+|u|_{U})$.
\end{description}

\subsection{Adjoint equations and the maximum principle}

We introduce the following simplified notations: for $\psi=a,b,a_{x}%
,b_{x},a_{xx},b_{xx}$ and $v\in U$, define
\[
\bar{\psi}(t) :=\psi(t,\bar{x}(t),\bar{u}(t)),\quad\delta\psi(t;v)
:=\psi(t,\bar{x}(t),v)-\bar{\psi}(t)
\]
and
\[
\bar{A}:=A+\bar{a}_{x}, \quad\bar{B}:=B+\bar{b}_{x}.
\]

Consider the following first-order $H$-valued adjoint backward stochastic
evolution equation (BSEE for short, and the well-posedness result is referred
to \cite{DM10}):
\begin{equation}
\left\{
\begin{array}
[c]{rl}%
-dp(t)= & \{[\bar{A}^{\ast}(t)+k_{y}(t)+k_{z}(t)\bar{B}^{\ast}(t)]p(t)+[\bar
{B}^{\ast}(t)+k_{z}(t)]q(t)+k_{x}(t)\}{d}t-q(t){d}w(t),\\
p(T)= & h_{x}(\bar{x}(T)),
\end{array}
\right.  \label{adjoint1}%
\end{equation}
and the following second-order $\mathfrak{L}(H)$-valued adjoint BSIE
\begin{equation}
P(t)=\mathbb{E}[\tilde{L}^{\ast}(t,T)h_{xx}(\bar{x}(T))\tilde{L}(t,T)+\int%
_{t}^{T}\tilde{L}^{\ast}(t,s)(k_{y}(s)P(s)+G(s))\tilde{L}(t,s)ds|\mathcal{F}%
_{t}],\quad0\leq t\leq T, \label{adjoint2}%
\end{equation}
where%
\begin{align*}
&  \phi(t):=\phi(t,\bar{x}(t),\bar{y}(t),\bar{z}(t),\bar{u}(t)),\quad
\text{for}\ \phi=k_{x},k_{y,},k_{z},D^{2}k,\\
&  \tilde{L}(t,s):=L_{\tilde{A},\tilde{B}}(t,s),\quad\text{for}\ \ \tilde
{A}(s):=\bar{A}(s)+\frac{k_{z}(s)}{2}\bar{B}(s)-\frac{(k_{z}(s))^{2}}{8}%
I_{d}\ \ \text{and}\ \ \tilde{B}(s):=\bar{B}(s)+\frac{k_{z}(s)}{2}I_{d},\\
&  G(t) :=D^{2}k(t)([I_{d},p(t),\bar{B}^{\ast}(t)p(t)+q(t)],[I_{d}%
,p(t),\bar{B}^{\ast}(t)p(t)+q(t)])+\langle p(t),\bar{a}_{xx}(t)\rangle\\
&  \ \ \ \ \ \ \ \ \ \ +k_{z}(t)\langle p(t),\bar{b}_{xx}(t)\rangle+\langle
q(t),\bar{b}_{xx}(t)\rangle.
\end{align*}

\begin{remark}
\upshape{
Letting the coefficients for the  first- and second-order adjoint equations
wait to be determined and plugging the It\^{o}'s formulas (\ref{w-adjoint1}) and
(\ref{w-adjiont2}) into the derivation of maximum principle, we can use a
similar analysis as in \cite{Hu-17} to derive heuristically the proper
generators for the first- and second-order adjoint equations (\ref{adjoint1})
and (\ref{adjoint2}). We may also give their formulations based on the adjoint
equations in \cite{Hu-17} and the discussion in (ii) of Remark \ref{Rm2-2}.}
\end{remark}

Our maximum principle is stated as follows.

\begin{theorem}
\label{SMP} Let Assumptions $(H4)$-$(H5)$ be satisfied. Assume that $\bar
{x}(\cdot)$ and $(\bar{y}(\cdot),\bar{z}(\cdot))$ are the solutions of SEE
(\ref{SEE1-1}) and BSDE (\ref{BSDE1-2}) corresponding to the optimal control
$\bar{u}(\cdot)$. Denote by processes $(p,q)\in L_{\mathbb{F}}^{2}(0,T;V\times
H)$ and $P\in L_{\mathbb{F},w}^{2}(0,T;\mathfrak{L}(H))$ the solutions of BSEE
(\ref{adjoint1}) and BSIE (\ref{adjoint2}), respectively. Then
\begin{equation}
\label{MP}%
\begin{split}
&  \inf_{v\in U}\{\mathcal{H}(t,\bar{x}(t),\bar{y}(t),\bar{z}%
(t),v,p(t),q(t))-\mathcal{H}(t,\bar{x}(t),\bar{y}(t),\bar{z}(t),\bar
{u}(t),p(t),q(t))\\
&  \ \ \ \ \ \ +\frac{1}{2}\langle P(t)(b(t,\bar{x}(t),v)-b(t,\bar{x}%
(t),\bar{u}(t))),b(t,\bar{x}(t),v)-b(t,\bar{x}(t),\bar{u}(t))\rangle\}=0,\quad
P\text{-a.s. a.e.,}%
\end{split}
\end{equation}
where the Hamiltonian
\begin{align*}
\mathcal{H}(t,x,y,z,v,p,q)  &  :=\langle p,a(t,x,v)\rangle+\langle
q,b(t,x,v)\rangle+k(t,x,y,z+\langle p,b(t,x,v)-b(t,\bar{x}(t),\bar
{u}(t))\rangle,v),\\
&  \quad\ \quad\quad\quad\quad\quad\quad\ \ (t,\omega,x,y,z,v,p,q)\in
\lbrack0,T]\times\Omega\times H\times\mathbb{R}\times\mathbb{R}\times U\times
H\times H.
\end{align*}

\end{theorem}

\begin{remark}
\label{Myrm4-1}
\upshape{When $k$ is independent of $y$ and $z$, Theorem \ref{SMP}
degenerates to the conventional maximum principle without utilities, which was
obtained in \cite{DM13,FHT13,LZ14}.}
\end{remark}

\subsection{Proof of Theorem \ref{SMP}}

\textbf{Step 1: Spike variation and dual analysis for SEEs.} Given any
admissible control $u(\cdot)\in\mathcal{U}[0,T]$ and $t_{0}\in\lbrack0,T)$, we
consider the spike variation perturbation%
\[
u^{\rho}(t):=\left\{
\begin{array}
[c]{ll}%
u(t), & \quad t\in E_{\rho},\\
\bar{u}(t), & \quad t\in\lbrack0,T]\setminus E_{\rho},
\end{array}
\right.
\]
with $E_{\rho}=[t_{0},t_{0}+\rho)$ for $\rho\in\lbrack0,T-t_{0}].$ We denote
\[
\delta\psi(t):=\delta\psi(t;u(t)),\quad\text{for}\ \psi=a,b,a_{x},b_{x}%
,a_{xx},b_{xx}.
\]

Let $(x^{\rho}(\cdot),y^{\rho}(\cdot),z^{\rho}(\cdot))$ solve the system
corresponding to the control $u^{\rho}(\cdot)$. Consider the following
linearized variational systems:
\begin{align*}
x^{1,\rho}(t)  &  =\int_{0}^{t}\bar{A}(s)x^{1,\rho}(s){d}s+\int_{0}^{t}%
[\bar{B}(s)x^{1,\rho}(s)+\delta b(s)I_{E_{\rho}}(s)]{d}w(s)
\end{align*}
and
\begin{align*}
x^{2,\rho}(t)=  &  \int_{0}^{t}[\bar{A}(s)x^{2,\rho}(s)+\frac{1}{2}\bar
{a}_{xx}(s)(x^{1,\rho}(s),x^{1,\rho}(s))+\delta a(s)I_{E_{\rho}}(s)]\,{d}s\\
&  +\int_{0}^{t}[\bar{B}(s)x^{2,\rho}(s)+\frac{1}{2}\bar{b}_{xx}(s)(x^{1,\rho
}(s),x^{1,\rho}(s))+\delta b_{x}(s)x^{1,\rho}(s)I_{E_{\rho}}(s)]\,{d}w(s).
\end{align*}

\begin{proposition}
\label{Myth4-3} Assume that $(H4)$ and $(H5)$ hold. Then for $\alpha\geq1$,
\begin{align*}
&  \mathbb{E}[\sup_{t\in\lbrack0,T]}\left\Vert x^{\rho}(t)-\bar{x}
(t)\right\Vert _{H}^{2\alpha}]=O(\rho^{\alpha}),\\
&  \mathbb{E}[\sup_{t\in\lbrack0,T]}\left\Vert x^{1,\rho}(t)\right\Vert
_{H}^{2\alpha}]=O(\rho^{\alpha}),\\
&  \mathbb{E}[\sup_{t\in\lbrack0,T]}\left\Vert x^{2,\rho}(t)\right\Vert
_{H}^{2\alpha}]=O(\rho^{2\alpha}),\\
&  \mathbb{E}[\sup_{t\in\lbrack0,T]}\left\Vert x^{\rho}(t)-\bar{x}
(t)-x^{1,\rho}(t)-x^{2,\rho}(t)\right\Vert _{H}^{2\alpha}]=o(\rho^{2\alpha}).
\end{align*}

\end{proposition}

\begin{proof}
The proof is quite standard. As an illustration, we give the proof of the
second estimate. By (\ref{SEE-estimate}) and the Lebesgue differentiation
theorem, we have (for a.e. $t_{0})$ that
\begin{align*}
\mathbb{E}[\sup_{t\in\lbrack0,T]}\left\Vert x^{1,\rho}(t)\right\Vert
_{H}^{2\alpha}]  &  \leq C\,\mathbb{E[}(\int_{0}^{T}I_{E_{\rho}}(t)\Vert\delta
b(t)\Vert_{H}^{2}dt)^{\alpha}]\\
&  \leq C\,\mathbb{E[}(\int_{0}^{T}I_{E_{\rho}}(t)(1+|u(t)|_{U}^{2}+|\bar
{u}(t)|_{U}^{2})dt)^{\alpha}]\\
&  \leq\,C\,\rho^{\alpha-1}\mathbb{E[}\int_{E_{\rho}}(1+|u(t)|_{U}^{2\alpha
}+|\bar{u}(t)|_{U}^{2\alpha})dt]\\
&  =O(\rho^{\alpha}).
\end{align*}

\end{proof}

According the assumptions on the coefficients, the adjoint processes $(p,q)$
and $P$ satisfy (see Appendix for the proofs): for any $\beta\geq2$,
\begin{equation}\label{Myeq4-22}
\sup_{t\in\lbrack0,T]}\mathbb{E}[\left\Vert p(t)\right\Vert _{H}^{\beta
}]+\mathbb{E}[(\int_{0}^{T}\left\Vert p(t)\right\Vert _{V}^{2}dt)^{\frac
	{\beta}{2}}]+\mathbb{E}[(\int_{0}^{T}\left\Vert q(t)\right\Vert _{H}%
^{2}dt)^{\frac{\beta}{2}}]<\infty\ \ \text{and}\ \ \sup_{t\in\lbrack
	0,T]}\mathbb{E}[\Vert P(t)\Vert_{\mathfrak{L}(H)}^{\beta}]<\infty.
\end{equation}

We have the following It\^{o}'s formula for the first-order adjoint equation
(see \cite{KR81}):
\begin{equation}
\langle p(t),x^{1,\rho}(t)+x^{2,\rho}(t)\rangle=\langle h_{x}(\bar
{x}(T)),x^{1,\rho}(T)+x^{2,\rho}(T)\rangle+\int_{t}^{T}J_{1}(s)ds-\int_{t}%
^{T}J_{2}(s)dw({s}), \label{w-adjoint1}%
\end{equation}
where
\begin{align*}
	&  J_{1}(t):=\langle k_{x}(t)+k_{y}(t)p(t)+k_{z}(t)q(t),x^{1,\rho
	}(t)+x^{2,\rho}(t)\rangle+k_{z}(t)\langle p(t),\bar{B}(t)(x^{1,\rho
	}(t)+x^{2,\rho}(t))\rangle-[\langle p(t),\delta a(t)\rangle\\
	& \ \ \  \ \   +\langle
	q(t),\delta b(t) +\delta b_{x}(t)x^{1,\rho}(t)\rangle]I_{E_{\rho}}(t)-\frac{1}{2}[\langle
	p(t),(\bar{a}_{xx}(t)(x^{1,\rho}(t),x^{1,\rho}(t))\rangle+\langle q(t),\bar
	{b}_{xx}(t)(x^{1,\rho}(t),x^{1,\rho}(t))\rangle],\\
	&  J_{2}(t):=\langle p(t),\bar{B}(t)(x^{1,\rho}(t)+x^{2,\rho}(t))\rangle
	+\langle q(t),x^{1,\rho}(t)+x^{2,\rho}(t)\rangle+\langle p(t),\delta
	b(t)+\delta b_{x}(t)x^{1,\rho}(t)\rangle I_{E_{\rho}}(t)\\
	& \ \ \  \ \     +\frac{1}{2}\langle p(t),\bar{b}_{xx}(t)(x^{1,\rho
	}(t),x^{1,\rho}(t))\rangle.
\end{align*}
By Theorem \ref{Myth3-7}, we also have the following It\^{o}'s formula for the
second-order adjoint equation:
\begin{equation}%
\begin{split}
&  \langle P(t)x^{1,\rho}(t),x^{1,\rho}(t)\rangle+\sigma(t)=\langle
h_{xx}(\bar{x}(T))x^{1,\rho}(T),x^{1,\rho}(T)\rangle+\int_{t}^{T}%
[k_{y}(s)\langle P(s)x^{1,\rho}(s),x^{1,\rho}(s)\rangle\\
&  \ \ \ \ \ \ \ \ \ \ \ \ \ \ \ +k_{z}(s)\mathcal{Z}(s)+\langle
G(s)x^{1,\rho}(s),x^{1,\rho}(s)\rangle-\langle P(s)\delta b(s),\delta
b(s)\rangle I_{E_{\rho}}(s)]{d}s-\int_{t}^{T}\mathcal{Z}(s){d}w(s),
\end{split}
\label{w-adjiont2}%
\end{equation}
for some processes $(\sigma,\mathcal{Z})\in L_{\mathbb{F}}^{2\alpha
}(0,T)\times L_{\mathbb{F}}^{2,2\alpha}(0,T)$ satisfying
\begin{equation}
\sup_{t\in\lbrack0,T]}\mathbb{E}[|\sigma(t)|^{2\alpha}]=o(\rho^{2\alpha})
\ \ \text{and}\ \ \mathbb{E}[(\int_{0}^{T}|\mathcal{Z}(t)|^{2}dt)^{\alpha
}]=O(\rho^{2\alpha}),\quad\text{for any}\ \alpha\geq1. \label{Myeq5-12}%
\end{equation}

Thus,%
\begin{align*}
&  \langle p(t),x^{1,\rho}(t)+x^{2,\rho}(t)\rangle+\frac{1}{2}\langle
P(t)x^{1,\rho}(t),x^{1,\rho}(t)\rangle+\frac{1}{2}\sigma(t)=\langle h_{x}%
(\bar{x}(T)),x^{1,\rho}(T)+x^{2,\rho}(T)\rangle\\
&  \ \ \ \ \ \ \ \ +\frac{1}{2}\langle h_{xx}(\bar{x}(T))x^{1,\rho
}(T),x^{1,\rho}(T)\rangle+\int_{t}^{T}I_{1}(s){d}s-\int_{t}^{T}[I_{2}%
(s)+\langle p(s),\delta b(s)\rangle I_{E_{\rho}}(s)]{d}w(s),
\end{align*}
where
\begin{align*}
	I_{1}(t) &  :=\langle k_{x}(t)+k_{y}(t)p(t)+k_{z}(t)q(t),x^{1,\rho
	}(t)+x^{2,\rho}(t)\rangle+k_{z}(t)\langle p(t),\bar{B}(t)(x^{1,\rho
	}(t)+x^{2,\rho}(t))\rangle+\frac{1}{2}\langle\{k_{y}(t)P(t)\\
	&   \ \ \ \ +D^{2}k(t)([I_{d},p(t),\bar{B}^{\ast}(t)p(t)+q(t)],[I_{d},p(t),\bar
	{B}^{\ast}(t)p(t)+q(t)])+k_{z}(t)\langle p(t),\bar{b}_{xx}(t)\rangle
	\}x^{1,\rho}(t),x^{1,\rho}(t)\rangle\\
	&  \ \ \ \  +\frac{1}{2}k_{z}(t)\mathcal{Z}(t)-[\langle p(t),\delta a(t)\rangle
	+\langle q(t),\delta b(t)+\delta b_{x}(t)x^{1,\rho}(t)\rangle+\frac{1}%
	{2}\langle{P}(t)\delta b(t),\delta b(t)\rangle]I_{E_{\rho}}(t)
\end{align*}
and
\begin{align*}
I_{2}(t)&:=\langle p(t),\bar{B}(t)(x^{1,\rho}(t)+x^{2,\rho}(t))\rangle+\langle
q(t),x^{1,\rho}(t)+x^{2,\rho}(t)\rangle+\langle p(t),\delta b_{x}(t)x^{1,\rho
}(t)\rangle I_{E_{\rho}}(t)\\
&
 \ \ \ \ +\frac{1}{2}\langle p(t),\bar{b}_{xx}(t)(x^{1,\rho
}(t),x^{1,\rho}(t))\rangle+\frac{1}{2}\mathcal{Z}(t).
\end{align*}

\noindent\textbf{Step 2: Variation calculation.} To obtain the maximum
principle, we consider the variation%
\begin{equation}%
\begin{split}
&  \hat{y}^{\rho}(t)-\frac{1}{2}\sigma(t) =h(x^{\rho}(T))-h(\bar
{x}(T))-\langle h_{x}(\bar{x}(T)),x^{1,\rho}(T)+x^{2,\rho}(T)\rangle-\frac
{1}{2}\langle h_{xx}(\bar{x}(T))x^{1,\rho}(T),x^{1,\rho}(T)\rangle\\
&  \ \ \ \ \ \ \ +\int_{t}^{T}\{k(s,x^{\rho}(s),y^{\rho}(s),z^{\rho
}(s),u^{\rho}(s))-k(s,\bar{x}(s),\bar{y}(s),\bar{z}(s),\bar{u}(s))-I_{1}%
(s)\}ds-\int_{t}^{T}\hat{z}^{\rho}(s){d}w(s),
\end{split}
\label{Myeq4-13}%
\end{equation}
where%
\begin{align*}
\hat{y}^{\rho}(t)  &  :=y^{\rho}(t)-\bar{y}(t)-\langle p(t),x^{1,\rho
}(t)+x^{2,\rho}(t)\rangle-\frac{1}{2}\langle P(t)x^{1,\rho}(t),x^{1,\rho
}(t)\rangle,\\
\hat{z}^{\rho}(t)  &  :=z^{\rho}(t)-\bar{z}(t)-I_{2}(t)-\langle p(t),\delta
b(t)\rangle I_{E_{\rho}}(t).
\end{align*}
Motivated from the Taylor's expansion of the above equation, we introduce the
following BSDE:
\begin{equation}%
\begin{split}
&  \hat{y}(t) =\int_{t}^{T}\{k_{y}(s)\hat{y}(s)+k_{z}(s)\hat{z}(s)+[\langle
p(s),\delta a(s)\rangle+\langle q(s),\delta b(s)\rangle+k(s,\bar{x}(s),\bar
{y}(s),\bar{z}(s)\\
&  \ \ +\langle p(s),\delta b(s)\rangle,u(s)) -k(s,\bar{x}(s),\bar{y}%
(s),\bar{z}(s),\bar{u}(s))+\frac{1}{2}\langle P(s)\delta b(s),\delta
b(s)\rangle]I_{E_{\rho}}(s)\}{d}s-\int_{t}^{T}\hat{z}(s){d}w(s).
\end{split}
\label{eq4-10}%
\end{equation}

\begin{proposition}
\label{Myth4-9} Assume that $(H4)$ and $(H5)$ hold. Then for $\alpha\geq1$,
\begin{align}
\sup_{t\in\lbrack0,T]}  &  \mathbb{E}[|\hat{y}(t)|^{2\alpha}]+\mathbb{E}%
[(\int_{0}^{T}|\hat{z}(t)|^{2}dt)^{\alpha}]=o(\rho^{\alpha}),\label{eq-3-0}\\
\sup_{t\in\lbrack0,T]}  &  \mathbb{E}[|\hat{y}^{\rho}(t)|^{2\alpha
}]+\mathbb{E}[(\int_{0}^{T}|\hat{z}^{\rho}(t)|^{2}dt)^{\alpha}]=o(\rho
^{\alpha}),\label{eq-3-1}\\
\sup_{t\in\lbrack0,T]}  &  \mathbb{E}[|\hat{y}^{\rho}(t)-\hat{y}%
(t)|^{2}]+\mathbb{E}[\int_{0}^{T}|\hat{z}^{\rho}(t)-\hat{z}(t)|^{2}%
dt)]=o(\rho^{2}). \label{eq-3-4}%
\end{align}

\end{proposition}

\begin{proof}
We first prove (\ref{eq-3-0}). Denote
\begin{align*}
	I_{3}(t)  &  :=\langle p(t),x^{1,\rho}(t)+x^{2,\rho}(t)\rangle+\frac{1}
	{2}\langle P(t)x^{1,\rho}(t),x^{1,\rho}(t)\rangle,\\
	I_{4}(t)  &  :=\langle p(t),\delta a(t)\rangle+\langle q(t),\delta
	b(t)\rangle+\frac{1}{2}\langle{P}(t)\delta b(t),\delta b(t)\rangle,\\
	I_{5}(t)  &  :=k(t,\bar{x}(t),\bar{y}(t),\bar{z}(t)+\langle p(t),\delta
	b(t)\rangle,u(t))-k(t,\bar{x}(t),\bar{y}(t),\bar{z}(t),\bar{u}(t)).
\end{align*}
We calculate directly that
\begin{equation}
	\label{Myeq4-30}%
	\begin{split}
		\mathbb{E}[(\int_{0}^{T}|\langle q(t),\delta b(t)\rangle|I_{E_{\rho}
		}(t)dt)^{2\alpha}]  &  \leq(\mathbb{E}[(\int_{0}^{T}\Vert q(t)\Vert_{H}
		^{2}I_{E_{\rho}}(t)dt)^{2\alpha}])^{\frac{1}{2}}(\mathbb{E}[(\int_{0}^{T}
		\Vert\delta b(t)\Vert_{H}^{2}I_{E_{\rho}}(t)dt)^{2\alpha}])^{\frac{1}{2}}\\
		&  \leq\rho^{\frac{2\alpha-1}{2}}(\mathbb{E}[(\int_{0}^{T}\Vert q(t)\Vert
		_{H}^{2}I_{E_{\rho}}(t)dt)^{2\alpha}])^{\frac{1}{2}}(\mathbb{E}[\int_{0}
		^{T}\Vert\delta b(t)\Vert_{H}^{4\alpha}I_{E_{\rho}}(t)dt])^{\frac{1}{2}}\\
		&  =o(\rho^{\alpha}).
	\end{split}
\end{equation}
Then from the a priori estimates for BSDEs and the Lebesgue differentiation
theorem, we have
\begin{align*}
	&  \sup_{t\in\lbrack0,T]}\mathbb{E}[|\hat{y}(t)|^{2\alpha}]+\mathbb{E}
	[(\int_{0}^{T}|\hat{z}(t)|^{2}dt)^{\alpha}]\\
	&  \leq C\rho^{2\alpha-1}\mathbb{E}[(\int_{0}^{T}|\langle p(t),\delta
	a(t)\rangle+\frac{1}{2}\langle{P}(t)\delta b(t),\delta b(t)\rangle
	+I_{5}(t)|^{2\alpha}I_{E_{\rho}}(t)dt]+C\mathbb{E}[(\int_{0}^{T}|\langle
	q(t),\delta b(t)\rangle|I_{E_{\rho}}(t)dt)^{2\alpha}]\\
	&  =O(\rho^{2\alpha})+o(\rho^{\alpha})=o(\rho^{\alpha}).
\end{align*}

We first consider (3.13). By the Taylor's expansion,
\begin{equation}%
\begin{split}
&  \hat{y}^{\rho}(t)-\frac{1}{2}\sigma(t)=J_{4}+\int_{t}^{T}\{\tilde{k}%
_{y}(s)(\hat{y}^{\rho}(s)-\frac{1}{2}\sigma(s))+\tilde{k}_{z}(s)\hat{z}^{\rho
}(s)+J_{3}(s)+\frac{1}{2}J_{5}(s)+\frac{1}{2}\tilde{k}_{y}(s)\sigma(s)\\
&  \ \ \ \ \ \ \ \,+[I_{4}(s)+\langle q(s),\delta b_{x}(s)x^{1,\rho}%
(s)\rangle+k_{z}(s)\langle p(s),\delta b_{x}(s)x^{1,\rho}(s)\rangle
]I_{E_{\rho}}(s)\}ds-\int_{t}^{T}\hat{z}^{\rho}(s){d}w(s),
\end{split}
\label{Myeq4-6}%
\end{equation}
where
\begin{align*}
	&  \tilde{k}_{y}(t):=\int_{0}^{1}k_{y}(t,\bar{x}(t)+x^{1,\rho}(t)+x^{2,\rho
	}(t),\bar{y}(t)+I_{3}(t)+\mu\hat{y}^{\rho}(t),\bar{z}(t)+I_{2}(t)+\mu\hat
	{z}^{\rho}(t),\bar{u}(t))d\mu,\\
	&  \tilde{k}_{z}(t):=\int_{0}^{1}k_{z}(t,\bar{x}(t)+x^{1,\rho}(t)+x^{2,\rho
	}(t),\bar{y}(t)+I_{3}(t)+\mu\hat{y}^{\rho}(t),\bar{z}(s)+I_{2}(t)+\mu\hat
	{z}^{\rho}(t),\bar{u}(t))d\mu,\\
	&  J_{3}(t):=k(t,x^{\rho}(t),y^{\rho}(t),z^{\rho}(t),u^{\rho}(t))\\
	&  \ \ \ \ \ \ \ \ \ \ \, -k(t,\bar{x}(t)\!+\!x^{1,\rho}(t)\!+\!x^{2,\rho
	}(t),\bar{y}(t)\!+\!I_{3}(t)\!+\!\hat{y}^{\rho}(t),\bar{z}(t)\!+\!I_{2}%
	(t)\!+\!\hat{z}^{\rho}(t),\bar{u}(t)),\\
	&  J_{4}:=h(x^{\rho}(T))-h(\bar{x}(T))-\langle h_{x}(\bar{x}(T)),x^{1,\rho
	}(T)+x^{2,\rho}(T)\rangle-\frac{1}{2}\langle h_{xx}(\bar{x}(T))x^{1,\rho
	}(T),x^{1,\rho}(T)\rangle,\\& J_{5}(t):=\tilde{D}^{2}k(t)([x^{1,\rho}(t)+x^{2,\rho}(t),I_{3}%
	(t),I_{2}(t)],[x^{1,\rho}(t)+x^{2,\rho}(t),I_{3}(t),I_{2}(t)])\\
	&  \ \ \ \ \ \ \ \ \ \ \,-\langle D^{2}k(t)([I_{d},p(t),\bar{B}^{\ast
	}(t)p(t)+q(t)],[I_{d},p(t),\bar{B}^{\ast}(t)p(t)+q(t)])x^{1,\rho}%
	(t),x^{1,\rho}(t)\rangle,
\end{align*}
with 
$$ \tilde{D}^{2}k(t):=2\int_{0}^{1}\int_{0}^{1}\mu D^{2}k(t,\bar{x}(t)+\mu
\nu(x^{1,\rho}(t)+x^{2,\rho}(t)),\bar{y}(t)+\mu\nu I_{3}(t),\bar{z}(t)+\mu\nu
I_{2}(t),\bar{u}(t))d\mu d\nu. $$
 We can write $$J_{5}%
(t)=J_{6}(t)+J_{7}(t),$$ where%
\begin{align*}
J_{6}(t):= &  \langle\tilde{D}^{2}k(t)([I_{d},p(t),\bar{B}^{\ast
}(t)p(t)+q(t)],[I_{d},p(t),\bar{B}^{\ast}(t)p(t)+q(t)])x^{1,\rho}%
(t),x^{1,\rho}(t)\rangle\\
&  -\langle D^{2}k(t)([I_{d},p(t),\bar{B}^{\ast}(t)p(t)+q(t)],[I_{d}%
,p(t),\bar{B}^{\ast}(t)p(t)+q(t)])x^{1,\rho}(t),x^{1,\rho}(t)\rangle,\\
J_{7}(t):= &  \tilde{D}^{2}k(t)([x^{1,\rho}(t)+x^{2,\rho}(t),I_{3}%
(t),I_{2}(t)],[x^{1,\rho}(t)+x^{2,\rho}(t),I_{3}(t),I_{2}(t)])\\
&  -\langle\tilde{D}^{2}k(t)([I_{d},p(t),\bar{B}^{\ast}(t)p(t)+q(t)],[I_{d}%
,p(t),\bar{B}^{\ast}(t)p(t)+q(t)])x^{1,\rho}(t),x^{1,\rho}(t)\rangle.
\end{align*}
First, under assumption $(H4)$, we can check that
\begin{equation}
	|\langle v,B(t,\omega)w\rangle|=|\langle B^{\ast}(t,\omega)v,w\rangle|\leq
	C(K)\Vert v\Vert_{V}\Vert w\Vert_{H},\quad\text{for}\ v,w\in V\ \text{and}%
	\ (t,\omega)\in\lbrack0,T]\times\Omega.\label{Myeq4-25}%
\end{equation}
Indeed, for any $(t,\omega)\in\lbrack0,T]\times\Omega,$ recall that the
coercivity condition ((1) in the assumption $(H4)$) implies $\Vert B(t,\omega)v\Vert_{H}\leq C(K)\Vert
v\Vert_{V}$, for $v\in V$. That is,
\[
\Vert B(t,\omega)\Vert_{\mathfrak{L}(V,H)}\leq C(K).
\]
Moreover, according to \cite[Remark 2.4 (2)]{DM13}, we have $B(t,\omega
)+B^{\ast}(t,\omega)\in\mathfrak{L}(H).$ From (2) in the assumption $(H4)$, we
also have $|\langle v,(B(t,\omega)+B^{\ast}(t,\omega))v\rangle|=2|\langle
v,B(t,\omega)v\rangle|\leq2K\Vert v\Vert_{H}^{2},$ for $v\in V.$ Then by
\cite[Theorem VII.3.3]{Yo80}, we have
\begin{align*}
	\Vert B(t,\omega)+B^{\ast}(t,\omega)\Vert_{\mathfrak{L}(H)} &  =\sup_{v\in
		H,\Vert v\Vert_{H}\leq1}|\langle v,(B(t,\omega)+B^{\ast}(t,\omega))v\rangle|\\
	&  =\sup_{v\in V,\Vert v\Vert_{H}\leq1}|\langle v,(B(t,\omega)+B^{\ast
	}(t,\omega))v\rangle|\\
	&  \leq2K.
\end{align*}
Thus from $B^{\ast}(t,\omega)=(B(t,\omega)+B^{\ast}(t,\omega))-B(t,\omega),$
we deduce that $\Vert B^{\ast}(t,\omega)\Vert_{\mathfrak{L}(V,H)}\leq C(K),$
which implies (\ref{Myeq4-25}). Now denoting $\Vert\tilde{D}^{2}%
k(t)-D^{2}k(t)\Vert:=\Vert\tilde{D}^{2}k(t)-D^{2}k(t)\Vert_{\mathfrak{L}%
	_{2}((H\times\mathbb{R\times R)\times}(H\times\mathbb{R\times R)};\mathbb{R}%
	)}$, from (\ref{Myeq4-25}) we have
\begin{align*}
&  \mathbb{E}[(\int_{0}^{T}|J_{6}(t)|dt)^{2\alpha}]\\
&\ \ \ \,\ \ \ \,\leq C\mathbb{E}[(\int%
_{0}^{T}\Vert\tilde{D}^{2}k(t)-D^{2}k(t)\Vert((1+\Vert p(t)\Vert_{H}^{2})\Vert
x^{1,\rho}(t)\Vert_{H}^{2}\\
&  \ \ \ \,\ \ \ \,\ \ \ \,+\Vert p(t)\Vert_{V}^{2}\Vert x^{1,\rho}%
(t)\Vert_{H}^{2}+\Vert q(t)\Vert_{H}^{2}\Vert x^{1,\rho}(t)\Vert_{H}%
^{2})dt)^{2\alpha}]\\
&  \ \ \ \,\ \ \ \,\leq C(\mathbb{E}[\int_{0}^{T}\Vert\tilde{D}^{2}%
k(t)-D^{2}k(t)\Vert^{4\alpha}(1+\Vert p(t)\Vert_{H}^{8\alpha})dt])^{\frac
	{1}{2}}(\mathbb{E}[\int_{0}^{T}\Vert x^{1,\rho}(t)\Vert_{H}^{8\alpha
}dt])^{\frac{1}{2}}\\
&  \ \ \ \,\ \ \ \,\ \ \ \,+C(\mathbb{E}[(\int_{0}^{T}\Vert\tilde{D}%
^{2}k(t)-D^{2}k(t)\Vert\Vert p(t)\Vert_{V}^{2}dt)^{4\alpha}])^{\frac{1}{2}%
}(\mathbb{E}[\sup_{t\in\lbrack0,T]}\Vert x^{1,\rho}(t)\Vert_{H}^{8\alpha
}])^{\frac{1}{2}}\\
&  \ \ \ \,\ \ \ \,\ \ \ \,+C(\mathbb{E}[(\int_{0}^{T}\Vert\tilde{D}%
^{2}k(t)-D^{2}k(t)\Vert\Vert q(t)\Vert_{H}^{2}dt)^{4\alpha}])^{\frac{1}{2}%
}(\mathbb{E}[\sup_{t\in\lbrack0,T]}\Vert x^{1,\rho}(t)\Vert_{H}^{8\alpha
}])^{\frac{1}{2}}\\
&  \ \ \ \,\ \ \ \,=o(\rho^{2\alpha}).
\end{align*}
Furthermore, we can decompose $$J_{7}(t)=J_{7a}(t)+J_{7b}(t),$$ where%
\begin{align*}
J_{7a}(t):= &  \tilde{D}^{2}k(t)([x^{1,\rho}(t)+x^{2,\rho}(t),I_{3}%
(t),I_{2}(t)],[x^{2,\rho}(t),\langle p(t),x^{2,\rho}(t)\rangle+\frac{1}%
{2}\langle P(t)x^{1,\rho}(t),x^{1,\rho}(t)\rangle,I_{2}(t)\\
&  -\langle p(t),\bar{B}(t)x^{1,\rho}(t)\rangle-\langle q(t),x^{1,\rho
}(t)\rangle])
\end{align*}
and 
\begin{align*}
J_{7b}(t):= &  \langle\tilde{D}^{2}k(t)([x^{2,\rho}(t),\langle p(t),x^{2,\rho
}(t)\rangle+\frac{1}{2}\langle P(t)x^{1,\rho}(t),x^{1,\rho}(t)\rangle
,I_{2}(t)-\langle p(t),\bar{B}(t)x^{1,\rho}(t)\rangle\\
& -\langle q(t),x^{1,\rho
}(t)\rangle], \lbrack x^{1,\rho}(t),\langle p(t),x^{1,\rho}(t)\rangle,\langle
p(t),\bar{B}(t)x^{1,\rho}(t)\rangle+\langle q(t),x^{1,\rho}(t)\rangle]).
\end{align*}
From a
similar analysis as for $J_{6},$ we have%
\begin{align*}
	&  \mathbb{E}[(\int_{0}^{T}|J_{7b}(t)|dt)^{2\alpha}]\\
	&\leq C\mathbb{E}[(\int%
	_{0}^{T}(\Vert x^{2,\rho}(t)\Vert_{H}+|\langle p(t),x^{2,\rho}(t)\rangle
	+\frac{1}{2}\langle P(t)x^{1,\rho}(t),x^{1,\rho}(t)\rangle|+|I_{2}(t)\\
	& \ \ \ \, -\langle p(t),\bar{B}(t)x^{1,\rho}(t)\rangle-\langle q(t),x^{1,\rho
	}(t)\rangle|)(\Vert x^{1,\rho}(t)\Vert_{H}+|\langle p(t),x^{1,\rho}%
	(t)\rangle|+|\langle p(t),\bar{B}(t)x^{1,\rho}(t)\rangle\\
	&\ \ \  \,+\langle
	q(t),x^{1,\rho}(t)\rangle|)dt)^{2\alpha}]\\
	&  \leq C(\mathbb{E}[(\int_{0}^{T}\Vert x^{2,\rho}(t)\Vert_{H}^{2}+\Vert
	p(t)\Vert_{H}^{2}\Vert x^{2,\rho}(t)\Vert_{H}^{2}+\Vert p(t)\Vert_{V}^{2}\Vert
	x^{2,\rho}(t)\Vert_{H}^{2}+\Vert q(t)\Vert_{H}^{2}\Vert x^{2,\rho}(t)\Vert
	_{H}^{2}\\
	& \ \ \ \,  +\Vert p(t)\Vert_{H}^{2}\Vert x^{1,\rho}(t)\Vert_{H}^{2}I_{E_{\rho}%
	}(t)+(\Vert p(t)\Vert_{H}^{2}+\Vert P(t)\Vert_{\mathfrak{L}(H)}^{2})\Vert
	x^{1,\rho}(t)\Vert_{H}^{4}+|\mathcal{Z}(t)|^{2})dt)^{2\alpha}])^{\frac{1}{2}%
	}\\
	& \ \ \  \, \cdot(\mathbb{E}[(\int_{0}^{T}(\Vert x^{1,\rho}(t)\Vert_{H}^{2}+\Vert
	p(t)\Vert_{H}^{2}\Vert x^{1,\rho}(t)\Vert_{H}^{2}+\Vert p(t)\Vert_{V}%
	^{2}\Vert x^{1,\rho}(t)\Vert_{H}^{2}+\Vert q(t)\Vert_{H}^{2}\Vert x^{1,\rho
	}(t)\Vert_{H}^{2})dt)^{2\alpha}])^{\frac{1}{2}}\\
	&  =O(\rho^{3\alpha}).
\end{align*}
In the same manner, we derive that $$\mathbb{E}[(\int_{0}^{T}|J_{7a}(t)|dt)^{2\alpha}%
]=O(\rho^{3\alpha}).$$ Thus, $$\mathbb{E}[(\int_{0}^{T}|J_{5}(t)|dt)^{2\alpha
}]=o(\rho^{2\alpha}).$$ 
From Proposition 3.4, it is direct to check that $\mathbb{E}[|J_{4}|^{2\alpha
}]=o(\rho^{2\alpha})$ and $\mathbb{E}[(\int_{0}^{T}|J_{3}(t)|dt)^{2\alpha
}]=O(\rho^{2\alpha})$.
Recall that in (3.9)  we have obtained that
\begin{equation}
\sup_{t\in\lbrack0,T]}\mathbb{E}[|\sigma(t)|^{2\alpha}]=o(\rho^{2\alpha}).
\label{eq3-3}%
\end{equation}
Then by (3.15), (\ref{eq3-3}) and the a priori estimates for classical BSDEs,
\[
\sup_{t\in\lbrack0,T]}\mathbb{E}[|\hat{y}^{\rho}(t)-\frac{1}{2}\sigma
(t)|^{2\alpha}]+\mathbb{E}[(\int_{0}^{T}|\hat{z}^{\rho}(t)|^{2}dt)^{\alpha
}]=o(\rho^{\alpha}).
\]
Making use of (\ref{eq3-3}) again, we obtain (3.13). 

Now we prove the last estimate. Denote
\[
\tilde{x}^{\rho}(t)=x^{\rho}(t)-\bar{x}(t)-x^{1,\rho}(t)-x^{2,\rho
}(t),\] \[\tilde{y}^{\rho}(t)=\hat{y}^{\rho}(t)-\hat{y}(t),\]
\[\tilde{z}^{\rho
}(t)=\hat{z}^{\rho}(t)-\hat{z}(t).
\]
Then from (3.11) and (\ref{Myeq4-6}),
\begin{align*}
\tilde{y}^{\rho}(t)-\frac{1}{2}\sigma(t)= &  J_{4}+\int_{t}^{T}\{k_{y}%
(s)(\tilde{y}^{\rho}(t)-\frac{1}{2}\sigma(s))+k_{z}(s)\tilde{z}^{\rho
}(s)+\frac{1}{2}\tilde{k}_{y}(s)\sigma(s)\\
&+(\tilde{k}_{y}(s)-k_{y}(s))(\hat
{y}^{\rho}(s)-\frac{1}{2}\sigma(s)) +(\tilde{k}_{z}(s)-k_{z}(s))\hat{z}^{\rho}(s)+\frac{1}{2}J_{5}(s)\\
& +[\langle
q(s),\delta b_{x}(s)x^{1,\rho}(s)\rangle+k_{z}(s)\langle p(s),\delta
b_{x}(s)x^{1,\rho}(s)\rangle]I_{E_{\rho}}(s)\\
&  +J_{3}(s)-I_{5}(s)I_{E_{\rho}}(s)\}{d}s-\int_{t}^{T}\ \tilde{z}^{\rho
}(s){d}w(s).
\end{align*}
Note that
\begin{align*}
	&  |J_{3}(t)-I_{5}(t)I_{E_{\rho}}(t)|\\
	&\leq C\{\Vert\tilde{x}^{\rho}(t)\Vert
	_{H}+[\Vert x^{1,\rho}(t)+x^{2,\rho}(t)\Vert_{H}+|\hat{y}^{\rho}(t)|+|\hat
	{z}^{\rho}(t)|+|I_{2}(t)|+|I_{3}(t)|]I_{E_{\rho}}(t)\}\\
	&  \leq C\{\Vert\tilde{x}^{\rho}(t)\Vert_{H}+[|\hat{y}^{\rho}(t)|+|\hat
	{z}^{\rho}(t)|+(1+\Vert p(t)\Vert_{H}+\Vert q(t)\Vert_{H})\Vert x^{1,\rho
	}(t)+x^{2,\rho}(t)\Vert_{H}+\Vert p(t)\Vert_{H}\Vert x^{1,\rho}(t)\Vert_{H}\\
	&  \ \ \ \,+\Vert p(t)\Vert_{V}\Vert x^{1,\rho}(t)+x^{2,\rho}(t)\Vert
	_{H}+|\mathcal{Z}(t)|+(\Vert p(t)\Vert_{H}+\Vert P(t)\Vert_{\mathfrak{L}%
		(H)})\Vert x^{1,\rho}(t)\Vert_{H}^{2}]I_{E_{\rho}}(t)\}.
\end{align*}
We have
\begin{align*}
	&  \mathbb{E}[(\int_{0}^{T}|J_{3}(t)-I_{5}(t)I_{E_{\rho}}(t)|dt)^{2}]\\
	&  \leq C\mathbb{E}[\int_{0}^{T}\Vert\tilde{x}^{\rho}(t)\Vert_{H}^{2}%
	dt]+C\rho\{\mathbb{E}[\int_{0}^{T}(|\hat{y}^{\rho}(t)|^{2}+|\hat{z}^{\rho
	}(t)|^{2}+|\mathcal{Z}(t)|^{2})dt]\\
	&  \ \ \ \,+(\mathbb{E}[\int_{0}^{T}(1+\Vert p(t)\Vert_{H}^{4})I_{E_{\rho}%
	}(t)dt])^{\frac{1}{2}}(\mathbb{E}[\int_{0}^{T}\Vert x^{1,\rho}(t)+x^{2,\rho
	}(t)\Vert_{H}^{4}dt])^{\frac{1}{2}}\\
	&  \ \ \ \,+(\mathbb{E}[\int_{0}^{T}\Vert p(t)\Vert_{H}^{4}I_{E_{\rho}%
	}(t)dt])^{\frac{1}{2}}(\mathbb{E}[\int_{0}^{T}\Vert x^{1,\rho}(t)\Vert_{H}%
	^{4}dt])^{\frac{1}{2}}\\
	&  \ \ \ \,+(\mathbb{E}[(\int_{0}^{T}\Vert q(t)\Vert_{H}^{2}I_{E_{\rho}%
	}(t)dt)^{2}])^{\frac{1}{2}}(\mathbb{E}[\sup_{t\in\lbrack0,T]}\Vert x^{1,\rho
	}(t)+x^{2,\rho}(t)\Vert_{H}^{4}])^{\frac{1}{2}}\\
	&  \ \ \ \,+(\mathbb{E}[(\int_{0}^{T}\Vert p(t)\Vert_{V}^{2}I_{E_{\rho}%
	}(t)dt)^{2}])^{\frac{1}{2}}(\mathbb{E}[\sup_{t\in\lbrack0,T]}\Vert x^{1,\rho
	}(t)+x^{2,\rho}(t)\Vert_{H}^{4}])^{\frac{1}{2}}\\
	&  \ \ \ \,+(\mathbb{E}[\int_{0}^{T}(\Vert p(t)\Vert_{H}^{4}+\Vert
	P(t)\Vert_{\mathfrak{L}_{2}(H\mathbb{\times}H)}^{4})I_{E_{\rho}}%
	(t)dt])^{\frac{1}{2}}(\mathbb{E}[\int_{0}^{T}\Vert x^{1,\rho}(t)\Vert_{H}%
	^{8}dt])^{\frac{1}{2}}\}\\
	&  =o(\rho^{2}).
\end{align*}
Analogously, from
\begin{align*}
&|\tilde{k}_{y}(t)-k_{y}(t)|+|\tilde{k}_{z}(t)-k_{z}(t)|\\
&\leq C[\Vert x^{1,\rho
}(t)+x^{2,\rho}(t)\Vert_{H}+|I_{2}(t)|+|I_{3}(t)|+|\hat{y}^{\rho}(t)|+|\hat
{z}^{\rho}(t)|]\\
	&  \leq C[|\hat{y}^{\rho}(t)|+|\hat
{z}^{\rho}(t)|+(1+\Vert p(t)\Vert_{H}+\Vert q(t)\Vert_{H})\Vert x^{1,\rho
}(t)+x^{2,\rho}(t)\Vert_{H}+\Vert p(t)\Vert_{H}\Vert x^{1,\rho}(t)\Vert_{H}\\
&  \ \ \ \,+\Vert p(t)\Vert_{V}\Vert x^{1,\rho}(t)+x^{2,\rho}(t)\Vert
_{H}+|\mathcal{Z}(t)|+(\Vert p(t)\Vert_{H}+\Vert P(t)\Vert_{\mathfrak{L}%
	(H)})\Vert x^{1,\rho}(t)\Vert_{H}^{2}],
	\end{align*}
we also obtain
\begin{align*}
	&  \mathbb{E}[(\int_{0}^{T}|(\tilde{k}_{y}(t)-k_{y}(t))\hat{y}^{\rho
	}(t)+(\tilde{k}_{z}(t)-k_{z}(t))\hat{z}^{\rho}(t)|dt)^{2}]\\
	&  \leq C(\mathbb{E}[(\int_{0}^{T}(|\tilde{k}_{y}(t)-k_{y}(t)|^{2}+|\tilde
	{k}_{z}(t)-k_{z}(t)|^{2})dt)^{2}])^{\frac{1}{2}}(\mathbb{E}[(\int_{0}^{T}%
	|\hat{y}^{\rho}(t)|^{2}+|\hat{z}^{\rho}(t)|^{2}dt)^{2}])^{\frac{1}{2}}\\
	&  \leq C(\mathbb{E}[(\int_{0}^{T}(|\hat{y}^{\rho}(t)|^{2}+|\hat{z}^{\rho
	}(t)|^{2}+(1+\Vert p(t)\Vert_{H}^{2}+\Vert q(t)\Vert_{H}^{2})\Vert x^{1,\rho
	}(t)+x^{2,\rho}(t)\Vert_{H}^{2}+\Vert p(t)\Vert_{H}^{2}\Vert x^{1,\rho
	}(t)\Vert_{H}^{2}\\
	&  \ \ \ \,+\,\Vert p(t)\Vert_{V}^{2}\Vert x^{1,\rho}(t)+x^{2,\rho}%
	(t)\Vert_{H}^{2}+|\mathcal{Z}(t)|^{2}+(\Vert p(t)\Vert_{H}^{2}+\Vert
	P(t)\Vert_{\mathfrak{L}(H)}^{2})\Vert x^{1,\rho}(t)\Vert_{H}^{4}%
	)dt)^{2}])^{\frac{1}{2}}\\
	&  \ \ \ \,\cdot(\mathbb{E}[(\int_{0}^{T}|\hat{y}^{\rho}(t)|^{2}+|\hat
	{z}^{\rho}(t)|^{2}dt)^{2}])^{\frac{1}{2}}\\
	&  \leq C\{(\mathbb{E}[(\int_{0}^{T}(|\hat{y}^{\rho}(t)|^{2}+|\hat{z}^{\rho
	}(t)|^{2}+|\mathcal{Z}(t)|^{2})dt)^{2}])^{\frac{1}{2}}\\
	&  \ \ \ \,+(\mathbb{E}[\int_{0}^{T}(1+\Vert p(t)\Vert_{H}^{8})dt])^{\frac
		{1}{4}}(\mathbb{E}[\int_{0}^{T}\Vert x^{1,\rho}(t)+x^{2,\rho}(t)\Vert_{H}%
	^{8}dt])^{\frac{1}{4}}\\
	&  \ \ \ \,+(\mathbb{E}[(\int_{0}^{T}\Vert q(t)\Vert_{H}^{2}dt)^{4}%
	])^{\frac{1}{4}}(\mathbb{E}[\sup_{t\in\lbrack0,T]}\Vert x^{1,\rho
	}(t)+x^{2,\rho}(t)\Vert_{H}^{8}])^{\frac{1}{4}}+(\mathbb{E}[\int_{0}^{T}\Vert
	p(t)\Vert_{H}^{8}dt])^{\frac{1}{4}}(\mathbb{E}[\int_{0}^{T}\Vert x^{1,\rho
	}(t)\Vert_{H}^{8}dt])^{\frac{1}{4}}\\
	&  \ \ \ \,+(\mathbb{E}[(\int_{0}^{T}\Vert p(t)\Vert_{V}^{2}dt)^{4}%
	])^{\frac{1}{4}}(\mathbb{E}[\sup_{t\in\lbrack0,T]}\Vert x^{1,\rho
	}(t)+x^{2,\rho}(t)\Vert_{H}^{8}])^{\frac{1}{4}}\\
	&  \ \ \ \,+(\mathbb{E}[\int_{0}^{T}(\Vert p(t)\Vert_{H}^{8}+\Vert
	P(t)\Vert_{\mathfrak{L}(H)}^{8})dt])^{\frac{1}{4}}(\mathbb{E}[\int_{0}%
	^{T}\Vert x^{1,\rho}(t)\Vert_{H}^{16}dt])^{\frac{1}{4}}\}(\mathbb{E}[(\int%
	_{0}^{T}|\hat{y}^{\rho}(t)|^{2}+|\hat{z}^{\rho}(t)|^{2}dt)^{2}])^{\frac{1}{2}%
	}\\
	&  =o(\rho^{2}).
\end{align*}
Therefore,
\[
\sup_{t\in\lbrack0,T]}\mathbb{E}[|\tilde{y}^{\rho}(t)-\frac{1}{2}%
\sigma(t)|^{2}]+\mathbb{E}[\int_{0}^{T}|\tilde{z}^{\rho}(t)|^{2}dt]=o(\rho
^{2}).
\]
This, together with (\ref{eq3-3}), implies (3.14).
\end{proof}
\begin{remark}\label{Rm-B}
	\upshape{From the proofs we can know that if $B\equiv 0$ or $k$ does not contain $z$, it is not necessary to estimate $p$ in the space $V$ in (\ref{Myeq4-22}).}
\end{remark}

\noindent\textbf{Step 3: Duality for BSDEs and the completion of the proof.}
Consider the following adjoint equation for BSDE (\ref{eq4-10}):
\begin{equation}
\lambda(t)=1+\int_{0}^{t}k_{y}(s)\,\lambda(s){d}s+\int_{0}^{t}k_{z}%
(s)\,\lambda(s){d}w(s). \label{Myeq5-5}%
\end{equation}
Applying It\^{o}'s formula to $\lambda(t)\hat{y}(t)$, we get
\begin{align*}
\hat{y}(0)=  &  \mathbb{E}\int_{0}^{T}\lambda(t)[\langle p(t),\delta
a(t)\rangle+\langle q(t),\delta b(t)\rangle+k(t,\bar{x}(t),\bar{y}(t),\bar
{z}(t)+\langle p(t),\delta b(t)\rangle,u(t))\\
&  -k(t,\bar{x}(t),\bar{y}(t),\bar{z}(t),\bar{u}(t)) +\frac{1}{2}\langle
P(t)\delta b(t),\delta b(t)\rangle]I_{E_{\rho}}(t){d}t.
\end{align*}
From the optimization assumption and (\ref{eq-3-4}),
\begin{align*}
0  &  \leq J(u^{\rho}(\cdot))-J(\bar{u}(\cdot))=y^{\rho}(0)-\bar{y}(0)\\
&  =\hat{y}^{\rho}(0)+\langle p(0),x^{1,\rho}(0)+x^{2,\rho}(0)\rangle+\frac
{1}{2}\langle P(0)x^{1,\rho}(0),x^{1,\rho}(0)\rangle\\
&  =\hat{y}(0)+o(\rho)\\
&  =\mathbb{E}\int_{0}^{T}\lambda(t)[\langle p(t),\delta a(t)\rangle+\langle
q(t),\delta b(t)\rangle+k(t,\bar{x}(t),\bar{y}(t),\bar{z}(t)+\langle
p(t),\delta b(t)\rangle,u(t))-k(t,\bar{x}(t),\bar{y}(t),\bar{z}(t),\bar
{u}(t))\\
&  \ \ \ \, +\frac{1}{2}\langle P(t)\delta b(t),\delta b(t)\rangle]I_{E_{\rho
}}(t)dt+o(\rho).
\end{align*}
Note that $\lambda(t)>0$ for $t\in\lbrack0,T],$ we then obtain the pointwise
maximum principle as%
\begin{align*}
&  \langle p(t),\delta a(t;v)\rangle+\langle q(t),\delta b(t;v)\rangle
+k(t,\bar{x}(t),\bar{y}(t),\bar{z}(t)+\langle p(t),\delta b(t;v)\rangle
,v)-k(t,\bar{x}(t),\bar{y}(t),\bar{z}(t),\bar{u}(t))\\
&  +\frac{1}{2}\langle{P}(t)\delta b(t;v),\delta b(t;v)\rangle\geq
0,\quad\forall v\in U,\ P\text{-a.s. a.e.,}%
\end{align*}
which can also be written as (\ref{MP}). The proof is now complete.

\subsection{Application on controlled SPDEs}

We present an example of controlled SPDEs that fits our framework. Let $G$ be
a bounded domain in $\mathbb{R}^{n}$. Consider super-parabolic stochastic PDE
(cf. \cite{Ro18})%
\[%
\begin{cases}
{d}x(t,\zeta) & =[\sum_{i,j=1}^{n}\partial_{\zeta_{i}}(\alpha_{ij}
(t,\zeta)\partial_{\zeta_{j}}x(t,\zeta))+a(t,\zeta,u(t),x(t,\zeta))]{d}
t+[\sum_{i=1}^{n}\beta_{i}(t,\zeta)\partial_{\zeta_{i}}x(t,\zeta)\\
& +b(t,\zeta,u(t),x(t,\zeta))]{d}w(t),\text{ }(t,\zeta)\in\lbrack0,T]\times
G,\\
x(0,\zeta) & =x_{0}(\zeta),\text{ }\zeta\in G,\\
x(t,\zeta) & =0,\text{ }(t,\zeta)\in\lbrack0,T]\times\partial G,
\end{cases}
\]
Here $\alpha_{ij},\beta_{i},a,b$ and\ $x_{0}$\ are given coefficients and
initial value, respectively.\ The control $u(t)$ is a progressive process
taking values in some metric space $U$. We consider the problem of minimizing
the cost functional
\[
J(u(\cdot))=y(0),
\]
where $y$ is the recursive utility subjected to a BSDE:
\[
y(t)=\int_{G}h(\zeta,x(T,\zeta))d\zeta+\int_{t}^{T}\int_{G}k(s,\zeta
,y(s),z(s),u(s),x(s,\zeta))d\zeta ds-\int_{t}^{T}z(s)dw({s}).
\]
We impose standard measurability conditions on the coefficients. We take
\[
H=L^{2}(G),\text{ }V=H_{0}^{1}(G),\text{ }A=\sum_{i,j=1}^{n}\partial
_{\zeta_{i}}(\alpha_{ij}(t,\zeta)\partial_{\zeta_{j}}),\text{ }B=\sum
_{i=1}^{n}\beta_{i}(t,\zeta)\partial_{\zeta_{i}}.
\]
To guarantee the condition (H4), we assume there exist some constants
$0<\kappa\leq K$ such that%
\[
\kappa I_{n\times n}+(\beta_{i}\beta_{j})_{n\times n}\leq2(\alpha
_{ij})_{n\times n}\leq KI_{n\times n},
\]
the function $\beta_{i}$ is continuously differentiable with respect to
$\zeta,$ and $\alpha_{ij},\beta_{i},\partial_{\zeta_{i}}\beta_{i}$ are bounded
by $K.$ Indeed, the proof for the coercivity condition is standard and can be
found in \cite{Ro18}) and the quasi-skew-symmetry condition can be deduced by
the observation that%
\begin{align*}
\int_{G}(\beta_{i}(t,\zeta)\partial_{\zeta_{i}}x(t,\zeta))x(t,\zeta)d\zeta &
=-\int_{G}x(t,\zeta)\partial_{\zeta_{i}}(\beta_{i}(t,\zeta)x(t,\zeta))d\zeta\\
&  =-\int_{G}x(t,\zeta)\beta_{i}(t,\zeta)\partial_{\zeta_{i}}x(t,\zeta
)d\zeta-\int_{G}\partial_{\zeta_{i}}\beta_{i}(t,\zeta)|x(t,\zeta)|^{2}d\zeta.
\end{align*}
Next, provided the corresponding differentiation and growth conditions on the
coefficients $a,b,h$ and $k,$ the assumption (H5) can be verified (cf.
\cite{LZ18}). 
	Therefore, we obtain the maximum principle for the above stochastic
optimal control problem.

\section{Appendix}

\subsection{Proof of Proposition \ref{Myth2-7}}

We have the decomposition:
\begin{align*}
|\langle P(t+\delta)u,v\rangle-\langle P(t)u,v\rangle|  &  \leq|\mathbb{E}%
[\langle\xi L(t+\delta,T)u,L(t+\delta,T)v\rangle|\mathcal{F}_{t+\delta
}]-\mathbb{E}[\langle\xi L(t,T)u,L(t+\delta,T)v\rangle|\mathcal{F}_{t+\delta
}]|\\
&  +|\mathbb{E}[\langle\xi L(t,T)u,L(t+\delta,T)v\rangle
|\mathcal{F}_{t+\delta}]-\mathbb{E}[\langle\xi L(t,T)u,L(t,T)v\rangle
|\mathcal{F}_{t+\delta}]|\\
&  +|\mathbb{E}[\langle\xi L(t,T)u,L(t,T)v\rangle|\mathcal{F}%
_{t+\delta}]-\mathbb{E}[\langle\xi L(t,T)u,L(t,T)v\rangle|\mathcal{F}_{t}]|\\
&  +\mathbb{E}[\int_{t+\delta}^{T}|\langle f(s)L(t+\delta,s)u,L(t+\delta
,s)v\rangle-\langle f(s)L(t,s)u,L(t+\delta,s)v\rangle|ds|\mathcal{F}%
_{t+\delta}]\\
&  +\mathbb{E}[\int_{t+\delta}^{T}|\langle f(s)L(t+\delta,s)u,L(t+\delta
,s)v\rangle-\langle f(s)L(t,s)u,L(t+\delta,s)v\rangle|ds|\mathcal{F}%
_{t+\delta}]\\
&  +\mathbb{E}[\int_{t+\delta}^{T}|\langle f(s)L(t,s)u,L(t+\delta
,s)v\rangle-\langle f(s)L(t,s)u,L(t,s)v\rangle|ds|\mathcal{F}_{t+\delta}]\\
&  +|\mathbb{E}[\int_{t+\delta}^{T}\langle f(s)L(t,s)u,L(t,s)v\rangle
ds|\mathcal{F}_{t+\delta}]-\mathbb{E}[\int_{t}^{T}\langle
f(s)L(t,s)u,L(t,s)v\rangle ds|\mathcal{F}_{t+\delta}]|\\
&  +|\mathbb{E}[\int_{t}^{T}\langle f(s)L(t,s)u,L(t,s)v\rangle ds|\mathcal{F}%
_{t+\delta}]-\mathbb{E}[\int_{t}^{T}\langle f(s)L(t,s)u,L(t,s)v\rangle
ds|\mathcal{F}_{t}]|.
\end{align*}
We only show the convergence of the first, third and fourth terms, and the
others can be estimated in the same manner. As $\delta\downarrow0,$ we have by
the assumption $(H3)$
\begin{align*}
&  \mathbb{E}[|\mathbb{E}[\langle\xi L(t+\delta,T)u,L(t+\delta,T)v\rangle
|\mathcal{F}_{t+\delta}]-\mathbb{E}[\langle\xi L(t,T)u,L(t+\delta
,T)v\rangle|\mathcal{F}_{t+\delta}]|^{\alpha}]^{\frac{1}{\alpha}}\\
&  \leq(\mathbb{E}[\Vert\xi\Vert_{\mathfrak{L}(H)}^{2\alpha}])^{\frac
{1}{2\alpha}}(\mathbb{E}[\Vert L(t+\delta,T)(u-L(t,t+\delta)u)\Vert
_{H}^{4\alpha}])^{\frac{1}{4\alpha}}(\mathbb{E}[\Vert L(t+\delta,T)v\Vert
_{H}^{4\alpha}])^{\frac{1}{4\alpha}}\\
&  \leq C_{1}(\mathbb{E}[\Vert u-L(t,t+\delta)u\Vert_{H}^{4\alpha}])^{\frac
{1}{4\alpha}}\rightarrow0,
\end{align*}
where $C_{1}$ is a constant independent of $\delta,$ and by the martingale
convergence theorem
\[
\mathbb{E}[|\mathbb{E}[\langle\xi L(t,T)u,L(t,T)v\rangle|\mathcal{F}%
_{t+\delta}]-\mathbb{E}[\langle\xi L(t,T)u,L(t,T)v\rangle|\mathcal{F}%
_{t}]|^{\alpha}]\rightarrow0.
\]
Making use of the assumption $(H3)$ again, we also obtain that, as $\delta\downarrow0,$
\begin{align*}
&  \mathbb{E}[|\mathbb{E}[\int_{t+\delta}^{T}\langle f(s)L(t+\delta
,s)u,L(t+\delta,s)v\rangle ds|\mathcal{F}_{t+\delta}]-\mathbb{E}%
[\int_{t+\delta}^{T}\langle f(s)L(t,s)u,L(t+\delta,s)v\rangle ds|\mathcal{F}%
_{t+\delta}]|^{\alpha}]^{\frac{1}{\alpha}}\\
&  \leq(\mathbb{E}[(\int_{t+\delta}^{T}\Vert f(s)\Vert_{\mathfrak{L}(H)}%
^{2}ds)^{\alpha}])^{\frac{1}{2\alpha}}(\mathbb{E}[\int_{t+\delta}^{T}\Vert
L(t+\delta,s)(u-L(t,t+\delta)u)\Vert_{H}^{4\alpha}ds])^{\frac{1}{4\alpha}%
}(\mathbb{E}[\int_{t+\delta}^{T}\Vert L(t+\delta,T)v\Vert_{H}^{4\alpha
}ds])^{\frac{1}{4\alpha}}\\
&  \leq C_{1}(\mathbb{E}[\int_{t+\delta}^{T}\Vert u-L(t,t+\delta)u\Vert
_{H}^{4\alpha}])^{\frac{1}{4\alpha}}\rightarrow0.
\end{align*}

\subsection{Proof of Theorem \ref{Myth3-7}}

One crucial ingredient in the proof is the following estimate.

\begin{theorem}
\label{Myito2-12} Let the assumptions of Theorem \ref{Myth3-7} hold. Define,
for $t\in\lbrack0,T]$,
\begin{equation}%
\begin{split}
\sigma(t):=  &  \mathbb{E}[\frac{\lambda(T)}{\lambda(t)}\langle\xi
x(T),x(T)\rangle+\int_{t}^{T}\frac{\lambda(s)}{\lambda(t)}\langle
f(s,P(s))x(s),x(s)\rangle ds\\
&  -\int_{t}^{T}\frac{\lambda(s)}{\lambda(t)}\langle P(s)\zeta(s),\zeta
(s)\rangle I_{E_{\rho}}(s)ds|\mathcal{F}_{t}] -\langle P(t)x(t),x(t)\rangle
\end{split}
\label{Myeq4-16}%
\end{equation}
with
\begin{equation}
\lambda(t):=e^{\int_{0}^{t}-\frac12{\beta^{2}(s)}ds+\beta(s)dw({s})}.
\label{Myeq4-29}%
\end{equation}
Then the process $\sigma$ satisfies (\ref{Myeq2-16}).
\end{theorem}

Let us admit for a moment the following result on moving the nonhomogeneous
term from the diffusion to the initial point.

\begin{proposition}
\label{Myle4-7} Suppose $(H4)$ holds. Given any $\alpha\geq1$ and $\zeta
_{0}\in L^{2\alpha}(\mathcal{F}_{t_{0}},V)$, let $y$ solve SEE
\[%
\begin{cases}
{d}y(t) & =A(t)y(t){d}t+[B(t)y(t)+\zeta_{0}I_{E_{\rho}}(t)]{d}w(t),\\
y(0) & =0,
\end{cases}
\]
and define
\[
z(t):=%
\begin{cases}
0, & t<t_{0},\\
\eta(t), & t_{0}\leq t<t_{0}+\rho,\\
z(t):z(t)\text{ solves }z(t)=\eta(t_{0}+\rho)+\int_{t_{0}+\rho}^{t}%
A(s)z(s){d}s+\int_{t_{0}+\rho}^{t}B(s)z(s){d}w(s), & t\geq t_{0}+\rho,
\end{cases}
\]
where
\[
\eta(t):=\frac{1}{\sqrt{\rho}}\zeta_{0}\int_{t_{0}}^{t}I_{E_{\rho}}%
(s){d}w(s),\quad t\geq t_{0}.
\]
Then there exists some constant $C>0$ depending on $\alpha$, $\delta$ and $K$
such that
\[
\mathbb{E}[\sup_{t\in\lbrack0,T]}\Vert y(t)-\sqrt{\rho}z(t)\Vert_{H}^{2\alpha
}]\leq C\mathbb{E}[\Vert\zeta_{0}\Vert_{V}^{2\alpha}]\rho^{2\alpha}\text{.}%
\]

\end{proposition}

\begin{proof}
[Proof of Theorem \ref{Myito2-12}]The proof is divided into the following
three steps. Moreover, we only need to give the estimate of $\mathbb{E}%
[|\sigma(t)|^{\alpha}]$ for any given $t$, since this bound can be chosen to
be independent of $t$ according to the latter proof.

\textit{Step 1: an auxiliary approximation result.} By the following Lemma
\ref{Myle4-1}, we have%
\[
\tilde{L}(\hat{t},s)=\frac{\lambda_{1}(s)}{\lambda_{1}(\hat{t})}L(\hat
{t},s),\quad\text{for any}\ \hat{t}\leq s\leq T,
\]
with
\[
L(\hat{t},s):=L_{A,B}(\hat{t},s)\quad\text{and}\quad\lambda_{1}(s):=e^{\int%
_{0}^{s}-\frac14\beta^{2}(r)dr+\frac12\beta(r)dw({r})}.
\]
Noting that $\lambda=\lambda_{1}\cdot\lambda_{1},$ then
\[
P(\hat{t})=\mathbb{E}[\frac{\lambda(T)}{\lambda(\hat{t})}L^{\ast}(\hat
{t},T)\xi L(\hat{t},T)+\int_{\hat{t}}^{T}\frac{\lambda(s)}{\lambda(\hat{t}%
)}L^{\ast}(\hat{t},s)f(s,P(s))L^{\ast}(\hat{t},s)ds|\mathcal{F}_{\hat{t}}].
\]
Given any $\zeta_{0}\in L^{4\alpha}(\mathcal{F}_{t_{0}},V),$ we define $z(t)$
as in Proposition \ref{Myle4-7}. For $\hat{t}\geq t_{0}+\rho,$ it holds that
$L(\hat{t},s)z(\hat{t})=z(s)$ for $s\geq\hat{t}$, and thus
\begin{equation}
\label{Myeq5-6}%
\begin{split}
\langle P(\hat{t})z(\hat{t}),z(\hat{t})\rangle &  =\mathbb{E}[\frac
{\lambda(T)}{\lambda(\hat{t})}\langle\xi L(\hat{t},T)z(\hat{t}),L(\hat
{t},T)z(\hat{t})\rangle+\int_{\hat{t}}^{T}\frac{\lambda(s)}{\lambda(\hat{t}%
)}\langle f(s,P(s))L(\hat{t},s)z(\hat{t}),L(\hat{t},s)z(\hat{t})\rangle
ds|\mathcal{F}_{\hat{t}}]\\
&  =\mathbb{E}[\frac{\lambda(T)}{\lambda(\hat{t})}\langle\xi z(T),z(T)\rangle
+\int_{\hat{t}}^{T}\frac{\lambda(s)}{\lambda(\hat{t})}\langle
f(s,P(s))z(s),z(s)\rangle ds|\mathcal{F}_{\hat{t}}].
\end{split}
\end{equation}
Fix any $t\in\lbrack0,T].$ Based on (\ref{Myeq5-6}), we separate our
discussions into two cases: (1) $t>t_{0}$; (2) $t\leq t_{0}$. For the first
case, when $\rho$ is small, it holds that $t\geq t_{0}+\rho$ and then
\begin{equation}
\langle P(t)z(t),z(t)\rangle=\mathbb{E}[\frac{\lambda(T)}{\lambda(t)}%
\langle\xi z(T),z(T)\rangle+\int_{t}^{T}\frac{\lambda(s)}{\lambda(t)}\langle
f(s,P(s))z(s),z(s)\rangle ds|\mathcal{F}_{t}]. \label{Myeq4-4}%
\end{equation}
For the second case,\ we have
\begin{align*}
\langle P(t_{0}+\rho)z(t_{0}+\rho),z(t_{0}+\rho)\rangle &  =\mathbb{E}%
[\frac{\lambda(T)}{\lambda(t_{0}+\rho)}\langle\xi z(T),z(T)\rangle+\int%
_{t_{0}+\rho}^{T}\frac{\lambda(s)}{\lambda(t_{0}+\rho)}\langle
f(s,P(s))z(s),z(s)\rangle ds|\mathcal{F}_{t_{0}+\rho}]\\
&  =\mathbb{E}[\frac{\lambda(T)}{\lambda(t_{0}+\rho)}\langle\xi
z(T),z(T)\rangle+\int_{t_{0}}^{T}\frac{\lambda(s)}{\lambda(t_{0}+\rho)}\langle
f(s,P(s))z(s),z(s)\rangle ds|\mathcal{F}_{t_{0}+\rho}]\\
&  \ \ \ \, -\int_{t_{0}}^{t_{0}+\rho}\frac{\lambda(s)}{\lambda(t_{0}+\rho
)}\langle f(s,P(s))z(s),z(s)\rangle ds.
\end{align*}
Taking $\mathcal{F}_{t}$-conditional expectation on both sides, we then get
\begin{equation}%
\begin{split}
&  \mathbb{E}[\frac{\lambda(t_{0}+\rho)}{\lambda(t)}\langle P(t_{0}%
+\rho)z(t_{0}+\rho),z(t_{0}+\rho)\rangle|\mathcal{F}_{t}]=\mathbb{E}%
[\frac{\lambda(T)}{\lambda(t)}\langle\xi z(T),z(T)\rangle\\
&  \quad+\int_{t_{0}}^{T}\frac{\lambda(s)}{\lambda(t)}\langle
f(s,P(s))z(s),z(s)\rangle ds|\mathcal{F}_{t}]-\mathbb{E}[\int_{t_{0}}%
^{t_{0}+\rho}\frac{\lambda(s)}{\lambda(t)}\langle f(s,P(s))z(s),z(s)\rangle
ds|\mathcal{F}_{t}].
\end{split}
\label{Myeq4-9}%
\end{equation}

We can write (\ref{Myeq4-4}) and (\ref{Myeq4-9}) into a unified form as
\begin{equation}%
\begin{split}
&  \langle P(t)\sqrt{\rho}z(t),\sqrt{\rho}z(t)\rangle+\mathbb{E}[\int_{t}%
^{T}\frac{\lambda(t_{0}+\rho)}{\lambda(t)}\langle P(t_{0}+\rho)z(t_{0}%
+\rho),z(t_{0}+\rho)\rangle I_{E_{\rho}}(s)ds|\mathcal{F}_{t}]\\
&  =\mathbb{E}[\frac{\lambda(T)}{\lambda(t)}\langle\xi\sqrt{\rho}%
z(T),\sqrt{\rho}z(T)\rangle+\int_{t}^{T}\frac{\lambda(s)}{\lambda(t)}\langle
f(s,P(s))\sqrt{\rho}z(s),\sqrt{\rho}z(s)\rangle ds|\mathcal{F}_{t}]\\
&  \ \ \ \,-\rho\mathbb{E}[\int_{t}^{T}\frac{\lambda(s)}{\lambda(t)}\langle
f(s,P(s))z(s),z(s)\rangle I_{E_{\rho}}(s)ds|\mathcal{F}_{t}],\quad
\text{when}\ \rho\ \text{is small.}%
\end{split}
\label{Myeq4-15}%
\end{equation}

\textit{Step 2: the case of $\zeta(s)=\zeta_{0},$ $s\geq t_{0},$ for some
$\zeta_{0}\in L^{4\alpha}(\mathcal{F}_{t_{0}},H)$.} In this case, we denote
the corresponding $\sigma$ by $\sigma^{t_{0},\zeta_{0}}.$

Assume first that $\zeta_{0}\in L^{4\alpha}(\mathcal{F}_{t_{0}},V)$ and define
the corresponding $z(t)$ as in Step 1$.$ From the identity (\ref{Myeq4-15}),
we have that
\begin{align*}
&  \sigma^{t_{0},\zeta_{0}}(t) =\{\mathbb{E}[\frac{\lambda(T)}{\lambda
(t)}\langle\xi x(T),x(T)\rangle|\mathcal{F}_{t}]-\mathbb{E}[\frac{\lambda
(T)}{\lambda(t)}\langle\xi\sqrt{\rho}z(T),\sqrt{\rho}z(T)\rangle
|\mathcal{F}_{t}]\}\\
&  \ \ \ +\{\mathbb{E}[\int_{t}^{T}\frac{\lambda(s)}{\lambda(t)}\langle
f(s,P(s))x(s),x(s)\rangle ds|\mathcal{F}_{t}]-\mathbb{E}[\int_{t}^{T}%
\frac{\lambda(s)}{\lambda(t)}\langle f(s,P(s))\sqrt{\rho}z(s),\sqrt{\rho
}z(s)\rangle ds|\mathcal{F}_{t}]\}\\
&  \ \ \ +\{\mathbb{E}[\int_{t}^{T}\frac{\lambda(t_{0}+\rho)}{\lambda
(t)}\langle P(t_{0}+\rho)z(t_{0}+\rho),z(t_{0}+\rho)\rangle I_{E_{\rho}%
}(s)ds|\mathcal{F}_{t}]-\mathbb{E}[\int_{t}^{T}\frac{\lambda(s)}{\lambda
(t)}\langle P(s)\zeta_{0},\zeta_{0}\rangle I_{E_{\rho}}(s)ds|\mathcal{F}%
_{t}]\}\\
&  \ \ \ +\{\langle P(t)\sqrt{\rho}z(t),\sqrt{\rho}z(t)\rangle-\langle
P(t)x(t),x(t)\rangle\}+\rho\mathbb{E}[\int_{t}^{T}\frac{\lambda(s)}%
{\lambda(t)}\langle f(s,P(s))z(s),z(s)\rangle I_{E_{\rho}}(s)ds|\mathcal{F}%
_{t}]\\
&  \ \ \ =:I_{1}+I_{2}+I_{3}+I_{4}+I_{5},\quad\text{when}\ \rho\ \text{is
small.}%
\end{align*}

We only provide the estimates for $I_{1}$, $I_{3}$ and $I_{5}$, the other
terms can be handled in a similar manner. For notational simplicity, we use
$C_{1}$ to denote a constant independent of $\rho$, which may vary from line
to line$.$ For the $I_{1}$ term, let $\alpha^{\prime}$ be the H\"{o}lder
conjugate of $\alpha$, since $\lambda$ is an exponential martingale, we have
\[%
\begin{split}
&  \mathbb{E}[|\frac{\lambda(T)}{\lambda(t)}||\langle\xi x(T),x(T)\rangle
-\langle\xi\sqrt{\rho}z(T),\sqrt{\rho}z(T)\rangle||\mathcal{F}_{t}]\\
&  \leq(\mathbb{E}[|\frac{\lambda(T)}{\lambda(t)}|^{\alpha^{\prime}%
}|\mathcal{F}_{t}])^{\frac{1}{\alpha^{\prime}}}(\mathbb{E}[|\langle\xi
x(T),x(T)\rangle-\langle\xi\sqrt{\rho}z(T),\sqrt{\rho}z(T)\rangle|^{\alpha
}|\mathcal{F}_{t}])^{\frac{1}{\alpha}}\\
&  \leq C_{1}(\mathbb{E}[|\langle\xi x(T),x(T)\rangle-\langle\xi\sqrt{\rho
}z(T),\sqrt{\rho}z(T)\rangle|^{\alpha}|\mathcal{F}_{t}])^{\frac{1}{\alpha}}.
\end{split}
\]
Thus in virtue of Proposition \ref{Myle4-7}, we obtain
\begin{align*}
(\mathbb{E}[|I_{1}|^{\alpha}])^{\frac{1}{\alpha}}  &  \leq C_{1}%
(\mathbb{E}[|\langle\xi x(T),x(T)\rangle-\langle\xi\sqrt{\rho}z(T),\sqrt{\rho
}z(T)\rangle|^{\alpha}])^{\frac{1}{\alpha}}\\
&  \leq C_{1}(\mathbb{E}[\Vert\xi\Vert_{\mathfrak{L}(H)}^{2\alpha}])^{\frac
{1}{2\alpha}}(\mathbb{E}[\Vert x(T)-\sqrt{\rho}z(T)\Vert_{H}^{4\alpha
}])^{\frac{1}{4\alpha}}\{(\mathbb{E}[\Vert x(T)\Vert_{H}^{4\alpha}])^{\frac
{1}{4\alpha}}+(\mathbb{E}[\Vert\sqrt{\rho}z(T)\Vert_{H}^{4\alpha}])^{\frac
{1}{4\alpha}}\}\\
&  \leq C_{1}\rho^{\frac{3}{2}}\\
&  =o_{\zeta_{0}}(\rho).
\end{align*}
Now we consider the $I_{3}$ term. If $t>t_{0},$ it holds trivially that
$I_{3}=0$ for $\rho$ small enough. Now we assume $t\leq t_{0}$. Denote
$t_{1}:=t_{0}+\rho$ for simplicity. Noting that $z(t_{1})=\frac{w(t_{1}%
)-w({t_{0}})}{\sqrt{\rho}}\zeta_{0}$, then from the It\^{o}'s isometry, we
have
\[
\mathbb{E}[\int_{t}^{T}\frac{\lambda(t_{0})}{\lambda(t)}\langle P(t_{0}%
)z(t_{1}),z(t_{1})\rangle I_{E_{\rho}}(s)ds|\mathcal{F}_{t}]=\mathbb{E}%
[\int_{t}^{T}\frac{\lambda(t_{0})}{\lambda(t)}\langle P(t_{0})\zeta_{0}%
,\zeta_{0}\rangle I_{E_{\rho}}(s)ds|\mathcal{F}_{t}].
\]
Thus,
\begin{align*}
I_{3}  &  =\{\mathbb{E}[\int_{t}^{T}\frac{\lambda(t_{1})}{\lambda(t)}\langle
P(t_{1})z(t_{1}),z(t_{1})\rangle I_{E_{\rho}}(s)ds|\mathcal{F}_{t}%
]-\mathbb{E}[\int_{t}^{T}\frac{\lambda(t_{0})}{\lambda(t)}\langle
P(t_{1})z(t_{1}),z(t_{1})\rangle I_{E_{\rho}}(s)ds|\mathcal{F}_{t}]\}\\
&  \ \ \ \,+\{\mathbb{E}[\int_{t}^{T}\frac{\lambda(t_{0})}{\lambda(t)}\langle
P(t_{1})z(t_{1}),z(t_{1})\rangle I_{E_{\rho}}(s)ds|\mathcal{F}_{t}%
]-\mathbb{E}[\int_{t}^{T}\frac{\lambda(t_{0})}{\lambda(t)}\langle
P(t_{0})z(t_{1}),z(t_{1})\rangle I_{E_{\rho}}(s)ds|\mathcal{F}_{t}]\}\\
&  \ \ \ \,+\{\mathbb{E}[\int_{t}^{T}\frac{\lambda(t_{0})}{\lambda(t)}\langle
P(t_{0})\zeta_{0},\zeta_{0}\rangle I_{E_{\rho}}(s)ds|\mathcal{F}%
_{t}]-\mathbb{E}[\int_{t}^{T}\frac{\lambda(s)}{\lambda(t)}\langle
P(s)\zeta_{0},\zeta_{0}\rangle I_{E_{\rho}}(s)ds|\mathcal{F}_{t}]\}\\
&  =:J_{1}+J_{2}+J_{3}.
\end{align*}
We only estimate $J_{2},$ and the other terms can be treated in the same way.
Still denote by $\alpha^{\prime}$ the H\"{o}lder conjugate of $\alpha$. Note
that
\begin{align*}
&  \mathbb{E}[\int_{t}^{T}|\frac{\lambda(t_{0})}{\lambda(t)}||\langle
P(t_{1})z(t_{1}),z(t_{1})\rangle-\langle P(t_{0})z(t_{1}),z(t_{1}%
)\rangle|I_{E_{\rho}}(s)ds|\mathcal{F}_{t}]\\
&  =\mathbb{E}[\int_{t}^{T}|\frac{\lambda(t_{0})}{\lambda(t)}\frac
{w(t_{1})-w({t_{0}})}{\sqrt{\rho}}||\langle P(t_{1})\zeta_{0},\zeta_{0}%
\rangle-\langle P(t_{0})\zeta_{0},\zeta_{0}\rangle|I_{E_{\rho}}%
(s)ds|\mathcal{F}_{t}]\\
&  \leq(\mathbb{E}[\int_{t}^{T}|\frac{\lambda(t_{0})}{\lambda(t)}\frac
{w(t_{1})-w({t_{0}})}{\sqrt{\rho}}|^{\alpha^{\prime}}I_{E_{\rho}%
}(s)ds|\mathcal{F}_{t}])^{\frac{1}{\alpha^{\prime}}}(\mathbb{E}[\int_{t}%
^{T}|\langle P(t_{1})\zeta_{0},\zeta_{0}\rangle-\langle P(t_{0})\zeta
_{0},\zeta_{0}\rangle|^{\alpha}I_{E_{\rho}}(s)ds|\mathcal{F}_{t}])^{\frac
{1}{\alpha}}\\
&  \leq C_{1}\rho^{\frac{1}{\alpha^{\prime}}}(\mathbb{E}[\int_{t}^{T}|\langle
P(t_{1})\zeta_{0},\zeta_{0}\rangle-\langle P(t_{0})\zeta_{0},\zeta_{0}%
\rangle)|^{\alpha}I_{E_{\rho}}(s)ds])^{\frac{1}{\alpha}}.
\end{align*}
Then by Proposition \ref{Myth2-7}, we have
\[
(\mathbb{E}[|J_{2}|^{\alpha}])^{\frac{1}{\alpha}}\leq C_{1}\rho^{\frac
{1}{\alpha^{\prime}}}(\mathbb{E}[\int_{t}^{T}|\langle P(t_{1})\zeta_{0}%
,\zeta_{0}\rangle-\langle P(t_{0})\zeta_{0},\zeta_{0}\rangle)|^{\alpha
}I_{E_{\rho}}(s)ds])^{\frac{1}{\alpha}}=o_{\zeta_{0}}(\rho).
\]
Thus,
\[
(\mathbb{E}[|I_{3}|^{\alpha}])^{\frac{1}{\alpha}}=o_{\zeta_{0}}(\rho).
\]

For the $I_{5}$ term, by a similar but simpler calculation,
\[
(\mathbb{E}[|I_{5}|^{\alpha}])^{\frac{1}{\alpha}}\leq C_{1}\rho^{2}%
=o_{\zeta_{0}}(\rho).
\]
Therefore,
\[
(\mathbb{E}[|\sigma^{t_{0},\zeta_{0}}(t)|^{\alpha}])^{\frac{1}{\alpha}%
}=o_{\zeta_{0}}(\rho).
\]

An approximation argument gives the result for the case of $\zeta_{0}\in
L^{4\alpha}(\mathcal{F}_{t_{0}},H)$. Indeed, for any $\delta>0,$ choose a
$\zeta_{0}^{\prime}\in L^{4\alpha}(\mathcal{F}_{t_{0}},V)$ such that
$\mathbb{E}[\Vert\zeta_{0}-\zeta_{0}^{\prime}\Vert_{H}^{4\alpha}]\leq\delta$
and let $x^{\prime}$ be the corresponding solution. Then
\begin{align*}
\sigma^{t_{0},\zeta_{0}}(t)  &  =\{\sigma^{t_{0},\zeta_{0}}(t)-\sigma
^{t_{0},\zeta_{0}^{\prime}}(t)\}+\sigma^{t_{0},\zeta_{0}^{\prime}}(t)\\
&  =\{\mathbb{E}[\frac{\lambda(T)}{\lambda(t)}\langle\xi x(T),x(T)\rangle
|\mathcal{F}_{t}]-\mathbb{E}[\frac{\lambda(T)}{\lambda(t)}\langle\xi
x^{\prime}(T),x^{\prime}(T)\rangle|\mathcal{F}_{t}]\}\\
&  \ \ \ \,+\{\mathbb{E}[\int_{t}^{T}\frac{\lambda(s)}{\lambda(t)}\langle
f(s,P(s))x(s),x(s)\rangle ds|\mathcal{F}_{t}]-\mathbb{E}[\int_{t}^{T}%
\frac{\lambda(s)}{\lambda(t)}\langle f(s,P(s))x^{\prime}(s),x^{\prime
}(s)\rangle ds|\mathcal{F}_{t}]\}\\
&  \ \ \ \,+\{\mathbb{E}[\int_{t}^{T}\frac{\lambda(s)}{\lambda(t)}\langle
P(s)\zeta_{0}^{\prime},\zeta_{0}^{\prime}\rangle I_{E_{\rho}}(s)ds|\mathcal{F}%
_{t}]-\mathbb{E}[\int_{t}^{T}\frac{\lambda(s)}{\lambda(t)}\langle
P(s)\zeta_{0},\zeta_{0}\rangle I_{E_{\rho}}(s)ds|\mathcal{F}_{t}]\}\\
&  \ \ \ \,+\{\langle P(t)x^{\prime}(t),x^{\prime}(t)\rangle-\langle
P(t)x(t),x(t)\rangle\}+\sigma^{t_{0},\zeta_{0}^{\prime}}(t)\\
&  =:K_{1}+K_{2}+K_{3}+K_{4}+\sigma^{t_{0},\zeta_{0}^{\prime}}(t).
\end{align*}
We only give the calculation of $K_{1}$, and the terms $K_{2}$, $K_{3},$
$K_{4}$ can be estimated similarly. From a similar analysis as for $I_{1}$, we
have for some constant $C_{2}$ independent of $\rho$ and $\zeta_{0}^{\prime}$
that
\begin{align*}
(\mathbb{E}[|K_{1}|^{\alpha}])^{\frac{1}{\alpha}}  &  \leq C_{2}%
(\mathbb{E}[\Vert x(T)-x^{\prime}(T)\Vert_{H}^{4\alpha}])^{\frac{1}{4\alpha}%
}\{(\mathbb{E}[\Vert x(T)\Vert_{H}^{4\alpha}])^{\frac{1}{4\alpha}}%
+(\mathbb{E}[\Vert x^{\prime}(T)\Vert_{H}^{4\alpha}])^{\frac{1}{4\alpha}}\}\\
&  \leq C_{2}(\mathbb{E}[\Vert\zeta_{0}-\zeta_{0}^{\prime}\Vert_{H}^{4\alpha
}])^{\frac{1}{4\alpha}}\rho.
\end{align*}
Therefore,
\[
(\mathbb{E}[|\sigma^{t_{0},\zeta_{0}}(t)|^{\alpha}])^{\frac{1}{\alpha}}\leq
C_{2}\delta^{\frac{1}{4\alpha}}\rho+o_{\zeta_{0}^{\prime}}(\rho),
\]
which can be written as
\[
\frac{1}{\rho}(\mathbb{E}[|\sigma^{t_{0},\zeta_{0}}(t)|^{\alpha}])^{\frac
{1}{\alpha}}\leq C_{2}\delta^{\frac{1}{4\alpha}}+o_{\zeta_{0}^{\prime}}(1).
\]
Letting $\rho\rightarrow0$ and utilizing the arbitrariness of $\delta$, we
obtain
\[
(\mathbb{E}[|\sigma^{t_{0},\zeta_{0}}(t)|^{\alpha}])^{\frac{1}{\alpha}}%
=o(\rho).
\]

\textit{Step 3: the general $\zeta$.} Let $x^{t_{0}}$ be the solution of SEE
(\ref{Myeq3-9}) corresponds to $\zeta^{\prime}$ satisfying $\zeta^{\prime
}(s)=\zeta(t_{0}),$ $s\geq t_{0},$ for each $t_{0}\in\lbrack0,T].$ From the
Lebesgue differentiation theorem (see also \cite[Theorem 2.2.9]{DU77}), we
have (for a.e. $t_{0}$)%
\[
\frac{1}{\rho}\int_{0}^{T}\mathbb{E}[\Vert\zeta(s)-\zeta(t_{0})\Vert
_{H}^{4\alpha}]I_{E_{\rho}}(s)ds=0,\quad\text{as}\ \rho\rightarrow0.
\]
From this we also get%
\[
\frac{1}{\rho^{2\alpha}}\mathbb{E}[\sup_{t\in\lbrack0,T]}\Vert x(t)-x^{t_{0}%
}(t)\Vert_{H}^{4\alpha}]\leq C\frac{1}{\rho}\int_{0}^{T}\mathbb{E}[\Vert
\zeta(s)-\zeta(t_{0})\Vert_{H}^{4\alpha}]I_{E_{\rho}}(s)ds=0,\quad
\text{as}\ \rho\rightarrow0.
\]
Therefore,%
\[
\int_{0}^{T}\mathbb{E}[\Vert\zeta(s)-\zeta(t_{0})\Vert_{H}^{4\alpha
}]I_{E_{\rho}}(s)ds=o(\rho)\quad\text{and}\quad\mathbb{E}[\sup_{t\in
\lbrack0,T]}\Vert x(t)-x^{t_{0}}(t)\Vert_{H}^{4\alpha}]=o(\rho^{2\alpha}).
\]
Noting that%
\begin{align*}
\sigma(t)  &  =\{\sigma(t)-\sigma^{t_{0},\zeta(t_{0})}(t)\}+\sigma
^{t_{0},\zeta(t_{0})}(t)\\
&  =\{\mathbb{E}[\frac{\lambda(T)}{\lambda(t)}\langle\xi x(T),x(T)\rangle
|\mathcal{F}_{t}]-\mathbb{E}[\frac{\lambda(T)}{\lambda(t)}\langle\xi x^{t_{0}%
}(T),x^{t_{0}}(T)\rangle|\mathcal{F}_{t}]\}\\
&  \ \ \ \,+\{\mathbb{E}[\int_{t}^{T}\frac{\lambda(s)}{\lambda(t)}\langle
f(s,P(s))x(s),x(s)\rangle ds|\mathcal{F}_{t}]-\mathbb{E}[\int_{t}^{T}%
\frac{\lambda(s)}{\lambda(t)}\langle f(s,P(s))x^{t_{0}}(s),x^{t_{0}}(s)\rangle
ds|\mathcal{F}_{t}]\}\\
&  \ \ \ \,+\{\mathbb{E}[\int_{t}^{T}\frac{\lambda(s)}{\lambda(t)}\langle
P(s)\zeta(t_{0}),\zeta(t_{0})\rangle I_{E_{\rho}}(s)ds|\mathcal{F}%
_{t}]-\mathbb{E}[\int_{t}^{T}\frac{\lambda(s)}{\lambda(t)}\langle
P(s)\zeta(s),\zeta(s)\rangle I_{E_{\rho}}(s)ds|\mathcal{F}_{t}]\}\\
&  \ \ \ \,+\{\langle P(t)x^{t_{0}}(t),x^{t_{0}}(t)\rangle-\langle
P(t)x(t),x(t)\rangle\}+\sigma^{t_{0},\zeta(t_{0})}(t),
\end{align*}
we can deduce by a similar analysis as in Step 2 that%
\[
(\mathbb{E}[|\sigma(t)|^{\alpha}])^{\frac{1}{\alpha}}\leq(\mathbb{E}%
[|\sigma(t)-\sigma^{t_{0},\zeta(t_{0})}(t)|^{\alpha}])^{\frac{1}{\alpha}%
}+(\mathbb{E}[|\sigma^{t_{0},\zeta(t_{0})}(t)|^{\alpha}])^{\frac{1}{\alpha}%
}=o(\rho).
\]

\end{proof}

\begin{lemma}
\label{Myle4-1} Suppose $(H4)$ holds. For $\mu_{1},\mu_{2}\in L_{\mathbb{F}%
}^{\infty}(0,T)$, define
\[
\tilde{A}(t):=A(t)+\mu_{1}(t)B(t)+\mu_{2}(t)I_{d},\quad\tilde{B}%
(t):=B(t)+\mu_{1}(t)I_{d}
\]
and
\[
\lambda_{1}(t):=e^{\int_{0}^{t}[\mu_{2}(s)-\frac12(\mu_{1}(s))^{2} ]ds+\mu
_{1}(s)dw({s})}.
\]
Then%
\[
L_{\tilde{A},\tilde{B}}(t,s)=\frac{\lambda_{1}(s)}{\lambda_{1}(t)}%
L_{A,B}(t,s),\quad\text{for}\ 0\leq t\leq s\leq T.
\]

\end{lemma}

\begin{proof}
For any $u\in L^{2}(\mathcal{F}_{t},H),$ $\{L_{A,B}(t,s)u\}_{t\leq s\leq T}$
solves the SEE (\ref{Myeq2-17}) with initial value $u.$ Then by It\^{o}'s
formula, we see that the process $\{\frac{\lambda_{1}(s)}{\lambda_{1}%
(t)}L_{A,B}(t,s)u\}_{t\leq s\leq T}$ is the solution of SEE (\ref{Myeq2-17})
with unbounded operators $\tilde{A}$, $\tilde{B}$ and initial value $u$. Thus
$\frac{\lambda_{1}(s)}{\lambda_{1}(t)}L_{A,B}(t,s)u=L_{\tilde{A},\tilde{B}%
}(t,s)u$ and the proof is complete.
\end{proof}

\begin{proof}
[Proof of Theorem \ref{Myth3-7}]According to Theorem \ref{Myito2-12}, we have
\begin{align*}
\langle P(t)x(t),x(t)\rangle+\sigma(t)  &  =\mathbb{E}[\frac{\lambda
(T)}{\lambda(t)}\langle\xi x(T),x(T)\rangle+\int_{t}^{T}\frac{\lambda
(s)}{\lambda(t)}\langle f(s,P(s))x(s),x(s)\rangle ds\\
&  \ \ \ \,-\int_{t}^{T}\frac{\lambda(s)}{\lambda(t)}\langle P(s)\zeta
(s),\zeta(s)\rangle I_{E_{\rho}}(s)ds|\mathcal{F}_{t}]
\end{align*}
with $\sigma$ and $\lambda$ being defined by (\ref{Myeq4-16}) and
(\ref{Myeq4-29}) respectively, and $\sigma$ satisfying (\ref{Myeq2-16}). This
is in fact the explicit formula of the linear BSDE (\ref{Myeq3-5}) with
solution $(\langle P(t)x(t),x(t)\rangle+\sigma(t),\mathcal{Z}(t))\in
L_{\mathbb{F}}^{\alpha}(0,T)\times L_{\mathbb{F}}^{2,\alpha}(0,T)$. The
uniqueness of $(\sigma,\mathcal{Z)}$ in the equation (\ref{Myeq3-5}) and the
estimate (\ref{Myeq2-20}) follow directly from the basic theory of BSDEs.
\end{proof}

Now it remains to prove Proposition \ref{Myle4-7}. We shall need an a priori
estimate of SEEs when the non-homogeneous term $a$ in the drift taking values
in $V^{\ast}$. It is worth to mention that if particularly $a$ takes values in
$H$, we can in fact have a better version for such kind of estimate (see
(\ref{SEE-estimate})).

\begin{lemma}
\label{Myth4-4} 
Assume $(H4)$ holds. For any given  $(a,b)\in L_{\mathbb{F}}%
^{2,2\alpha}(t,T;V^{\ast}\times H)$ and $z_{0}\in L^{2\alpha}(\mathcal{F}%
_{t},H)$ with $\alpha\geq1$, denote by $z$ the solution of
\[%
\begin{cases}
{d}z(s) & =[A(s)z(s)+a(s)]{d}s+[B(s)z(s)+b(s)]{d}w(s),\quad s\in\lbrack
t,T],\\
z(t) & =z_{0}.
\end{cases}
\]
Then there is a constant $C>0$ depending on $\delta$, $K$ and $\alpha$
such that
\[
\mathbb{E}[\sup_{s\in\lbrack t,T]}\left\Vert z(s)\right\Vert _{H}^{2\alpha
}]\leq C\,\mathbb{E}[\Vert z_{0}\Vert_{H}^{2\alpha}+(\int_{t}^{T}\Vert
a(s)\Vert_{V^{\ast}}^{2}ds)^{\alpha}+(\int_{t}^{T}\Vert b(s)\Vert_{H}%
^{2}ds)^{\alpha}].
\]
\end{lemma}

\begin{proof}
The proof is a variant of the one for (\ref{SEE-estimate}) in \cite{DM13}. We
only present the case of $t=0$, and the other cases can be proved in a similar
way. By the coercivity condition,
\begin{equation}
\Vert Bu\Vert_{H}\leq C(K)\Vert u\Vert_{V},\text{ for }u\in V. \label{Myeq4-5}%
\end{equation}
Then,
\begin{align*}
&  2\langle Az(t)+a(t),z(t)\rangle_{\ast}+\Vert Bz(t)+b(t)\Vert_{H}^{2}\\
&  \leq2\langle Az(t),z(t)\rangle_{\ast}+\Vert Bz(t)\Vert_{H}^{2}+2\langle
Bz(t),b(t)\rangle+\Vert b(t)\Vert_{H}^{2}+2\langle a(t),z(t)\rangle_{\ast}\\
&  \leq-\delta\Vert z(t)\Vert_{V}^{2}+K\Vert z(t)\Vert_{H}^{2}+C\Vert
z(t)\Vert_{V}\Vert b(t)\Vert_{H}+\Vert b(t)\Vert_{H}^{2}+2\Vert a(t)\Vert
_{V^{\ast}}\Vert z(t)\Vert_{V}\\
&  \leq-\delta\Vert z(t)\Vert_{V}^{2}+K\Vert z(t)\Vert_{H}^{2}+\frac{\delta
}{2}\Vert z(t)\Vert_{V}+C(\delta)\Vert b(t)\Vert_{H}^{2}+C(\delta)\Vert
a(t)\Vert_{V^{\ast}}^{2}\\
&  \leq C(\delta,K)(\Vert z(t)\Vert_{H}^{2}+\Vert a(t)\Vert_{V^{\ast}}
^{2}+\Vert b(t)\Vert_{H}^{2})
\end{align*}
and
\[
|\langle Bz(t)+b(t),z(t)\rangle|^{2}\leq2|\langle Bz(t),z(t)\rangle
|^{2}+2|\langle b(t),z(t)\rangle|^{2}\leq2K^{2}\Vert z(t)\Vert_{H}^{4}+2\Vert
b(t)\Vert_{H}^{2}\Vert z(t)\Vert_{H}^{2}
\]
Let $\varepsilon>0$ and $\gamma>0$ be undetermined. We have by the H\"{o}lder
inequality and the Young's inequality that
\begin{align*}
\mathbb{E}[\int_{0}^{T}e^{-\gamma t}\Vert z(t)\Vert_{H}^{2(\alpha-1)}\Vert
a(t)\Vert_{V^{\ast}}^{2}{d}t]  &  \leq\varepsilon^{2}\mathbb{E[}\sup
_{t\in\lbrack0,T]}e^{-\gamma t}\Vert z(t)\Vert_{H}^{2\alpha}]+C(\varepsilon
)\mathbb{E[}(\int_{0}^{T}e^{-\frac{\gamma t}{\alpha}}\Vert a(t)\Vert_{V^{\ast
}}^{2}{d}t)\,^{\alpha}]\\
&  \leq\varepsilon^{2}\mathbb{E[}\sup_{t\in\lbrack0,T]}e^{-\gamma t}\Vert
z(t)\Vert_{H}^{2\alpha}]+C(\varepsilon)\mathbb{E[}(\int_{0}^{T}\Vert
a(t)\Vert_{V^{\ast}}^{2}{d}t)\,^{\alpha}],
\end{align*}
and similarly,
\[
\mathbb{E}[\int_{0}^{T}e^{-\gamma t}\Vert z(t)\Vert_{H}^{2(\alpha-1)}\Vert
b(t)\Vert_{H}^{2}){d}t]\leq\varepsilon^{2}\mathbb{E[}\sup_{t\in\lbrack
0,T]}e^{-\gamma t}\Vert z(t)\Vert_{H}^{2\alpha}]+C(\varepsilon)\mathbb{E[}
(\int_{0}^{T}\Vert b(t)\Vert_{H}^{2}{d}t)\,^{\alpha}].
\]
In the sequel of this proof, for the sake of notation simplicity, we use
$C_{1}$ to denote a generic constant independent of $\varepsilon$ and $\gamma
$, which may be different from line to line. From the quasi-skew-symmetry
condition, we can calculate
\begin{align*}
&  \mathbb{E}[\sup_{t\in\lbrack0,T]}|\int_{0}^{t}e^{-\gamma s}\Vert
z(s)\Vert_{H}^{2(\alpha-1)}\langle Bz(s)+b(s),z(s)\rangle\,{d}w(s)|]\\
&  \leq C_{1}\mathbb{E[(}\int_{0}^{T}e^{-2\gamma t}\Vert z(t)\Vert
_{H}^{4\alpha-4}|\langle Bz(t)+b(t),z(t)\rangle|^{2}\,{d}t)^{\frac{1}{2}}]\\
&  \leq C_{1}\mathbb{E[}\sup_{t\in\lbrack0,T]}e^{-\frac{\gamma t}{2}}\Vert
z(t)\Vert_{H}^{\alpha}\mathbb{(}\int_{0}^{T}e^{-\gamma t}(\Vert z(t)\Vert
_{H}^{2\alpha}\,+\Vert z(t)\Vert_{H}^{2\alpha-2}\Vert b(t)\Vert_{H}^{2}
){d}t)^{\frac{1}{2}}]\\
&  \leq\varepsilon\mathbb{E[}\sup_{t\in\lbrack0,T]}e^{-\gamma t}\Vert
z(t)\Vert_{H}^{2\alpha}]+\frac{C_{1}}{\varepsilon}\mathbb{E[}\int_{0}
^{T}e^{-\gamma t}(\Vert z(t)\Vert_{H}^{2\alpha}\,+\Vert z(t)\Vert_{H}
^{2\alpha-2}\Vert b(t)\Vert_{H}^{2}){d}t]\\
&  \leq C_{1}\varepsilon\mathbb{E[}\sup_{t\in\lbrack0,T]}e^{-\gamma t}\Vert
z(t)\Vert_{H}^{2\alpha}]+\frac{C_{1}}{\varepsilon}\mathbb{E[}\int_{0}
^{T}e^{-\gamma t}\Vert z(t)\Vert_{H}^{2\alpha}\,{d}t]+C(\varepsilon
)\mathbb{E[}(\int_{0}^{T}\Vert b(t)\Vert_{H}^{2}{d}t)\,^{\alpha}].
\end{align*}
Then applying It\^{o} formula to $e^{-\gamma t}\Vert z(t)\Vert_{H}^{2\alpha}$,
we obtain
\begin{align*}
&  e^{-\gamma t}\Vert z(t)\Vert_{H}^{2\alpha}+\gamma\int_{0}^{t}e^{-\gamma
s}\Vert z(s)\Vert_{H}^{2\alpha}\,{d}s\\
&  =\Vert z_{0}\Vert_{H}^{2\alpha}+\alpha\int_{0}^{t}e^{-\gamma s}\Vert
z(s)\Vert_{H}^{2(\alpha-1)}(2\langle Az(s)+a(s),z(s)\rangle_{\ast}+\Vert
Bz(s)+b(s)\Vert_{H}^{2})\,{d}s\\
&  +2\alpha(\alpha-1)\int_{0}^{t}e^{-\gamma s}\Vert z(s)\Vert_{H}
^{2(\alpha-2)}|\langle Bz(s)+b(s),z(s)\rangle|^{2}\,{d}s\\
&  +2\alpha\int_{0}^{t}e^{-\gamma s}\Vert z(s)\Vert_{H}^{2(\alpha-1)}\langle
Bz(s)+b(s),z(s)\rangle\,{d}w(s)\\
&  \leq\Vert z_{0}\Vert_{H}^{2}+C_{1}\int_{0}^{t}e^{-\gamma s}\Vert
z(s)\Vert_{H}^{2(\alpha-1)}(\Vert z(s)\Vert_{H}^{2}+\Vert a(s)\Vert_{V^{\ast}
}^{2}+\Vert b(s)\Vert_{H}^{2})\,{d}s\\
&  +C_{1}\int_{0}^{t}e^{-\gamma s}\Vert z(s)\Vert_{H}^{2(\alpha-2)}(\Vert
z(s)\Vert_{H}^{4}+\Vert b(s)\Vert_{H}^{2}\Vert z(s)\Vert_{H}^{2})\,{d}s\\
&  +2\alpha\int_{0}^{t}e^{-\gamma s}\Vert z(s)\Vert_{H}^{2(\alpha-1)}\langle
Bz(s)+b(s),z(s)\rangle\,{d}w(s)
\end{align*}
Taking supremum and expectation on both sides, we get
\begin{align*}
&  \mathbb{E}[\sup_{t\in\lbrack0,T]}e^{-\gamma t}\Vert z(t)\Vert_{H}^{2\alpha
}]+\gamma\mathbb{E}[\int_{0}^{T}e^{-\gamma t}\Vert z(t)\Vert_{H}^{2\alpha
}\,{d}t]\\
&  \leq C_{1}(\varepsilon+\varepsilon^{2})\mathbb{E[}\sup_{t\in\lbrack
0,T]}e^{-\gamma t}\Vert z(t)\Vert_{H}^{2\alpha}]+\mathbb{E}[\Vert z_{0}
\Vert_{H}^{2\alpha}]+C(\varepsilon)\mathbb{E}[\int_{0}^{t}e^{-\gamma t}\Vert
z(t)\Vert_{H}^{2\alpha}\,{d}s]\\
&  +C(\varepsilon)\mathbb{E}[\int_{0}^{T}(\Vert a(t)\Vert_{V^{\ast}}
^{2})\,^{\alpha}{d}t]+C(\varepsilon)\mathbb{E}[\int_{0}^{T}(\Vert
b(t)\Vert_{H}^{2})\,^{\alpha}{d}t]
\end{align*}
Choosing $\varepsilon$ small and $\gamma$ large, we obtain
\[
\mathbb{E}[\sup_{t\in\lbrack0,T]}\Vert z(t)\Vert_{H}^{2\alpha}]\leq
C\,\mathbb{E}[\Vert z_{0}\Vert_{H}^{2\alpha}+(\int_{0}^{T}\Vert a(t)\Vert
_{V^{\ast}}^{2}dt)^{\alpha}+(\int_{0}^{T}\Vert b(t)\Vert_{H}^{2}dt)^{\alpha
}].
\]
The proof is complete.
\end{proof}

\begin{proof}
[Proof of Proposition \ref{Myle4-7}]On $[t_{0},t_{0}+\rho],$ we denote
$\delta(t):=y(t)-\sqrt{\rho}\eta(t)$ and have%
\[
{d}\delta(t)=[A\delta(t)+\sqrt{\rho}A\eta(t)]\,{d}t+[B\delta(t)+\sqrt{\rho
}B\eta(t)]{d}w(t),\quad\delta(t_{0})=0.
\]
Note that from the coercivity condition,
\[
\Vert Bu\Vert_{H}\leq C(K)\Vert u\Vert_{V},\quad\text{for}\ u\in V.
\]
Then according to Lemma \ref{Myth4-4},
\begin{align*}
\mathbb{E}[\sup_{[t_{0},t_{0}+\rho]}\left\Vert \delta(t)\right\Vert
_{H}^{2\alpha}]  &  \leq C\rho^{\alpha}\mathbb{E}[(\int_{t_{0}}^{t_{0}+\rho
}\Vert A\eta(t)\Vert_{V^{\ast}}^{2}dt)^{\alpha}+(\int_{t_{0}}^{t_{0}+\rho
}\Vert B\eta(t)\Vert_{H}^{2}dt)^{\alpha}]\\
&  \leq C\rho^{\alpha}\mathbb{E}[(\int_{t_{0}}^{t_{0}+\rho}\Vert\eta
(t)\Vert_{V}^{2}dt)^{\alpha}]\\
&  =C\rho^{2\alpha-1}\int_{t_{0}}^{t_{0}+\rho}\mathbb{E}[\Vert\eta(t)\Vert
_{V}^{2\alpha}]dt\\
&  \leq C\,\mathbb{E}[\left\Vert \zeta_{0}\right\Vert _{V}^{2\alpha}%
]\rho^{2\alpha}.
\end{align*}
We also note that
\[
\mathbb{E}[\sup_{t\in\lbrack0,t_{0}]}\left\Vert y(t)-\sqrt{\rho}%
z(t)\right\Vert _{H}^{2\alpha}]=0,
\]
and from the basic estimate of SEEs,
\[
\mathbb{E}[\sup_{t\in\lbrack t_{0}+\rho,T]}\Vert y(t)-\sqrt{\rho}z(t)\Vert
_{H}^{2\alpha}]\leq C\mathbb{E}[\Vert y(t_{0}+\rho)-\sqrt{\rho}\eta(t_{0}%
+\rho)\Vert_{H}^{2\alpha}]\leq C\,\mathbb{E}[\left\Vert \zeta_{0}\right\Vert
_{V}^{2\alpha}]\rho^{2\alpha}.
\]
Combining the above analysis, we obtain the desired result.
\end{proof}

\subsection{Proof of the $L^\beta$-estimate (\ref{Myeq4-22}) of adjoint equations}
We shall give a general result for possible future  applications. We also note that the case of $\beta=2$ for the first-order equation has already proved in \cite{DM10}.

We consider the following backward stochastic evolution equation (BSEE)
\begin{equation}
	\left\{
	\begin{aligned}
		-dp(t)=  &  [\mathcal{M}(t)p(t)+\mathcal{N}(t)q(t)+f(p(t),q(t),t)]dt\\
		&  -q(t)dw(t),\quad t\in\lbrack0,T],\\
		p(T)=  &  \xi,
	\end{aligned}
	\right.  \label{bsee-11}%
\end{equation}
where $\xi$ is the terminal condition,
\[
\mathcal{M}:[0,T]\times\Omega\rightarrow\mathcal{L}(V,V^{\ast}%
),\,\,\,\mathcal{N}:[0,T]\times\Omega\rightarrow\mathcal{L}(H,V^{\ast})
\]
are unbounded operators and
\[
f:[0,T]\times\Omega\times H\times H\rightarrow H
\]
is a nonlinear function.

Given $\beta\geq2.$ We denote by $L_{\mathbb{F}}^{1,\beta}(0,T;H)$) the space
of $H$-valued progressively measurable processes $y(\cdot)$ with norm $\Vert
y\Vert_{L_{\mathbb{F}}^{1,\beta}(0,T;H)}=\{\mathbb{\mathbb{E}}[(\int_{0}%
^{T}\Vert y(t)\Vert_{H}{d}t)^{\beta}]\}^{\frac{1}{\beta}}$.

We impose the following assumptions.
\begin{description}
	\item[$(A)$] For each $u\in V,$ $\mathcal{M}(t,\omega)u$ and $\mathcal{N}
	(t,\omega)u$ are progressively measurable. There exist some constants
	$\delta>0$ and $K\geq0$ such that the following two assertions hold: for each
	$(t,\omega)\in\lbrack0,T]\times\Omega$ and $x\in V$,
	\begin{description}
		\item  [\rm{{{(1)}}}]
		 Coercivity condition:
		\[
		2\left\langle \mathcal{M}(t)x,x\right\rangle _{\ast}+\Vert\mathcal{N}^{\ast
		}(t)x\Vert_{H}^{2}\leq-\delta\Vert x\Vert_{V}^{2}+K\Vert x\Vert_{H}%
		^{2}\ \text{and }\Vert\mathcal{M}(t)x\Vert_{V^\ast}\leq K\Vert x\Vert_{V};
		\]

		\item[\rm{{{(2)}}}] For each $(p,q)\in H\times H,$ $f(\cdot,\cdot,p,q)$
		are progressively measurable. $f(\cdot,\cdot,0,0)\in L_{\mathbb{F}}^{1,\beta
		}(0,T;H)$, $\xi\in L^{\beta}(\mathcal{F}_{T},H)$, and
		\[
		\Vert f(t,p,q)-f(t,p^{\prime},q^{\prime})\Vert_{H}\leq K(\Vert p-p^{\prime
		}\Vert_{H}+\Vert q-q^{\prime}\Vert_{H}).
		\]
		
	\end{description}
\end{description}

\begin{lemma}
	\label{Le-11} Assume the condition $(A)$. If $(p(\cdot),q(\cdot))$ is the
	solution to BSEE (\ref{bsee-11}), then there exists some positive constant $C$
	depending on $\delta$ and $K$ that
	\begin{align*}
		&  \mathbb{E}[\sup_{t\in\lbrack0,T]}\left\Vert p(t)\right\Vert _{H}^{\beta
		}]+\mathbb{E}[(\int_{0}^{T}\left\Vert p(t)\right\Vert _{V}^{2}dt)^{\frac
			{\beta}{2}}]+\mathbb{E}[(\int_{0}^{T}\left\Vert q(t)\right\Vert _{H}
		^{2}dt)^{\frac{\beta}{2}}]\\
		&  \ \ \ \ \ \ \ \ \leq C\{\mathbb{E}[\Vert\xi\Vert_{H}^{\beta}]+\mathbb{E(}%
		\int_{0}^{T}\Vert f(t,0,0)\Vert_{H}dt)^{\beta}\}.
	\end{align*}
	
\end{lemma}

\begin{proof}
In the proof, we use $C>0$ to denote a generic constant that may change from
line to line. Applying the It\^{o} formula to $\Vert p(t)\Vert_{H}^{2},$ we
have
\begin{align*}
	\Vert p(t)\Vert_{H}^{2}+\int_{t}^{T}\Vert q(s)\Vert_{H}^{2}ds &  =\Vert
	\xi\Vert_{H}^{2}+2\int_{t}^{T}[\left\langle \mathcal{M}%
	(s)p(s),p(s)\right\rangle _{\ast}+\left\langle \mathcal{N}%
	(s)q(s),p(s)\right\rangle _{\ast}\\
	&  +\left\langle f(s,p(s),q(s)),p(s)\right\rangle _{H}]ds-2\int_{t}%
	^{T}\left\langle q(s),p(s)\right\rangle _{H}dw(s).
\end{align*}
Applying again the It\^{o} formula to $\Vert
p(t)\Vert_{H}^{\beta}=(\Vert p(t)\Vert_{H}^{2})^{\frac{\beta}{2}}$, we get%
\begin{align*}
	&  \Vert p(t)\Vert_{H}^{\beta}+\frac{1}{2}\beta\int_{t}^{T}\Vert p(s)\Vert
	_{H}^{\beta-2}\Vert q(s)\Vert_{H}^{2}dt+\int_{t}^{T}\beta(\frac{\beta}%
	{2}-1)\Vert p(s)\Vert_{H}^{\beta-4}|\left\langle p(s),q(s)\right\rangle
	_{H}|^{2}ds\\
	&  =\Vert\xi\Vert_{H}^{\beta}+\int_{t}^{T}\beta\Vert p(s)\Vert_{H}^{\beta
		-2}[\left\langle \mathcal{M}(s)p(s),p(s)\right\rangle _{\ast}+\left\langle
	\mathcal{N}(s)q(s),p(s)\right\rangle _{\ast}+\left\langle
	p(s),f(s,p(s),q(s))\right\rangle ]ds\\
	&  \ \ \ \,-\beta\int_{t}^{T}\Vert p(s)\Vert_{H}^{\beta-2}\left\langle
	p(s),q(s)\right\rangle _{H}dw(s).
\end{align*}
Making use of the coercivity condition, we obtain for some undetermined
$\varepsilon>0$ that
\begin{align*}
	&  \Vert p(t)\Vert_{H}^{\beta}+\frac{1}{2}\beta\int_{t}^{T}\Vert p(s)\Vert
	_{H}^{\beta-2}\Vert q(s)\Vert_{H}^{2}dt+\int_{t}^{T}\beta(\frac{\beta}%
	{2}-1)\Vert p(s)\Vert_{H}^{\beta-4}|\left\langle p(s),q(s)\right\rangle
	_{H}|^{2}ds\\
	&  \leq\Vert\xi\Vert_{H}^{\beta}+\int_{t}^{T}\frac{\beta}{2}\Vert
	p(s)\Vert_{H}^{\beta-2}[2\varepsilon K\Vert p(s)\Vert_{V}^{2}+(1+\varepsilon
	)(-\delta\Vert p(s)\Vert_{V}^{2}+K\Vert p(s)\Vert_{H}^{2})+\frac
	{1}{1+\varepsilon}\Vert q(s)\Vert_{H}^{2}\\
	&  \ \ \ \,+\frac{\varepsilon}{2}\Vert q(s)\Vert_{H}^{2}+C_{\varepsilon}\Vert
	p(s)\Vert_{H}^{2}]ds+\beta\sup_{s\in\lbrack0,T]}\Vert p(s)\Vert_{H}^{\beta
		-1}\int_{t}^{T}\Vert f(s,0,0)\Vert_{H}ds\\
	&  \ \ \ \,-\beta\int_{t}^{T}\Vert p(s)\Vert_{H}^{\beta-2}\left\langle
	p(s),q(s)\right\rangle _{H}dw(s)\\
	&  \leq\Vert\xi\Vert_{H}^{\beta}+\int_{t}^{T}\frac{\beta}{2}\Vert
	p(s)\Vert_{H}^{\beta-2}[(2\varepsilon K-(1+\varepsilon)\delta)\Vert
	p(s)\Vert_{V}^{2}+C_{\varepsilon}\Vert p(s)\Vert_{H}^{2}+\frac{1+\frac
		{\varepsilon}{2}+\frac{\varepsilon^{2}}{2}}{1+\varepsilon}\Vert q(s)\Vert
	_{H}^{2}]ds\\
	&  \ \ \ \,+\beta\sup_{s\in\lbrack0,T]}\Vert p(s)\Vert_{H}^{\beta-1}\int%
	_{t}^{T}\Vert f(s,0,0)\Vert_{H}ds-\beta\int_{t}^{T}\Vert p(s)\Vert_{H}%
	^{\beta-2}\left\langle p(s),q(s)\right\rangle _{H}dw(s).
\end{align*}
Choose $\varepsilon$ small enough so that $(2\varepsilon K-(1+\varepsilon
)\delta)<0$ and $\frac{1+\frac{\varepsilon}{2}+\frac{\varepsilon^{2}}{2}%
}{1+\varepsilon}<1$, we get
\begin{equation}%
	\begin{split}
		&  \Vert p(t)\Vert_{H}^{\beta}+\int_{t}^{T}\Vert p(s)\Vert_{H}^{\beta-2}\Vert
		q(s)\Vert_{H}^{2}ds+\int_{t}^{T}\Vert p(s)\Vert_{H}^{\beta-2}\Vert
		p(s)\Vert_{V}^{2}ds\leq C[\Vert\xi\Vert_{H}^{\beta}+\int_{t}^{T}\Vert
		p(s)\Vert_{H}^{\beta}ds\\
		&  \ \ \ \ \ \ \ \,+\sup_{s\in\lbrack0,T]}\Vert p(s)\Vert_{H}^{\beta-1}%
		\int_{t}^{T}\Vert f(s,0,0)\Vert_{H}ds]-C_{\beta}\int_{t}^{T}\Vert
		p(s)\Vert_{H}^{\beta-2}\left\langle p(s),q(s)\right\rangle _{H}dw(s).
	\end{split}
	\label{estimate-0}%
\end{equation}
Taking expectation on both sides, we obtain (from standard truncation
techniques, the stochastic
integral above can be assumed to be a martingale;  see the proof of Theorem 4.4.4 in \cite{Zhang17})
\begin{equation}
	\mathbb{E}[\Vert p(t)\Vert_{H}^{\beta}]+\mathbb{E[}\int_{t}^{T}\Vert
	p(s)\Vert_{H}^{\beta-2}\Vert q(s)\Vert_{H}^{2}ds]\leq C\mathbb{E[}\Vert
	\xi\Vert_{H}^{\beta}+\int_{t}^{T}\Vert p(s)\Vert_{H}^{\beta}ds+\sup
	_{t\in\lbrack0,T]}\Vert p(t)\Vert_{H}^{\beta-1}\int_{0}^{T}\Vert
	f(t,0,0)\Vert_{H}dt].\label{eq0-2}%
\end{equation}
Applying the Gronwall inequality, we obtain%
\[
\mathbb{E}[\Vert p(t)\Vert_{H}^{\beta}]\leq C\mathbb{E[}\Vert\xi\Vert
_{H}^{\beta}+\sup_{t\in\lbrack0,T]}\Vert p(t)\Vert_{H}^{\beta-1}\int_{0}%
^{T}\Vert f(t,0,0)\Vert_{H}dt].
\]
Plugging this back into (\ref{eq0-2}), we get%
\[
\mathbb{E}[\Vert p(t)\Vert_{H}^{\beta}]+\mathbb{E[}\int_{0}^{T}\Vert
p(t)\Vert_{H}^{\beta-2}\Vert q(t)\Vert_{H}^{2}dt]\leq C\mathbb{E[}\Vert
\xi\Vert_{H}^{\beta}+\sup_{t\in\lbrack0,T]}\Vert p(t)\Vert_{H}^{\beta-1}%
\int_{0}^{T}\Vert f(t,0,0)\Vert_{H}dt].
\]
Then by the Young's inequality, we obtain that for an undetermined $\delta>0$
that
\begin{equation}%
	\begin{split}
		&  \mathbb{E}[\Vert p(t)\Vert_{H}^{\beta}]+\mathbb{E[}\int_{0}^{T}\Vert
		p(t)\Vert_{H}^{\beta-2}\Vert q(t)\Vert_{H}^{2}dt]\\
		&  \leq C\mathbb{E[}\Vert\xi\Vert_{H}^{\beta}]+\delta\mathbb{E[}\sup
		_{t\in\lbrack0,T]}\Vert p(t)\Vert_{H}^{\beta}]+C_{\delta}\mathbb{E[}(\int%
		_{0}^{T}\Vert f(t,0,0)\Vert_{H}dt)^{\beta}].
	\end{split}
	\label{eq-11}%
\end{equation}
On the other hand, taking supremum and expectation on both sides of
(\ref{estimate-0}), we have
\begin{align*}
	\mathbb{E[}\sup_{t\in\lbrack0,T]}\Vert p(t)\Vert_{H}^{\beta}] &  \leq
	C\mathbb{E[}\Vert\xi\Vert_{H}^{\beta}+\int_{0}^{T}\Vert p(t)\Vert_{H}^{\beta
	}dt]+C\mathbb{E[}(\int_{0}^{T}\Vert f(t,0,0)\Vert_{H}dt)^{\beta}]+\frac{1}%
	{4}\mathbb{E[}\sup_{t\in\lbrack0,T]}\Vert p(t)\Vert_{H}^{\beta}]\\
	&  \ \ \ \,+C\mathbb{E}[\sup_{t\in\lbrack0,T]}|\int_{t}^{T}\Vert p(s)\Vert
	_{H}^{\beta-2}\left\langle p(s),q(s)\right\rangle _{H}dw(s)|].\\
	&  \leq C\mathbb{E[}\Vert\xi\Vert_{H}^{\beta}+\int_{0}^{T}\Vert p(t)\Vert
	_{H}^{\beta}dt]+C\mathbb{E[}(\int_{0}^{T}\Vert f(t,0,0)\Vert_{H}dt)^{\beta
	}]+\frac{1}{4}\mathbb{E[}\sup_{t\in\lbrack0,T]}\Vert p(t)\Vert_{H}^{\beta}]\\
	&  \ \ \ \,+C\mathbb{E[(}\int_{0}^{T}\Vert p(t)\Vert_{H}^{2\beta-2}\Vert
	q(t)\Vert_{H}^{2}dt)^{\frac{1}{2}}]\\
	&  \leq C\mathbb{E[}\Vert\xi\Vert_{H}^{\beta}+\int_{0}^{T}\Vert p(t)\Vert
	_{H}^{\beta}dt]+C\mathbb{E[}(\int_{0}^{T}\Vert f(t,0,0)\Vert_{H}dt)^{\beta
	}]+\frac{1}{4}\mathbb{E[}\sup_{t\in\lbrack0,T]}\Vert p(t)\Vert_{H}^{\beta}]\\
	&  \ \ \ \,+C\mathbb{E[}\sup_{t\in\lbrack0,T]}\Vert p(t)\Vert_{H}^{\frac
		{\beta}{2}}(\int_{t}^{T}\Vert p(t)\Vert_{H}^{\beta-2}\Vert q(t)\Vert_{H}%
	^{2}dt)^{\frac{\beta}{2}}]\\
	&  \leq C\mathbb{E[}\Vert\xi\Vert_{H}^{\beta}+\int_{0}^{T}\Vert p(t)\Vert
	_{H}^{\beta}dt]+C\mathbb{E[}(\int_{0}^{T}\Vert f(t,0,0)\Vert_{H}dt)^{\beta
	}]+\frac{1}{2}\mathbb{E[}\sup_{t\in\lbrack0,T]}\Vert p(t)\Vert_{H}^{\beta}]\\
	&  \ \ \ \,+C\mathbb{E[}(\int_{0}^{T}\Vert p(t)\Vert_{H}^{\beta-2}\Vert
	q(t)\Vert_{H}^{2}dt)^{\beta}].
\end{align*}
Thus,
\begin{equation}%
	\begin{split}
		&  \mathbb{E[}\sup_{t\in\lbrack0,T]}\Vert p(t)\Vert_{H}^{\beta}]\\
		&  \leq C\{\mathbb{E[}\Vert\xi\Vert_{H}^{\beta}+\int_{0}^{T}\Vert
		p(t)\Vert_{H}^{\beta}dt]+\mathbb{E[}(\int_{0}^{T}\Vert f(t,0,0)\Vert
		_{H}dt)^{\beta}]+\mathbb{E[}(\int_{0}^{T}\Vert p(t)\Vert_{H}^{\beta-2}\Vert
		q(t)\Vert_{H}^{2}dt)^{\beta}]\}.
	\end{split}
	\label{eq-12}%
\end{equation}
Then plugging (\ref{eq-11}) into (\ref{eq-12}), we get
\[
\mathbb{E[}\sup_{t\in\lbrack0,T]}\Vert p(t)\Vert_{H}^{\beta}]\leq
C\mathbb{E[}\Vert\xi\Vert_{H}^{\beta}]+C\delta\mathbb{E[}\sup_{t\in
	\lbrack0,T]}\Vert p(t)\Vert_{H}^{\beta}]+C_{\delta}\mathbb{E[}(\int_{0}%
^{T}\Vert f(t,0,0)\Vert_{H}dt)^{\beta}].
\]
Choosing $\delta$ small enough, we get
\begin{equation}
	\mathbb{E[}\sup_{t\in\lbrack0,T]}\Vert p(t)\Vert_{H}^{\beta}]\leq
	C\{\mathbb{E[}\Vert\xi\Vert_{H}^{\beta}]+\mathbb{E[}(\int_{0}^{T}\Vert
	f(t,0,0)\Vert_{H}dt)^{\beta}]\}.\label{eq-001}%
\end{equation}

Next, taking $\beta=2$ in (\ref{estimate-0}), we have
\begin{align*}
	&  \int_{t}^{T}\Vert q(s)\Vert_{H}^{2}ds+\int_{t}^{T}\Vert p(s)\Vert_{V}%
	^{2}ds\leq C\Vert\xi\Vert_{H}^{2}+C\int_{t}^{T}\Vert p(s)\Vert_{H}^{2}ds\\
	& \ \ \ \ \ \ \   +C\sup_{s\in\lbrack0,T]}\Vert p(s)\Vert_{H}\int_{t}^{T}\Vert f(s,0,0)\Vert
	_{H}dt-C_{2}\int_{t}^{T}\left\langle p(s),q(s)\right\rangle _{H}dw(s).
\end{align*}
Then
\begin{align*}
	&  \mathbb{E[}(\int_{0}^{T}\Vert q(t)\Vert_{H}^{2}dt)^{\frac{\beta}{2}%
	}]+\mathbb{E}[(\int_{0}^{T}\Vert p(t)\Vert_{V}^{2}dt)^{\frac{\beta}{2}}]\\
	&  \leq C\{\mathbb{E}[\Vert\xi\Vert_{H}^{\beta}]+\mathbb{E}[\int_{0}^{T}\Vert
	p(t)\Vert_{H}^{\beta}dt]+\mathbb{E}[(\int_{0}^{T}|\left\langle
	p(s),q(s)\right\rangle _{H}|^{2}dt)^{\frac{\beta}{4}}]\\
	&  \ \ \ \,+\mathbb{E}[\sup_{t\in\lbrack0,T]}\Vert p(t)\Vert_{H}^{\frac{\beta
		}{2}}(\int_{0}^{T}\Vert f(t,0,0)\Vert_{H}dt)^{\frac{\beta}{2}}]\}\\
	&  \leq C\{\mathbb{E}[\Vert\xi\Vert_{H}^{\beta}]+\mathbb{E}[\int_{0}^{T}\Vert
	p(t)\Vert_{H}^{\beta}dt]+\mathbb{E}[\sup_{t\in\lbrack0,T]}\Vert p(t)\Vert
	_{H}^{\frac{\beta}{2}}(\int_{0}^{T}\Vert q(t)\Vert_{H}^{2}dt)^{\frac{\beta}%
		{4}}]\\
	&  +\mathbb{E}[\sup_{t\in\lbrack0,T]}\Vert p(t)\Vert_{H}^{\frac{\beta}{2}%
	}(\int_{0}^{T}\Vert f(t,0,0)\Vert_{H}dt)^{\frac{\beta}{2}}]\}\\
	&  \leq C\{\mathbb{E}[\Vert\xi\Vert_{H}^{\beta}]+\mathbb{E}[\int_{0}^{T}\Vert
	p(t)\Vert_{H}^{\beta}dt]+\mathbb{E}[(\int_{0}^{T}\Vert f(t,0,0)\Vert
	_{H}dt)^{\beta}]+\mathbb{E}[\sup_{t\in\lbrack0,T]}\Vert p(t)\Vert_{H}^{\beta
	}]\}\\
	&  \ \ \ \,+\frac{1}{2}\mathbb{E}[(\int_{0}^{T}\Vert q(t)\Vert_{H}%
	^{2}dt)^{\frac{\beta}{2}}]
\end{align*}
From this and (\ref{eq-001}), we get
\[
\mathbb{E[}(\int_{0}^{T}\Vert q(t)\Vert_{H}^{2}dt)^{\frac{\beta}{2}%
}]+\mathbb{E}[(\int_{0}^{T}\Vert p(t)\Vert_{V}^{2}dt)^{\frac{\beta}{2}}]\leq
C\mathbb{E[}\Vert\xi\Vert_{H}^{\beta}+(\int_{0}^{T}\Vert f(t,0,0)\Vert
_{H}dt)^{\beta}].
\]
This completes the proof.
\end{proof}

On the other hand,  the estimate $\sup_{t\in\lbrack0,T]}\mathbb{E}[\Vert P(t)\Vert_{\mathfrak{L}%
	(H)}^{\beta}]<\infty$ for any $\beta\geq2$ follows trivially from the estimate (2.14) of
the BSIE. Indeed, from (2.14) we have
\[
\Vert P(t)\Vert_{\mathfrak{L}(H)}^{\beta}\leq C\mathbb{E}[\Vert\xi
\Vert_{\mathfrak{L}(H)}^{\beta}+(\int_{t}^{T}\Vert f(s,0)\Vert_{\mathfrak{L}%
	(H)}^{2}ds)^{\frac{\beta}{2}}|\mathcal{F}_{t}],\quad P\text{-a.s}.
\]
Taking expectation on both sides, we obtain
\[
\mathbb{E}[\Vert P(t)\Vert_{\mathfrak{L}(H)}^{\beta}]\leq C\mathbb{E}[\Vert
\xi\Vert_{\mathfrak{L}(H)}^{\beta}+(\int_{0}^{T}\Vert f(t,0)\Vert
_{\mathfrak{L}(H)}^{2}dt)^{\frac{\beta}{2}}],\quad\text{for each }t\in\lbrack0,T].
\]

\noindent\textbf{Acknowledgments.} The first author would like to thank
Professor Kai Du, Shanghai Center for Mathematical Sciences at Fudan
University, for helpful discussions. The first author also thanks Doctor Ruoyang Liu, Professor Qingxin Meng, Professore Falei Wang and
Professor Tianxiao Wang for valuable comments.

\end{document}